\documentclass{article}

\usepackage{amsmath,amsthm} 
\usepackage{amssymb,mathrsfs} 
\usepackage{hyperref}
\usepackage{graphicx} 
\usepackage{color,subfigure} 
\usepackage{enumerate}  
\usepackage{fullpage}

\newtheorem{theorem}{Theorem}
\newtheorem{lemma}{Lemma}
\newtheorem{assumption}{Assumption}

\newtheorem{remark}[theorem]{Remark}

\DeclareMathAlphabet{\mathpzc}{OT1}{pzc}{m}{it}
\newcommand{\dps}{\displaystyle } 
\newcommand{\rme}{\mathrm{e}}

\newcommand{\PP}{\mathbb{P}}

\newcommand{\RR}{\mathbb{R}}
\newcommand{\R}{\mathbb{R}}
\newcommand{\Id}{\mathrm{Id}}
\newcommand{\EE}{\mathbb{E}}

\newcommand{\cR}{\mathcal{R}}
\newcommand{\cM}{\mathcal{M}}
\newcommand{\cL}{\mathcal{L}}

\newcommand{\dt}{{\Delta t}}
\newcommand{\ind}{\mathbf{1}}

\renewcommand{\leq}{\leqslant}
\renewcommand{\geq}{\geqslant}

\newcommand{\Tr}{\operatorname{Tr}}
\renewcommand{\mod}{{\mathrm{mod}}}
\newcommand{\EM}{{\mathrm{EM}}}
\newcommand{\Barker}{{\mathrm{Barker}}}

\begin{document}

  
\title{Improving dynamical properties of metropolized discretizations of overdamped Langevin dynamics}
\author{M. Fathi$^{1}$ and G. Stoltz$^{2}$ \\
{\small $^{1}$ LPMA, 4 place Jussieu, 75005 Paris, France} \\
{\small $^{2}$ Universit\'e Paris-Est, CERMICS (ENPC), INRIA, F-77455 Marne-la-Vall\'ee, France} \\
}
 
\maketitle

\abstract{
The discretization of overdamped Langevin dynamics, through schemes such as the Euler-Maruyama method, can be corrected by some acceptance/rejection rule, based on a Metropolis-Hastings criterion for instance. In this case, the invariant measure sampled by the Markov chain is exactly the Boltzmann-Gibbs measure. However, rejections perturb the dynamical consistency of the resulting numerical method with the reference dynamics. We present in this work some modifications of the standard correction of discretizations of overdamped Langevin dynamics on compact spaces by a Metropolis-Hastings procedure, which allow us to either improve the strong order of the numerical method, or to decrease the bias in the estimation of transport coefficients characterizing the effective dynamical behavior of the dynamics. For the latter approach, we rely on modified numerical schemes together with a Barker rule for the acceptance/rejection criterion. 
}


\section{Introduction}

Molecular simulation is nowadays a very common tool to quantitatively predict macroscopic properties of matter starting from a microscopic description. These macroscopic properties can be either static properties (such as the average pressure or energy in a system at fixed temperature and density), or transport properties (such as thermal conductivity or shear viscosity). Molecular simulation can be seen as the computational version of statistical physics, and is therefore often used by practitioners of the field as a black box to extract the desired macroscopic properties from some model of interparticle interactions. Most of the work in the physics and chemistry fields therefore focuses on improving the microscopic description, most notably developing force fields of increasing complexity and accuracy. In comparison, less attention has been paid to the estimation of errors in the quantities actually computed by numerical simulation. Usually, due to the very high dimensionality of the systems under consideration, macroscopic properties are computed as ergodic averages over a very long trajectory of the system, evolved under some appropriate dynamics. There are two main types of errors in this approach: (i) statistical errors arising from incomplete sampling, and (ii) systematic errors (bias) arising from the fact that continuous dynamics are numerically integrated using a finite time-step~$\dt > 0$. 

The aim of this work is to reduce the bias arising from the use of finite time steps in the computation of dynamical quantities, for a certain type of dynamics called Brownian dynamics, or overdamped Langevin dynamics in the chemistry literature. Such dynamics are used to simulate ionic solutions (see~\cite{Jardat}). The methods we develop here can also be used to integrate the fluctuation/dissipation for numerical schemes based on a splitting of underdamped Langevin dynamics, in the case when the kinetic energy is not quadratic (see~\cite{RST16,ST16}). 

We denote by $\cM$ the configurational space, which is in this work a compact domain with periodic boundary conditions, such as $\mathbb{T}^{d}$ (with $\mathbb{T} = \mathbb{R} / \mathbb{Z}$ the one-dimensional torus). Unbounded configuration spaces $\R^d$ can also be considered under some assumptions (see Remark~\ref{rmk:generalization_unbounded}). The overdamped Langevin dynamics is a stochastic differential equation on the configuration~$q \in \cM$ of the system:
\begin{equation}
\label{eq:dynamics}
dq_t = -\beta \nabla V(q_t) \, dt + \sqrt{2} \, dW_t,
\end{equation}
where $\beta = 1/(k_{\rm B}T) > 0$ is the inverse temperature ($k_{\rm B}$ being Boltzmann's constant and $T$ being the temperature) and $W_t$ is a standard $d$-dimensional Brownian motion. The function $V:\mathcal{M}\to \mathbb{R}$ is the potential energy, assumed to be smooth for the mathematical analysis. The generator associated with~\eqref{eq:dynamics} acts on smooth functions $\varphi$ as 
\begin{equation}
  \label{eq:generator}
  \cL\varphi = -\beta \nabla V^T \nabla \varphi + \Delta \varphi.
\end{equation}
The probability measure
\begin{equation}
\label{eq:Gibbs}
\mu(dq) = Z^{-1} \,\rme^{-\beta V(q)} \, dq, \qquad Z = \int_\mathcal{M} \rme^{-\beta V}
\end{equation}
is invariant by the dynamics~\eqref{eq:dynamics} since a simple computation shows that, for appropriate functions~$\varphi$ (\textit{e.g.} smooth and such that $\cL\varphi$ is integrable with respect to~$\mu$),
\[
\int_\cM \cL \varphi \, d\mu = 0.
\]
Simple discretizations of~\eqref{eq:dynamics} may fail to be ergodic when the dynamics is considered on unbounded spaces and the potential energy function is not globally Lipschitz~\cite{MSH02}. In simulations of ionic solutions, potential energy functions with Coulomb-type singularities are used and it has been observed that the energy blows up (see for instance~\cite[Section~3.2]{FHS14}).

In order to stabilize numerical discretizations or simply to remove the bias in the invariant measure arising from the use of finite timesteps, it was suggested to consider the move obtained by a numerical scheme approximating~\eqref{eq:dynamics} as a proposal to be accepted or rejected according to a Metropolis-Hastings ratio~\cite{MRRTT53,Hastings70}. The corresponding method is known as ``Smart MC'' in the chemistry literature~\cite{RDF78}, and was rediscovered later on in the computational statistics literature~\cite{RT96} where it is known as ``Metropolis Adjusted Langevin Algorithm'' (MALA). The interest of the acceptance/rejection step is that it ensures that the Markov chain is reversible and samples the Gibbs measure~$\mu$. This prevents in particular blow-ups, which are observed for dynamics in unbounded position spaces when the forces are non-globally Lipschitz, or in periodic position spaces with singular potentials of Coulomb type~\cite{Jardat}. On the other hand, the use of an acceptance/rejection procedure limits the possible numerical schemes one can use as a proposal. Indeed, when generating a proposed move, the probability of coming back to the original configuration from the proposed move has to be computed, by evaluating the probability to observe the Gaussian increment necessary to perform the backward move. Therefore, it is unclear that proposals which are nonlinear in the Gaussian increment (such as the ones produced, say, by the Milstein's scheme, see~\cite{MT04} for instance) can be used, except in specific cases. 

The previous works on the numerical analysis of dynamical properties of MALA established (i)~strong error estimates over finite times~\cite{BV09}, and, as a consequence, errors on finite time correlations~\cite{BE12}; (ii)~exponential convergence rates towards the invariant measure, uniformly in the timestep~\cite{BH13} (which holds up to a small error term in $\dt$ for systems in infinite volume); (iii)~error estimates on the effective diffusion of the dynamics~\cite{FHS14}. 

\medskip

The aim of this work is to present new proposal functions, and also to advocate the use of acceptance/rejection rules different than the Metropolis one, in order to reduce the systematic bias in the estimation of dynamical quantities. More precisely, we propose in Section~\ref{sec:strong_order} a modified proposal which allows to obtain a strong error of order~1 rather than~3/4 for MALA (see Theorem~\ref{thm:improved_strong}). We also show in Section~\ref{sec:transport} that the error on the computed self-diffusion can be reduced from~$\dt$ for MALA to~$\dt^2$ provided a modified proposal is considered together with a Barker rule~\cite{Barker65} (see Theorems~\ref{thm:improved_GK} and~\ref{thm:fluctuation}). As discussed in Section~\ref{sec:multiplicative_diff}, the use of a Barker rule also reduces the bias on the self-diffusion for dynamics with multiplicative noise (\textit{i.e.} a non-trivial diffusion matrix in front the Brownian motion in~\eqref{eq:dynamics}). On the other hand, resorting to a Barker rule increases the statistical error in the simulation, roughly by a factor~2 since the rejection rate is 1/2 in the limit $\dt \to 0$ (as made precise in Remark~\ref{rmk:variance}). Such an increase in the variance was already shown in its full generality by Peskun in~\cite{Peskun73}. However, the reduction in the bias more than compensates the increase in the statistical error if one is interested in simulations with a given error tolerance (as discussed more precisely in Remark~\ref{rmk:compensation}). We provide numerical illustrations of our predictions at the end of Sections~\ref{sec:strong_order} and~\ref{sec:transport}. The proofs of all our results are gathered in Section~\ref{sec:proof}. As mentioned earlier, for simplicity we shall work under the following assumption: 

\begin{assumption}
  The state spaces~$\cM$ is compact and has periodic boundary conditions, and the potential~$V$ is smooth.
\end{assumption}

Compact spaces with periodic boundary conditions anyway is the relevant setting for self-diffusion (which is the focus of Section~\ref{sec:transport}) since there is no effective diffusion for dynamics in unbounded spaces with confining potentials. For the strong error estimates provided in Section~\ref{sec:strong_order}, we could have considered unbounded position spaces. We chose to present our results for compact position spaces since this allows to simplify some technical elements of the proofs. See however Remark~\ref{rmk:generalization_unbounded} for the extension to the noncompact setting.

\section{Improving the strong order}
\label{sec:strong_order}

In this section, we show how a suitable modification of the proposal leads to a Metropolized scheme with a better strong accuracy than the standard MALA scheme. As was realized in~\cite{BV09}, the local error over one timestep arises at dominant order from the rejections and not from the integration errors of the accepted move. The strategy to improve the strong error is therefore to add terms of lower order (in the timestep) to the proposal, which will not have a significant impact when the proposal is accepted, but which will significantly reduce the average rejection rate. As made precise in Lemma~\ref{lem:improved_rejection_rate}, the rejection rate scales as $\Delta t^{5/2}$ for the modified dynamics, instead of $\Delta t^{3/2}$ for MALA. Let us also mention that, instead of considering the modified dynamics we propose (see~\eqref{eq:modified_proposal_strong_order} below), we could in fact consider dynamics with an arbitrarily low rejection rate, see Remark~\ref{rmk:HMC}. However, this would not improve the strong error estimates we obtain any further.

We start by recalling the known error estimates for MALA in Section~\ref{sec:MALA}, before presenting the error bounds obtained for our modified dynamics in Section~\ref{eq:strong_modified}. Our theoretical predictions are illustrated by numerical simulations in Section~\ref{sec:num_strong}. 

\subsection{Strong error estimates for the standard MALA algorithm}
\label{sec:MALA}

Let us first recall the definition of the MALA scheme. It is a Metropolis-Hastings algorithm whose proposal function is obtained by a Euler-Maruyama discretization of the dynamics~\eqref{eq:dynamics}. Given a timestep $\dt > 0$ and a previous configuration~$q^n$, the proposed move is 
\[
\widetilde{q}^{n+1} = q^n - \beta \dt\, \nabla V(q^n) + \sqrt{2\dt} \, G^n,
\]
where (here and elsewhere) $(G^n)_{n \geq 0}$ is a sequence of independent and identically distributed (i.i.d.) $d$-dimensional standard Gaussian random variables. For further purposes, it will be convenient to encode proposals using a function $\Phi_\dt(q,G)$ depending on the previous position~$q$ and the Gaussian increment~$G$ used to generate the proposed move. For the above Euler-Maruyama proposal, $\widetilde{q}^{n+1} = \Phi^\EM_\dt(q^n,G^n)$ with
\begin{equation}
  \label{eq:EM}
  \Phi^\EM_\dt(q,G) = q - \beta \dt\, \nabla V(q) + \sqrt{2\dt} \, G.
\end{equation}
We next accept or reject the proposed move according to the Metropolis-Hastings ratio $A_\dt\left(q^n,\widetilde{q}^{n+1}\right)$:
\begin{equation} 
  \label{def_accept_rate}
  A_\dt\left(q^n,\widetilde{q}^{n+1}\right) = \min\left( \frac{\rme^{-\beta V(\widetilde{q}^{n+1})}T_\dt(\widetilde{q}^{n+1},q^n)}{\rme^{-\beta V(q^n)}T_\dt(q^n,\widetilde{q}^{n+1})},1\right),
\end{equation}
where
\[
T_\dt(q,q') = \left(\frac{1}{4\pi\dt}\right)^{d/2} \exp \left (
-\frac{|q'-q+ \beta \dt \, \nabla V(q)|^2}{4\dt} \right )
\]
is the probability transition of the Markov chain encoded by~\eqref{eq:EM}. 
When the proposition is accepted, which is decided with probability $A_{\dt}\left(q^n,\widetilde{q}^{n+1}\right)$, we project~$\widetilde{q}^{n+1}$ into the periodic simulation cell~$\cM$. If the proposal is rejected, the previous configuration is counted twice: $q^{n+1} = q^n$ (It is very important to count rejected configuration as many times as needed to ensure that the Boltzmann-Gibbs measure~$\mu$ is invariant). In conclusion,
\begin{equation}
\label{eq:mEM}
q^{n+1} = \Psi_\dt(q^n,G^n,U^n) = q^n + \ind_{U^n \leq A_\dt\left(q^n,\Phi^\EM_\dt\left(q^n,G^n\right)\right)} \Big(\Phi^\EM_\dt\left(q^n,G^n\right)-q^n\Big),
\end{equation}
where $(U^n)_{n \geq 0}$ is a sequence of i.i.d. random variables uniformly distributed in~$[0,1]$, and $\ind_{U^n \leq \alpha}$ is an indicator function whose value is~1 when $U^n \leq \alpha$ and~0 otherwise. 

It is easy to see that the Markov chain is irreducible with respect to the Lebesgue measure and that $\mu$ is an invariant probability measure. It is therefore ergodic, and in fact reversible with respect to~$\mu$ (see for instance the references in~\cite[Section~2.2]{LRS10}). In particular, this guarantees that the scheme does not blow up.

\medskip

The following result on the strong convergence of MALA on finite time intervals has been obtained in~\cite[Theorem~2.1]{BV09}. We state the results for dynamics in compact spaces although it has been obtained for dynamics in unbounded spaces, under some additional assumptions on the potential energy function. This allows us to write strong error estimates which are uniform with respect to the initial condition~$q_0$, rather than average errors for initial conditions distributed according to the canonical measure~$\mu$ (a restriction arising from the lack of geometric ergodicity for MALA in the case of unbounded spaces, see~\cite{RT96}).

\begin{theorem}[Strong convergence of MALA~\cite{BV09}]
  \label{thm:old_error_strong}
  Consider the Markov chain 
  \[
  q^{n+1} = \Psi_\dt(q^n,W_{(n+1)\dt}-W_{n\dt},U^n)
  \]
  started from $q^0 = q_0$, and denote by $q_t^{\Delta t}$ the piecewise constant function defined as $q^{\Delta t}_t = q^n$ for $t \in [n\Delta t, (n+1)\Delta t)$. Then, there exists $\Delta t^* > 0$ and, for any $T > 0$, a constant $C(T) > 0$ such that, for any $0 < \Delta t \leq \Delta t^*$ and all $t \in [0, T]$, and for all $q_0 \in \cM$,
    \[
    \left(\mathbb{E}_{q_0}\left[\left|q^{\Delta t}_t - q_t\right|^2\right]\right)^{1/2} \leq C(T) \, \dt^{3/4},
    \]
    where the expectation is over all realizations of the Brownian motion $(W_t)_{0 \leq t \leq T}$. 
\end{theorem}

Numerical simulations confirm that the exponent $3/4$ is sharp (see~\cite[Section~3.1]{BV09} and Section~\ref{sec:num_strong} below). Let us emphasize that, in order to have correct strong error estimates, the Gaussian increments in the numerical scheme have to be consistent with the realization of the Brownian motion used to construct the solution of the continuous stochastic dynamics. 

\subsection{Strong error estimates for a modified dynamics}
\label{eq:strong_modified}

The un-Metropolized Euler scheme has strong order~1 for dynamics with additive noise such as~\eqref{eq:dynamics}. In order to improve the convergence rate $\dt^{3/4}$ given by Theorem~\ref{thm:old_error_strong}, we modify~\eqref{eq:EM} with terms of order~$\dt^{3/2}$ and~$\dt^2$:
\begin{equation}
  \label{eq:modified_proposal_strong_order}
  \widetilde{q}^{n+1} = q^n + \dt \Big(-\beta \nabla V(q^n) + \dt \, F(q^n) \Big) + \sqrt{2\dt}\Big(\Id + \dt \, \sigma(q^n) \Big)^{-1/2} G^n.
\end{equation}
This can be encoded by the proposal function
\[
\Phi^\mod_\dt(q,G) = \Phi^\EM_\dt(q,G) + \dt^2 F(q) + \sqrt{2\dt} \left[\left(\Id+\dt \, \sigma(q)\right)^{-1/2} - \Id \right] G.
\]
The above definitions are formal since the matrix $\Id + \dt \, \sigma(q^n)$ should be symmetric definite positive in order for the inverse square root to make sense. As made clear below, this is indeed the case for sufficiently small timesteps since the matrix valued function~$\sigma$ is smooth hence bounded on~$\cM$ (see Remark~\ref{rmk:generalization_unbounded} for unbounded spaces). 

\begin{remark}[Alternative expressions of the diffusion matrix]
It is possible to consider other diffusion matrices which agree with $\left(\Id+\dt \, \sigma(q)\right)^{-1/2}$ at dominant order in~$\dt$, such as $\Id-\dt \, \sigma(q)/2$ which need not be inverted but may be negative for large values of~$\dt$. In~\cite{Durmus}, the authors suggest to consider matrices such as $\exp(-\dt \, \sigma(q)/2)$. The latter matrix is indeed always non-negative but may be cumbersome to compute in practice. On unbounded position spaces however, $\sigma$ may be unbounded from below, so that, no matter how small $\dt$ is, there may be configurations at which $\Id + \dt \, \sigma(q)$ is non-invertible. In this case, $\exp(-\dt \, \sigma(q)/2)$ should be considered instead. Another interest of such choices is that geometric ergodicity results can be obtained for dynamics on unbounded spaces, see~\cite{Durmus}. 
\end{remark}

The Metropolis-Hastings ratio associated with the proposal~\eqref{eq:modified_proposal_strong_order} reads
\begin{equation}
\label{eq:Metropolis_ratio_strong}
A^\mod_\dt\left(q^n,\widetilde{q}^{n+1}\right) = \min\left( \frac{\rme^{-\beta V(\widetilde{q}^{n+1})}T^\mod_\dt(\widetilde{q}^{n+1},q^n)}{\rme^{-\beta V(q^n)}T^\mod_\dt(q^n,\widetilde{q}^{n+1})},1\right),
\end{equation}
with a transition rate taking into account the spatial dependence of the diffusion:
\begin{align}
T^\mod_\dt&(q,q') := \left(\frac{1}{4\pi\dt}\right)^{d/2} \det\Big(\Id + \dt \, \sigma(q) \Big)^{1/2} \notag \\
& \times \exp\left(-\frac{ \Big(q'-q+ \dt \left[\beta \nabla V(q) - \dt F(q)\right]\Big)^T \Big(\Id + \dt\,\sigma(q)\Big) \Big(q'-q+ \dt \left[\beta \nabla V(q) - \dt F(q)\right]\Big) }{4\dt} \right ).
\end{align}
We denote by $\Psi_\dt^\mod(q,G,U)$ the random variable obtained by a single step of the modified MALA scheme~\eqref{eq:modified_proposal_strong_order} starting from $q$:
\[
q^{n+1} = \Psi^\mod_\dt(q^n,G^n,U^n),
\]
with
\begin{equation}
\label{eq:def_Psimod}
\Psi^\mod_\dt(q,G,U) = q + \ind_{U \leq A^\mod_\dt\left(q,\Phi_\dt^\mod(q,G)\right)}\left(\Phi_\dt^\mod(q,G) - q \right).
\end{equation}
We can then state the following strong error estimate (see Section~\ref{sec:proof:thm:improved_strong} for the proof).

\begin{theorem}[Improved strong error estimates]
\label{thm:improved_strong}
For a smooth potential~$V$ on the compact space~$\cM$, choose $\sigma$ and $F$ as follows:
\begin{equation}\label{choices_mod_mala}
\sigma = \frac{\beta}{3}\nabla^2 V, 
\qquad 
F = \frac{\beta}{6}\Big( \nabla (\Delta V) - \beta \nabla^2 V \nabla V \Big). 
\end{equation}
Consider, for $\dt < 1/\| \sigma \|_{L^\infty}$, the Markov chain associated with the modified proposal~\eqref{eq:modified_proposal_strong_order} corrected by a Metropolis rule: $q^{n+1} = \Psi^\mod_\dt(q^n,W_{(n+1)\dt}-W_{n\dt},U^n)$, and denote by $q_t^{\Delta t, \mod}$ the piecewise constant function defined as $q^{\Delta t, \mod}_t = q^n$ for $t \in [n\Delta t, (n+1)\Delta t)$. Then, there exists $\dt^* > 0$ and, for any $T > 0$, a constant $C(T) > 0$ such that, for any $0 < \dt \leq \dt^*$ and $t \in [0,T]$, and for all $q_0 \in \cM$, 
\[
\left(\mathbb{E}_{q_0}\left[\left|q^{\Delta t, \mod}_t - q_t\right|^2\right]\right)^{1/2} \leq C(T)\dt.
\]
\end{theorem}

Let us mention that the modified scheme~\eqref{eq:modified_proposal_strong_order} with the choice~\eqref{choices_mod_mala} is more difficult to implement in practice than the standard Euler-Maruyama scheme since it requires the computation of derivatives of order up to~3 of the potential, which is often cumbersome in molecular dynamics. The crucial estimate to prove Theorem~\ref{thm:improved_strong} is the following bound on the rejection rate, which makes precise the fact that the rejection rate scales as $\dt^{5/2}$ rather than~$\dt^{3/2}$ as for MALA (see Section~\ref{sec:proof:improved_rejection_rate} for the proof). As the proof shows, there is some algebraic miracle since the corrections terms, chosen to eliminate the dominant contribution of order~$\dt^{3/2}$ in the rejection rate, in fact also eliminate the next order contribution of order~$\dt^2$.

\begin{lemma} 
\label{lem:improved_rejection_rate}
For any $\ell \geq 1$, there exists a constant $C_{\ell} > 0$ such that
\[
\forall q \in \R^d, 
\qquad 
\mathbb{E}_G\left(\left[1 - A^\mod_\dt\left(q, \Phi_\dt^\mod(q,G)\right)\right]^{2\ell}\right) \leq C_{\ell}\dt^{5\ell}.
\]
\end{lemma}

Let us end this section with the following considerations on the extension of the improved strong error estimates to unbounded spaces.

\begin{remark}[Generalization to unbounded spaces]
\label{rmk:generalization_unbounded}
It is possible to extend the results of Theorem~\ref{thm:improved_strong} to unbounded spaces under appropriate assumptions on the potential, such as~\cite[Assumption~4.1]{BV09} and~\cite[Assumptions~2.1]{BH13}. These assumptions ensure that the potential energy function grows sufficiently fast at infinity, with derivatives bounded by~$V$ (up to order~5 in the present case), and a lower bound on the Hessian. In this setting, error estimates can be stated for initial conditions $q_0 \sim \nu$, where $\nu$ is such that $\EE_\nu(|V(q^k)|^p) \leq R < +\infty$ for some integer~$p$ sufficiently large. A convenient choice is $\nu = \mu$, in which case $q^k \sim \mu$, and the finiteness of the moments of~$V$ is guaranteed under mild conditions on~$V$. More generally, error estimates for generic initial conditions can be stated when the dynamics is geometrically ergodic. See also Remark~\ref{rmk:unbounded_spaces} for technical precisions on the extension of Lemma~\ref{lem:improved_rejection_rate}.
\end{remark}

\subsection{Numerical results}
\label{sec:num_strong}

We illustrate the error bounds predicted by Theorems~\ref{thm:old_error_strong} and~\ref{thm:improved_strong}, and also check the scaling of the rejection rate obtained in Lemma~\ref{lem:improved_rejection_rate}. We consider a simple one-dimensional system with the potential energy function $V(q) = q^4$ already used as a test case in~\cite{BV09}. This example is particularly relevant since it can be shown that in absence of Metropolization the associated Euler-Maruyama scheme is transient~\cite{MSH02}.

We compute a reference trajectory over a given time interval~$[0,T]$ with a very small timestep $\dt_{\rm ref}$. We next compute trajectories for larger timesteps $\dt_k = k\dt_{\rm ref}$, using Brownian increments consistent with the ones used to generate the trajectory with the smallest timestep. More precisely, denoting by $(G_{\rm ref}^n)_{n=0,\dots,T/\dt_{\rm ref}-1}$ the Gaussian random variables used for the reference trajectory, the trajectories $(q_k^m)_{m \geq 0}$ with coarser timesteps $\dt_k = k\dt_{\rm ref}$ are computed with the following iterative rule for $0 \leq m \leq T/\dt_k$ (assuming that $T/\dt_k$ is an integer):
\[
q_k^{m+1} = \Psi_{\dt_k}(q^m_k,G^m_k,U_k^m),
\]
with the Gaussian random variables
\[
G_k^m = \frac{1}{\sqrt{k}}\sum_{n=mk}^{(m+1)k-1} G_{\rm ref}^n,
\]
but where the uniform random variables $(U_k^m)_{k,m}$ are i.i.d. and independent of the variables $(U^n_{\rm ref})_{n \geq 0}$ used to generated the reference trajectory. We denote by $d_k$ the maximal error between the reference trajectory and the trajectory generated with the timestep $\dt_k$, considered at the same times $m\dt_k = mk\dt_{\rm ref}$:
\[
d_k = \max_{m=1,\dots,T/\dt_k} \left| q_k^m - q_{\rm ref}^{mk} \right|.
\]
We next perform $I$ independent realizations of this procedure, henceforth generating trajectories $(q_k^{i,m})_{m \geq 0}$ and $(q_{\rm ref}^{i,m})_{m \geq 0}$. Denoting by $d_k^i$ the so-obtained errors, the strong error for the timestep $\dt_k$ is estimated by the empirical average
\[
\widehat{E}_{k,I} = \sqrt{\frac1I \sum_{i=1}^I \left(d_k^i\right)^2}
\]
In fact, confidence intervals on this error can be obtained thanks to a Central Limit theorem, which shows that the following approximate equality holds in law:
\[
\widehat{E}_{k,I} \simeq \sqrt{\mathbb{E}(d_k^2)}\left(1 + \frac{\sigma_k}{2\mathbb{E}(d_k^2)\sqrt{I}} \, \mathcal{G}\right), \qquad \sigma_k^2 = \mathbb{E}(d_k^4) - \left( \mathbb{E}(d_k^2) \right)^2, 
\]
where $\mathcal{G}$ is a standard Gaussian random variable. Estimates of the average rejection rates are obtained with the empirical average of $1-A_\dt(q^n,\widetilde{q}^{n+1})$ computed along the generated trajectories.

Figure~\ref{fig:strong_error} presents the estimates $\widehat{E}_{k,I}$ of the strong error and the scaling of the rejection rate as a function of the time step $\dt_k$.  The reference timestep is $\dt_{\rm ref} = 10^{-6}$, and we chose $k = 100 \times 2^l$ where $l = 0,\dots,L$ with $L=12$, as well as an integration time $T = 0.8192$ (which corresponds to $2\dt_{L}$). We fixed $\beta =1$ and averaged over $I = 10^6$ realizations for the standard proposal~\eqref{eq:EM} and $I=5 \times 10^5$ realizations for the modified proposal~\eqref{eq:modified_proposal_strong_order}. The first initial condition is $q_{\rm ref}^{0,0} = 0$, while the subsequent initial conditions are the end point of the previous trajectory: $q_{\rm ref}^{i+1,0} = q_{\rm ref}^{i,T/\dt_{\rm ref}}$. The maximal value of $\sigma_d/(2\mathbb{E}(d^2))$ was estimated to be less than~32 for all reported simulations, so that the maximal relative statistical error is lower than~0.05 in all cases, and often much smaller.

As predicted by Theorems~\ref{thm:old_error_strong} and~\ref{thm:improved_strong}, we find that the strong error decreases as $\dt^{3/4}$ for MALA (with a rejection rate scaling as $\dt^{3/2}$, as predicted by~\cite[Lemma~4.7]{BV09}), while it scales as $\dt$ for the modified dynamics~\eqref{eq:def_Psimod}; see Figure~\ref{fig:strong_error} (Left). A key element to the improvement in the strong error is the fastest decrease of the rejection rate, indeed confirmed to be of order~$\dt^{5/2}$; see Figure~\ref{fig:strong_error} (Right).

\begin{figure}
\begin{center}
\includegraphics[width=8.2cm]{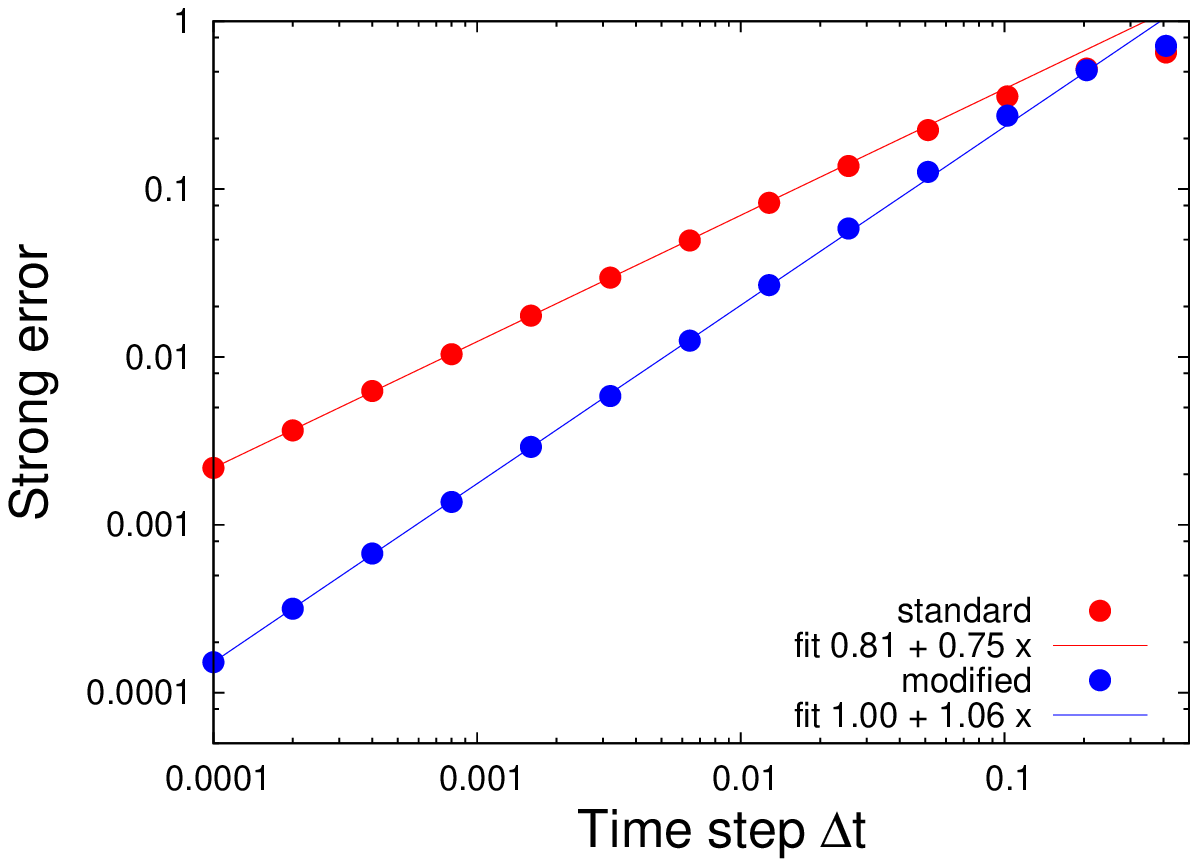}
\includegraphics[width=8.2cm]{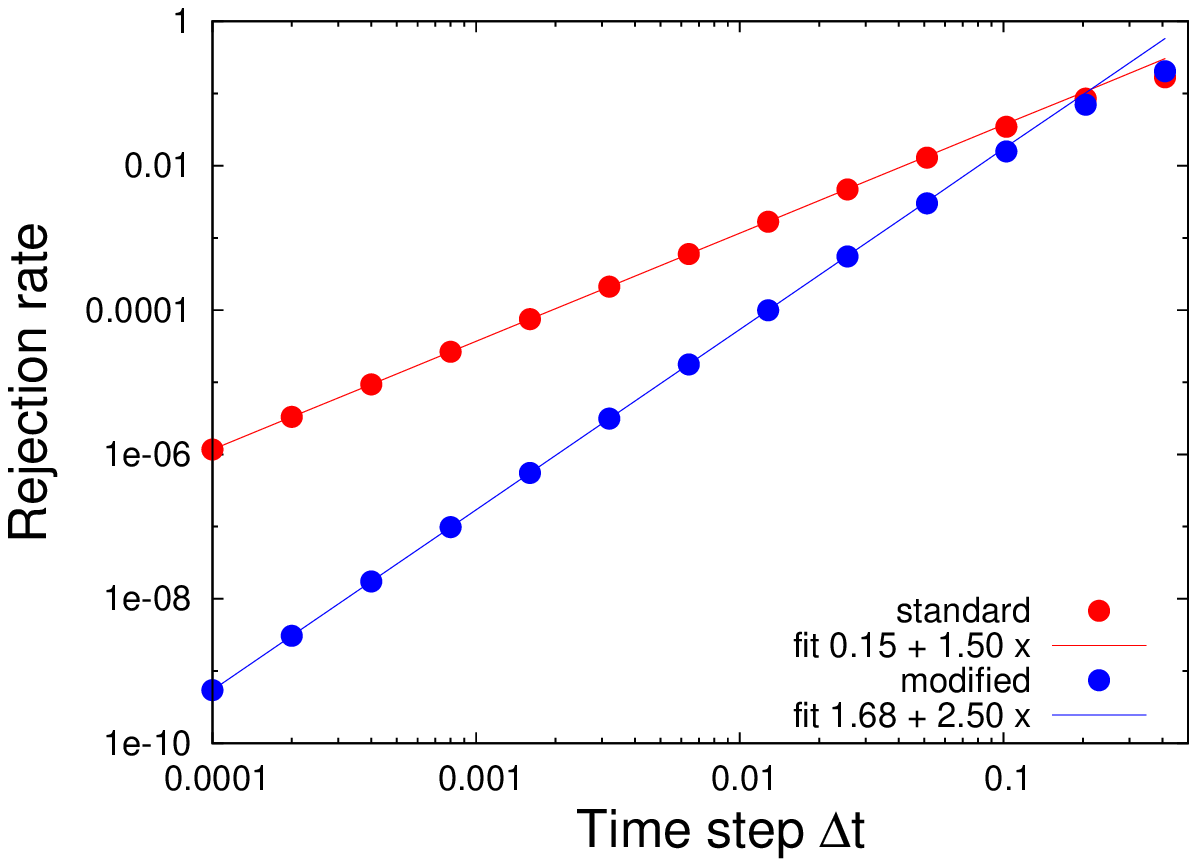}
\end{center}
\caption{\label{fig:strong_error} Left: Strong error over the trajectory as a function of the timestep. Right: Average rejection rate as a function of the timestep. The results obtained with MALA correspond to 'standard' while the ones obtained with the scheme~\eqref{eq:def_Psimod} correspond to 'modified'. For both plots, we superimpose a linear fit in log-log scale, with $x = \log(\dt)$.}
\end{figure}

\section{Reducing the bias in the computation of transport coefficients}
\label{sec:transport}

Our aim in this section is to modify the standard MALA algorithm in order to obtain better approximations of transport coefficients such as the self-diffusion. There are two complementary ways to do so: consider better proposal functions, and replace the Metropolis-Hastings rule with the Barker rule~\cite{Barker65}. 

We start by reviewing in Section~\ref{sec:def_mob_diff} the definitions of the self-diffusion for the continuous dynamics~\eqref{eq:dynamics}, and recall the known error estimates of the approximations of these quantities based on MALA in Section~\ref{sec:error_diff_MALA}. We next propose in Section~\ref{sec:schemes_modified_weak} modified schemes to compute more accurately transport coefficients, and state the related error estimates in Section~\ref{sec:error_modified_weak}. As we show in Section~\ref{sec:multiplicative_diff}, the use of a Barker rule also allows to reduce errors for dynamics with multiplicative noise. Finally, we illustrate the theoretical results of this section by numerical experiments in Section~\ref{sec:numerics_diffusion}.

\subsection{Definition of the self-diffusion}
\label{sec:def_mob_diff}

We briefly recall the setting already presented in~\cite[Section~1.2]{FHS14}. Since the positions~$q_t$ are restricted to the periodic domain~$\cM$, they are uniformly bounded in time. To obtain a diffusive behavior from the evolution of~$q_t$, we consider the following additive functional defined on the whole space~$\R^{d}$: starting from $Q_0 = q_0$,
\begin{equation}
  \label{eq:def_Q_t}
  Q_t = Q_0 - \beta \int_0^t \nabla V(q_s) \, ds + \sqrt{2} \, W_t.
\end{equation}
The difference with $q_t$ is that $Q_t$ is not reprojected in~$\cM$ by the periodization procedure (By this, we mean that we do not choose among all the images of $Q_t$ by translations of the lattice $\mathbb{Z}^d$ the one for which all components are in the interval~$[0,1)$). The diffusion tensor is then given by the following limit (which exists in our context thanks to the existence of a spectral gap for the generator~$\cL$, see for instance~\cite[Theorem~1]{FHS14}): 
\begin{equation}
\label{eq:def_diffusion_coeff}
\mathscr{D} = \lim_{t \to +\infty} \mathbb{E}\left(\frac{Q_t-Q_0}{\sqrt{t}} \otimes \frac{Q_t-Q_0}{\sqrt{t}}\right),
\end{equation}
where the expectation is over all realizations of the continuous dynamics~\eqref{eq:dynamics}, starting from initial conditions distributed according to the Boltzmann-Gibbs measure~$\mu$ defined in~\eqref{eq:Gibbs}. In fact, as made precise in~\cite[Section~1.2]{FHS14}, the self-diffusion constant $\mathcal{D}$, defined as the average mean-square displacement of the individual particles, has two equivalent expressions:
\begin{align}
\mathcal{D} = \frac{1}{2d}\mathrm{Tr}(\mathscr{D})& = \lim_{t \to +\infty} \mathbb{E}\left(\frac{1}{2dt} \sum_{i=1}^{d} (Q_{i,t} - Q_{i,0})^2 \right)  \label{eq:Einstein_continuous} \\
& = 1 - \frac{\beta^2}{d} \int_0^{+\infty} \mathbb{E}\Big[ \nabla V(q_t)^T \nabla V(q_0)\Big] dt. \label{eq:GK_continuous}
\end{align}
The expression~\eqref{eq:Einstein_continuous} is called the \emph{Einstein formula}. The second expression~\eqref{eq:GK_continuous} involves an integrated autocorrelation function. In accordance with the standard physics and chemistry nomenclature, we call~\eqref{eq:GK_continuous} the \emph{Green-Kubo formula} for the self-diffusion in the sequel.

\subsection{Error estimates for MALA}
\label{sec:error_diff_MALA}

For the MALA scheme described in Section~\ref{sec:MALA}, the self-diffusion can be estimated by either discretizing~\eqref{eq:Einstein_continuous} or~\eqref{eq:GK_continuous}. In both cases, the error is of order~$\dt$. For the Green-Kubo approach, it is proved in~\cite[Section~2.1]{FHS14} (by an adaptation of the results obtained in~\cite{LMS13}) that the integrated autocorrelation of smooth observables with average~0 with respect to~$\mu$ can be estimated up to error terms of order~$\dt$. More precisely, introducing
\[
\widetilde{C}^\infty(\mathcal{M}) = \left\{ \psi \in C^\infty(\mathcal{M}) \, \left| \int_\mathcal{M} \psi \, d\mu = 0 \right. \right\}, 
\] 
it is stated in~\cite[Theorem~2]{FHS14} that, for any $\psi,\varphi \in \widetilde{C}^{\infty}(\cM)$, there exists $\dt^* > 0$ and $R > 0$ such that, for any $0 < \dt \leq \dt^*$,
\begin{equation}
\label{eq:GK_general_MALA}
\int_0^{+\infty} \mathbb{E}\Big[\psi(q_t)\,\varphi(q_0)\Big]  dt = \dt \sum_{n=0}^{+\infty} \mathbb{E}_\dt \left[ \psi(q^n)\, \varphi(q^0) \right] + \dt \, r_{\psi,\varphi,\dt},
\end{equation}
with $|r_{\psi,\varphi,\dt}| \leq R$, and where the expectation on the left hand side of the above equation is with respect to initial conditions $q_0 \sim \mu$ and over all realizations of the dynamics~\eqref{eq:dynamics}, while the expectation on the right hand side is with respect to initial conditions $q^0 \sim \mu$ and over all realizations of the MALA scheme. As a corollary, 
\begin{equation}
  \label{eq:error_bounds_GK_MALA}
  \mathcal{D}^{\rm GK}_\dt = \mathcal{D} + \dt\, \widetilde{\mathcal{D}}^{\rm GK}_\dt,
\end{equation}
where $\widetilde{\mathcal{D}}^{\rm GK}_\dt$ is uniformly bounded for~$\dt$ sufficiently small, and where the numerically computed self-diffusion reads
\begin{equation}
\label{eq:approx_GK_MALA}
\mathcal{D}^{\rm GK}_\dt = 1 - \frac{\beta^2 \dt}{d} \sum_{n=0}^{+\infty} \mathbb{E}_\dt \left[ \nabla V(q^n)^T \nabla V(q^0) \right].
\end{equation}

For the approach based on Einstein's formula~\eqref{eq:Einstein_continuous}, we introduce, in accordance with the definition~\eqref{eq:def_Q_t}, a discrete additive functional allowing to keep track of the diffuse behavior of the Markov chain: Starting from $Q^0 = q^0$,
\begin{equation}
\label{eq:def_Qn}
Q^n = \sum_{k=0}^{n-1} \delta_\dt\left(q^k,G^k,U^k\right),
\end{equation}
with
\begin{equation}
\label{eq:def_delta}
\delta_\dt\left(q^k,G^k,U^k\right) = \ind_{U^k \leq A_\dt\left(q^k,\Phi_\dt\left(q^{k},G^k\right)\right)}\left(\Phi_\dt\left(q^{k},G^k\right) - q^n\right).
\end{equation}
While the Markov chain $(q^n)_{n \geq 0}$ remains in~$\cM$, the additive functional~$(Q^n)_{n \geq 0}$ has values in~$\R^{d}$. The diffusion tensor actually computed by the numerical scheme is
\begin{equation}
\label{eq:def_D_Einstein_dt}
\mathscr{D}_\dt = \lim_{n \to +\infty} \mathbb{E}_\dt\left[\frac{Q^n - Q^0}{\sqrt{n\dt}} \otimes \frac{Q^n - Q^0}{\sqrt{n\dt}} \right], 
\end{equation}
where the expectation on the right hand side is with respect to initial conditions $Q^0 = q^0 \sim \mu$ and for all realizations of the numerical scheme. It is shown in~\cite[Theorem~3]{FHS14} that there exists $\dt^* > 0$ such that, for any $0 < \dt \leq \dt^*$, 
\[
\mathscr{D}_\dt = \mathscr{D} + \dt \, \widetilde{\mathscr{D}}_\dt,
\]
where the coefficients of the symmetric matrix $\widetilde{\mathscr{D}}_\dt \in \R^{dN \times dN}$ are uniformly bounded for $\dt \leq \dt^*$. As an immediate corollary,
\begin{equation}
\label{eq:error_bounds_Einstein}
\mathcal{D}^{\rm Einstein}_\dt = \mathcal{D} + \dt \, \widetilde{\mathcal{D}}^{\rm Einstein}_\dt, 
\qquad
\mathcal{D}^{\rm Einstein}_\dt = \frac{1}{2dN} \mathrm{Tr}\left(\mathscr{D}_\dt\right),
\end{equation}
where $\widetilde{\mathcal{D}}^{\rm Einstein}_\dt$ is uniformly bounded for $\dt$ sufficiently small. 

\subsection{Presentation of the modified schemes}
\label{sec:schemes_modified_weak}

The modified scheme presented in Section~\ref{eq:strong_modified} has a very low rejection rate, but it can be checked that the choice~\eqref{choices_mod_mala} leads to a numerical scheme whose weak order is still~1, as for the standard Euler-Maruyama scheme and MALA. Therefore, the error in the computation of the transport coefficients for the scheme of Section~\ref{eq:strong_modified} is still of order~$\dt$. It is not sufficient to aim for lower rejection rates to improve transport properties.

We present in this section a way to decrease the errors in the transport coefficients, even if the rejection rate is still of order~$\dt^{3/2}$ -- and maybe even more surprisingly if the rejection rate is close to~1/2! The first idea is to introduce two proposals $\widetilde{q}^{n+1} = \Phi_\dt(q^n,G^n)$ which are alternatives to the standard Euler-Maruyama proposal~\eqref{eq:EM}: an implicit midpoint
\begin{equation}
  \label{eq:proposal_weak_2nd_order}
  \widetilde{q}^{n+1} = q^n - \beta \dt \, \nabla V\left(\frac{\widetilde{q}^{n+1}+q^n}{2}\right) + \sqrt{2\dt} \, G^n,
\end{equation}
and 
\begin{equation}
  \label{eq:proposal_HMC}
  \widetilde{q}^{n+1} = q^n - \beta \dt\,\nabla V\left(q^n + \frac{\sqrt{2\dt}}{2}\,G^n\right) + \sqrt{2\dt} \, G^n.
\end{equation}
The latter proposal has already been used in~\cite{BDV13}. Since the first proposal uses an implicit scheme, the timestep has to be sufficiently small for $\Phi_\dt$ to be well defined (see Lemma~\ref{lem:geometric_ergodicity} below). 
For the second proposal, it seems that computing the probability of reverse moves from $\widetilde{q}^{n+1}$ to~$q^n$ is going to be difficult in view of the explicit nonlinear dependence on the Gaussian increment~$G^n$. It turns out however that~\eqref{eq:proposal_HMC} can be reformulated as one step of a reversible, symplectic scheme starting from a random momentum $p^n = \beta^{-1/2} G^n$ and with a timestep $h = \sqrt{2\beta\dt}$:
\begin{equation}
\label{eq:Verlet_reformulation}
\left\{ 
\begin{aligned}
q^{n+1/2} & = q^n + \frac{h}{2} \, p^n, \\
p^{n+1}  & = p^n - h \, \nabla V\left(q^{n+1/2}\right), \\
\widetilde{q}^{n+1}  & = q^{n+1/2} + \frac{h}{2} \, p^{n+1}. \\
\end{aligned} 
\right.
\end{equation}
As made precise below in~\eqref{eq:def_alpha_Barker_HMC}, this allows to compute ratios of transition probabilities relying on Hybrid Monte Carlo algorithms~\cite{DKPR87}, see for instance the presentation in~\cite[Section~2.1.3]{LRS10}. Note however that this scheme is different from the one usually used in Hybrid Monte Carlo since it corresponds to a position Verlet method rather than a velocity Verlet method, which would start by integrating the momenta (and which would reduce to a standard Euler-Maruyama discretization with the choice $p^n = \beta^{-1/2} G^n$ and with a timestep $h = \sqrt{2\beta\dt}$). 

For practical purposes, \eqref{eq:proposal_HMC} should be preferred over \eqref{eq:proposal_weak_2nd_order} since there is no extra computational cost contrarily to the implicit proposal~\eqref{eq:proposal_weak_2nd_order}, which has to be solved using fixed point iterations or a Newton method. 

\begin{remark}[Higher order HMC schemes]
  \label{rmk:HMC}
  It is possible in principle to integrate the Hamiltonian dynamics using a reversible scheme of arbitrary accuracy, relying on appropriate composition methods (see~\cite[Sections~II.4 and~V.3]{HairerLubichWanner06}). As for~\eqref{eq:proposal_HMC}, discretizations of the overdamped Langevin dynamics~\eqref{eq:dynamics} are then obtained by starting from a random momentum $p^n = \beta^{-1/2} G^n$ and using a timestep $h = \sqrt{2\beta\dt}$. Since the energy difference in the underlying Hamiltonian scheme can be made arbitrarily small by increasing the order of the composition method, this seemingly leads to proposals with an arbitrarily low rejection rate. In practice however, the efficiency of these schemes is limited by the numerical instabilities related to round off errors. We therefore focus on simple composition methods of order~2 such as~\eqref{eq:Verlet_reformulation}.  
\end{remark}

The second idea to improve the computation of transport coefficients is to consider acceptance/rejection rules different from the Metropolis procedure. In fact, we consider the Barker rule~\cite{Barker65} (rather than other possible rules, see Remark~\ref{rmk:other_rules}). More precisely, the acceptance rates for the Metropolis and Barker rules respectively read
\begin{equation}
  \label{eq:acceptance_rate}
  A^{\rm MH}_\dt(q^n,\widetilde{q}^{n+1}) = \min\left(1, \rme^{-\alpha_\dt(q^n,\widetilde{q}^{n+1})}\right),
  \qquad
  A^{\rm Barker}_\dt(q^n,\widetilde{q}^{n+1}) = \frac{\rme^{-\alpha_\dt(q^n,\widetilde{q}^{n+1})}}{1 + \rme^{-\alpha_\dt(q^n,\widetilde{q}^{n+1})}},
\end{equation}
with, for the midpoint proposal~\eqref{eq:proposal_weak_2nd_order},
\begin{equation}
\begin{aligned}
  \label{eq:def_alpha_Barker}
  \alpha_\dt(q,q') & = \beta \Big( V(q') - V(q) \Big) + \frac{1}{4\dt} \left(\left[q-q'+\beta \dt \, \nabla V\left(\frac{q+q'}{2}\right)\right]^2 - \left[q'-q+\beta \dt \, \nabla V\left(\frac{q+q'}{2}\right)\right]^2\right) \\
  & = \beta \left( V(q') - V(q) - \nabla V\left(\frac{q+q'}{2}\right) \cdot (q'-q) \right),
  \end{aligned}
\end{equation}
while, for~\eqref{eq:proposal_HMC}, we define the ratio when $q' = \Phi_\dt(q,G)$ (since we shall only need it in that case):
\begin{equation}
  \label{eq:def_alpha_Barker_HMC}
  \alpha_\dt\Big(q,\Phi_\dt(q,G)\Big) = \beta \Big( V(\Phi_\dt(q,G)) - V(q) \Big) + \frac{1}{2} \left[ \left(G-\sqrt{2\dt} \, \beta \nabla V\left(q+\frac{\sqrt{2\dt}}{2}\,G\right)\right)^2- G^2 \right].
\end{equation}
With the notation used in~\eqref{eq:Verlet_reformulation}, the right-hand side of the above equation can be rewritten as $\beta \left[ H(\widetilde{q}^{n+1},p^{n+1})-H(q^n,p^n) \right]$, where $H(q,p) = V(q) + p^2/2$ is the Hamiltonian of the system, and $p^n = \beta^{-1/2} G^n$. The new configuration is in all cases
\[
q^{n+1} = \Psi_\dt(q^n,G^n,U^n) = q^n + \ind_{U^n \leq A_\dt\left(q^n,\Phi_\dt\left(q^n,G^n\right)\right)} \Big(\Phi_\dt\left(q^n,G^n\right)-q^n\Big),
\]
where $A_\dt$ is one of the rates defined in~\eqref{eq:acceptance_rate}. Note that, for the Barker rule, the average rejection rate is close to~$1/2$ in the limit $\dt\to 0$ (as made precise in~\eqref{eq:scaling_rejection_rate_Barker}). Let us also emphasize that the use of a Barker rule leads to a Markov chain which is reversible with respect to~$\mu$ (see for instance~\cite[Section~2.1.2.2]{LRS10}).

Let us now discuss the ergodic properties of the schemes based on the proposals~\eqref{eq:proposal_weak_2nd_order} and~\eqref{eq:proposal_HMC}. We introduce the evolution operator
\[
P_\dt \psi(q) = \mathbb{E}\Big( \psi\left(q^{n+1}\right) \, \Big| \, q^n = q \Big),
\]
as well as the discrete generator
\begin{equation}
\label{eq:def_discrete_generator}
\begin{aligned}
\frac{P_\dt-\Id}{\dt} \psi(q) & = \frac{1}{\dt} \left[ \mathbb{E}\Big( \psi\left(q^{n+1}\right) \, \Big| \, q^n = q \Big) - \psi(q^n) \right] \\
& = \int_{\R^d} A_\dt(q,\Phi_\dt(q,g)) \frac{\psi(\Phi_\dt(q,g)) - \psi(q)}{\dt} \, \frac{\rme^{-g^2/2}}{(2\pi)^{d/2}} \, dg,
\end{aligned}
\end{equation}
where the expectation in the first line is over $G^n,U^n$ and $A_\dt$ is one of the rates defined in~\eqref{eq:acceptance_rate}. We also define the space of bounded functions with average~0 with respect to~$\mu$:
\[
\widetilde{L}^\infty(\cM) = \left\{ \varphi \in L^\infty(\cM) \, \left| \int_\cM f \, d\mu = 0 \right. \right\}.
\]
The operator norm on this space is 
\[
\| T \|_{\mathcal{B}(\widetilde{L}^{\infty})} = \sup_{\substack{f \in \widetilde{L}^\infty(\cM) \\ f \neq 0}} \frac{\| Tf\|_{L^\infty(\cM)}}{\| f\|_{L^\infty(\cM)}}.
\]
We can then state the following lemma (see Section~\ref{sec:proof:lem:geometric_ergodicity} for the proof).

\begin{lemma}
  \label{lem:geometric_ergodicity}
  There exists $\dt^* > 0$ such that, for any $0 < \dt \leq \dt^*$, the proposal~\eqref{eq:proposal_weak_2nd_order} is well defined whatever $q^n \in \cM$ and $G^n \in \R^d$. In addition, there exist $C,\lambda>0$ such that, for all $n \in \mathbb{N}$ and $0 < \dt \leq \dt^*$, and any $f \in \widetilde{L}^{\infty}(\cM)$, the evolutions generated by either the proposal~\eqref{eq:proposal_weak_2nd_order} or~\eqref{eq:proposal_HMC}, together with a Metropolis-Hastings rule or a Barker rule, satisfy
  \begin{equation}
    \label{eq:geom_ergod}
    \left \| P_\dt^n f \right \|_{L^\infty} \leq C \, \rme^{-\lambda n \dt} \|f\|_{L^\infty}.
  \end{equation}
  As a consequence, there exists $K > 0$ such that, uniformly in $\dt \in (0,\dt^*]$,
  \begin{equation}
    \label{eq:bound_discrete_generator}
    \left\| \left(\frac{\Id - P_\dt}{\dt}\right)^{-1} \right\|_{\mathcal{B}(\widetilde{L}^{\infty})} \leq K.
  \end{equation}
\end{lemma}

Such ergodicity results are already provided in~\cite{BH13,FHS14} for proposals based on a simple Euler-Maruyama discretization and a Metropolis-Hastings rule. The resolvent bound~\eqref{eq:bound_discrete_generator} is a crucial ingredient to appropriately estimate errors in transport coefficients.

\subsection{Error estimates on the self-diffusion for the modified schemes}
\label{sec:error_modified_weak}

We are now able to state error estimates concerning the numerical approximation of Green-Kubo formulas.

\begin{theorem}[Improved Green-Kubo formula]
  \label{thm:improved_GK}
  Set $a=1/2$ and $\alpha = 2$ when the Barker rule is used with either the midpoint proposal~\eqref{eq:proposal_weak_2nd_order} or the HMC proposal~\eqref{eq:proposal_HMC}, and $a=1$ and $\alpha = 3/2$ when the Metropolis-Hastings rule is used for these proposals. 
  Consider two observables $\psi,\varphi \in \widetilde{C}^{\infty}(\cM)$. Then, there exists $\dt^* > 0$ such that, for any $0 < \dt \leq \dt^*$,
  \[
  \int_0^{+\infty} \mathbb{E}\Big[\psi(q_t)\,\varphi(q_0)\Big]  dt = \dt \left( a\sum_{n=0}^{+\infty} \mathbb{E}_\dt \left[ \psi(q^n)\, \varphi(q^0) \right] - \frac12 \int_\cM \psi \varphi \, d\mu\right) + \dt^{\alpha} r_{\psi,\varphi,\dt},
  \]
  with $r_{\psi,\varphi,\dt}$ uniformly bounded (with respect to~$\dt$), and where the expectation on the left hand side of the above equation is with respect to initial conditions $q_0 \sim \mu$ and over all realizations of the dynamics~\eqref{eq:dynamics}, while the expectation on the right hand side is with respect to initial conditions $q^0 \sim \mu$ and over all realizations of the numerical scheme. 
\end{theorem}

Note that, for the Barker rule, the quadrature formula corresponds to the standard way of computing Green-Kubo formulas (as on the right-hand side of~\eqref{eq:GK_general_MALA}) except that the first term $n=0$ is discarded and a global factor $1/2$ is required in order to correct for the fact that the average rejection rate of the Barker algorithm is close to $1/2$. 

\begin{remark}[Increase in the asymptotic variance for the Barker rule]
\label{rmk:variance}
The asymptotic variance of the time average of a function of interest~$f$ for a Markov chain $(q^n)_{n \geq 0}$ with invariant measure~$\mu$ is (assuming without loss of generality that the average of~$f$ with respect to~$\mu$ vanishes)
\[
\sigma^2(f) = \int_\cM f^2 \, d\mu + 2 \sum_{n=1}^{+\infty} \mathbb{E}\Big(f(q^n)f(q^0)\Big), 
\]
where the expectation is with respect to initial conditions $q^0 \sim \mu$ and over all realizations of the Markov chain; see for instance~\cite[Section~2.3.1.2]{LRS10} and~\cite{ActaNum} and references therein. In view of the estimates provided by Theorem~\ref{thm:improved_GK}, the asymptotic variance of time averages computed along trajectories of schemes based on the Metropolis rule are, for $\dt$ sufficiently small,
\[
\dt\,\sigma_{{\rm MH},\dt}^2(f) = \int_0^{+\infty} \mathbb{E}\Big[f(q_t)\,f(q_0)\Big]  dt + \mathrm{O}(\dt),
\]
while, for the Barker rule,
\[
\dt\,\sigma_{\Barker,\dt}^2(f) = 2\int_0^{+\infty} \mathbb{E}\Big[f(q_t)\,f(q_0)\Big]  dt + \mathrm{O}(\dt).
\]
This shows that $\dt \, \sigma_{\Barker,\dt}^2(f) = 2\dt\, \sigma_{{\rm MH},\dt}^2(f) + \mathrm{O}(\dt)$, which allows to quantify the increase in the variance as $\dt\to 0$ due to the increased rejection rate of the Barker rule.
\end{remark}

As a corollary of Theorem~\ref{thm:improved_GK}, we obtain error bounds on the computation of the self-diffusion as
\begin{equation}
  \label{eq:error_bounds_GK}
  \mathcal{D}^{\rm GK}_\dt = \mathcal{D} + \dt^{\alpha}\, \widetilde{\mathcal{D}}^{\rm GK}_\dt,
\end{equation}
where $\mathcal{D}$ is the diffusion coefficient of the continuous dynamics (see~\eqref{eq:GK_continuous}), $\widetilde{\mathcal{D}}^{\rm GK}_\dt$ is uniformly bounded for~$\dt$ sufficiently small, and where the numerically computed self-diffusion reads, for the Metropolis rule,
\[
\mathcal{D}^{\rm GK}_\dt = 1 - \frac{\beta^2 \dt}{d} \left(\frac12 \int_\cM |\nabla V|^2 \, d\mu  + \sum_{n=1}^{+\infty} \mathbb{E}_\dt \left[ \nabla V(q^n)^T \nabla V(q^0) \right] \right).
\]
while, for the Barker rule,
\begin{equation}
\label{eq:approx_GK_addtive}
\mathcal{D}^{\rm GK}_\dt = 1 - \frac{\beta^2 \dt}{2d} \sum_{n=1}^{+\infty} \mathbb{E}_\dt \left[ \nabla V(q^n)^T \nabla V(q^0) \right].
\end{equation}
Compared to the results of~\cite{FHS14} (recalled in~\eqref{eq:error_bounds_GK_MALA}), \eqref{eq:error_bounds_GK} allows to compute the diffusion coefficient up to error terms of order~$\dt^{2}$ rather than~$\dt$ provided the Barker rule is used. 

\begin{remark}
\label{rmk:compensation}
To numerically approximate formulas such as~\eqref{eq:approx_GK_addtive}, it is necessary to introduce some truncation on the maximal number of timesteps in the sum, and estimate the expectation by independent realizations. In fact, the truncation should match the physical time $T_{\rm corr}$ for which correlations are sufficiently small for the underlying continuous dynamics. It therefore scales as $T_{\rm corr}/\dt$. When the Barker rule is used, the rejection rate is close to 1/2, so that the correlation roughly decays twice slower than for Metropolis-based discretizations. The number of timesteps to be used to compute the integrated autocorrelation should then scale as $2T_{\rm corr}/\dt$. On the other hand, the timestep can be increased a lot since the bias is of order $\dt^2$ rather than $\dt$ (for standard MALA) or $\dt^{3/2}$ (for modified Metropolis-based schemes). Therefore, fixing an admissible error level $\varepsilon \ll 1$ for the bias, we see that the computational cost of one realization for the modified schemes with the Barker rule scales as $2T_{\rm corr}/\sqrt{\varepsilon}$, which is much smaller than the cost $T_{\rm corr}/\varepsilon$ for one realization with standard MALA and $T_{\rm corr}/\varepsilon^{2/3}$ for modified Metropolis-based schemes.
\end{remark}

\medskip

We next present error estimates on the numerical approximation of Einstein's formula. We still define $Q^n$ by~\eqref{eq:def_Qn}, but replace the proposal function $\Phi_\dt$ by the ones associated with the schemes presented in Section~\ref{sec:schemes_modified_weak}, and consider the appropriate acceptance rates~\eqref{eq:acceptance_rate}.

\begin{theorem}[Improved Einstein fluctuation formula]
  \label{thm:fluctuation}
  Set $a=1/2$ and $\alpha = 2$ when the Barker rule is used with either the midpoint or the HMC proposal, and $a=1$ and $\alpha = 3/2$ when the Metropolis-Hastings rule is used for these proposals. Consider the unperiodized displacement $Q^n$ defined~\eqref{eq:def_Qn}, with $Q^0 = q^0 \sim \mu$. Then,
  \[
  \mathscr{D}_\dt = \lim_{n \to +\infty} \mathbb{E}\left[ \frac{Q^n-Q^0}{\sqrt{n\dt}} \otimes \frac{Q^n-Q^0}{\sqrt{n\dt}} \right] = a \mathscr{D} + \dt^{\alpha}\, \widetilde{\mathscr{D}}_\dt,
  \]
  where the expectation on the left hand side of the above equation is with respect to initial conditions $q^0 \sim \mu$ and over all realizations of the numerical scheme, and where $\widetilde{\mathscr{D}}^{\rm Einstein}_\dt$ is a $d \times d$ matrix whose coefficients are uniformly bounded for $\dt$ sufficiently small.
\end{theorem}

It is possible to generalize such fluctuation formulas to more general increments $f(q,\delta_\dt(q,G,U))$ instead of $\delta_\dt(q,G,U)$, but the corresponding formulas are quite complicated and we therefore refrain from stating them. An immediate corollary of Theorem~\ref{thm:fluctuation} is the following error estimate on the self-diffusion:
\begin{equation}
\label{eq:prediction_Einstein_additive}
\mathcal{D}_\dt^{\rm Einstein} = \mathcal{D}  + \dt^\alpha \, \widetilde{\mathcal{D}}_\dt^{\rm Einstein}, 
\qquad 
\mathcal{D}_\dt^{\rm Einstein} = \frac{1}{2ad}\Tr\left(\mathscr{D}_\dt\right),
\end{equation}
where we recall $a=1$ and $\alpha = 3/2$ for the Metropolis-Hastings rule, while $a=1/2$ and and $\alpha = 2$ for the Barker rule. The error estimates provided by this formula are consistent with the ones provided by the Green-Kubo approach (see~\eqref{eq:error_bounds_GK}).

\medskip

The proofs of Theorems~\ref{thm:improved_GK} and~\ref{thm:fluctuation}, which can be read in Sections~\ref{sec:proof_thm:improved_GK} and~\ref{sec:proof_improved_Einstein} respectively, crucially rely on the following weak-type expansion of the discrete generator associated with the proposals~\eqref{eq:proposal_weak_2nd_order} or~\eqref{eq:proposal_HMC}. 

\begin{lemma}
  \label{lem:weak_type_expansion}
  For any $\psi \in C^\infty(\cM)$,
  \begin{equation}
    \label{eq:weak_type_expansion}
    \frac{P_\dt-\Id}{\dt} \psi = a \left(\cL \psi + \frac{\dt}{2} \cL^2 \psi\right) + \dt^{\alpha} \, r_{\psi,\dt},
\end{equation}
  where $r_{\psi,\dt}$ is uniformly bounded in $L^\infty(\cM)$ for $\dt$ sufficiently small; and $\alpha = 2$, $a=1/2$ for schemes using the Barker rule, while $\alpha =3/2$, $a=1$ for Metropolis-based methods.
\end{lemma}

The proof of this lemma is provided in Section~\ref{sec:proof:lem:weak_type_expansion}. Let us emphasize that the proposals corrected by the Barker rule do not lead to evolution operators $P_\dt$ which have weak second order, and cannot even be written as $\rme^{\dt \cL/2}$ up to corrections of order~$\dt^3$. This is the reason why we call~\eqref{eq:weak_type_expansion} a ``second order type weak expansion'', by which we mean that the $\dt^2$ term in the expansion of $P_\dt$ is proportional to~$\cL^2$. This is in fact all we need to perform the proofs of Theorems~\ref{thm:improved_GK} and~\ref{thm:fluctuation}.

\begin{remark}
  \label{rmk:add_drift}
  As made precise in Section~\ref{sec:proof:lem:weak_type_expansion}, it is possible to replace the drift $-\nabla V$ in the proposals with more general drifts, namely $-\nabla V + \dt \, \widetilde{F}$ for the midpoint proposal~\eqref{eq:proposal_weak_2nd_order}, or $-\nabla V + \dt \, \nabla \widetilde{V}$ for the HMC proposal~\eqref{eq:proposal_HMC}. The choice $\widetilde{F} = -\nabla (\Delta V)/4$ ensures that the underlying scheme is of weak order~2 (see~\cite[Theorem~3.2]{ACVZ12}). However, even when $\widetilde{F} = 0$, in which case the underlying scheme is only of weak order~1, the average drift introduced by the acceptance/rejection procedures automatically corrects the drift and ensures that the expansion~\eqref{eq:weak_type_expansion} holds whatever the expression of~$\widetilde{F}$.
\end{remark}

\subsection{Extension to diffusions with multiplicative noise}
\label{sec:multiplicative_diff}

\subsubsection{Description of the dynamics}

In this section, we extend the results obtained for the dynamics~\eqref{eq:dynamics} with additive noise to dynamics with multiplicative noise. More precisely, we consider a position-dependent diffusion matrix $M \in \R^{d \times d}$, assumed to be a smooth function of the positions~$q \in \cM$ with values in the space of symmetric, definite, positive matrices. The corresponding overdamped Langevin dynamics reads
\begin{equation}
\label{eq:dynamics_mult}
dq_t = \Big(- \beta M(q_t)\nabla V(q_t) + \mathrm{div}(M)(q_t) \Big) dt + \sqrt{2} M^{1/2}(q_t) \, dW_t,
\end{equation}
where $\mathrm{div} M$ is the vector whose $i$th component is the divergence of the $i$th column of the matrix $M = [M_1,\dots,M_d]$: 
\[
\mathrm{div} M = \left(\mathrm{div} M_1,\dots,\mathrm{div} M_d\right)^T.
\] 
The generator of~\eqref{eq:dynamics_mult} acts on test functions~$\varphi$ as
\[
\cL \varphi = \Big(-\beta M \nabla V + \mathrm{div}(M)\Big)^T \nabla \varphi + M : \nabla^2 \varphi.
\]
A simple computation gives, for any smooth functions $\varphi,\phi$,
\[
\left\langle \varphi, \cL \phi \right\rangle_{L^2(\mu)} = - \int_\cM (\nabla \varphi)^T M (\nabla \phi) \, d\mu. 
\]
This shows that the canonical measure~$\mu$ is invariant by the dynamics~\eqref{eq:dynamics_mult} (choose $\varphi = 1$). 

\subsubsection{Definition of the self-diffusion}

To simplify the notation, we introduce the total drift
\begin{equation}
\label{eq:def_F_mult}
F(q) = - \beta M(q)\nabla V(q) + \mathrm{div}(M)(q).
\end{equation}
A straightforward adaptation of the proof of~\cite[Theorem~1]{FHS14} shows that the effective diffusion tensor associated with the dynamics~\eqref{eq:dynamics_mult} is well defined and that
\[
\mathscr{D} = \lim_{t \to +\infty} \mathbb{E}\left(\frac{Q_t-Q_0}{\sqrt{t}} \otimes \frac{Q_t-Q_0}{\sqrt{t}}\right) = 2\left( \int_\mathcal{M} M(q) \, \mu(dq) - \int_0^{+\infty} \mathbb{E}\left[ F(q_t) \otimes F(q_0) \right] dt \right).
\]
where
\[
Q_t = q_0 + \int_0^t F(q_s) \,ds + \sqrt{2} \int_0^t M^{1/2}(q_s) \, dW_s
\]
has values in~$\R^d$ and where the expectations are over all initial conditions $q_0$ distributed according to~$\mu$, and over all realizations of the overdamped Langevin dynamics~\eqref{eq:dynamics_mult}. The associated effective self-diffusion 
\begin{equation}
  \label{eq:self_diff_mult_noise}
  \mathcal{D} = \frac{1}{2d}\mathrm{Tr}(\mathscr{D}) = \frac1d\left( \int_\cM \Tr(M) \, d\mu - \int_0^{+\infty} \mathbb{E}\Big[ F(q_t)^T F(q_0)\Big] dt\right)
\end{equation}
can be estimated either via Green-Kubo or Einstein formulas. Note that the above equalities reduce to~\eqref{eq:Einstein_continuous}-\eqref{eq:GK_continuous} when $M = \Id$.

\subsubsection{Proposal functions}

Proposals that can be used in conjunction with a Metropolis-Hastings procedure, and which avoid the computation of the divergence of~$M$, are presented in~\cite{BDV13}. The resulting numerical method is a discretization of the dynamics~\eqref{eq:dynamics_mult} of the same weak order as the Metropolis-Hastings method based on proposals obtained by the standard Euler-Maruyama scheme (compare the results of~\cite[Section~5]{BDV13} and Lemma~\ref{lem:evolution_operator_multiplicative_noise} below). In order to simplify the presentation, we therefore restrict ourselves to the simple scheme
\begin{equation}
  \label{eq:EM_multiplicative}
  \widetilde{q}^{n+1} = \Phi_\dt(q^n,G^n) = q^n + F(q^n) \dt + \sqrt{2\dt} \, M^{1/2}(q^n) \, G^n,
\end{equation}
although its practical implementation may be cumbersome. Our results can however straightforwardly be extended to the proposal function described in~\cite{BDV13}.

As in the previous sections, we consider two types of acceptance/rejection procedures to correct the proposal~\eqref{eq:EM_multiplicative}: a Metropolis-Hastings rule and a Barker rule. The spatial dependence of the diffusion matrix has to be taken into account in the acceptance/rejection criteria. In fact, we still consider the acceptance rates~\eqref{eq:acceptance_rate} but change the definition of~$\alpha_\dt$ as follows:
\begin{equation}
\label{eq:alpha_dt_multiplicative}
\begin{aligned}
\alpha_\dt(q,q') = \beta \Big( V(q')-V(q) \Big) & + \frac{1}{4\dt}\Big(|q-q'-\dt F(q')|^2_{M(q')} - |q'-q-\dt F(q)|^2_{M(q)} \Big) \\& + \frac{1}{2} \Big[ \ln \big(\det M(q') \big) - \ln \big(\det M(q) \big) \Big],
\end{aligned}
\end{equation}
where, for a symmetric, definite, positive matrix~$M$, we introduced the norm $|q|_M = \sqrt{q^T M^{-1} q}$. The crucial estimate to obtain error bounds on Green-Kubo formulas and the effective diffusion is the following result, which is the equivalent of Lemma~\ref{lem:weak_type_expansion} for diffusions with multiplicative noise (see Section~\ref{sec:proof:lem:evolution_operator_multiplicative_noise} for the proof).

\begin{lemma}
\label{lem:evolution_operator_multiplicative_noise}
For any $\psi \in C^\infty(\cM)$,
\begin{equation}
  \label{eq:weak_type_expansion_multiplicative}
  \frac{P_\dt-\Id}{\dt} \psi = a \cL \psi + \dt^{\alpha} \, r_{\psi,\dt},
\end{equation}
where $a=1$ and $\alpha = 1/2$ when a Metropolis-Hastings rule is used, while $a=1/2$ and $\alpha = 1$ when a Barker rule is considered.
\end{lemma}

Note that the remainder terms are much larger than for diffusions with additive noise (where $\alpha=3/2$ for the Metropolis-Hastings rule and $\alpha=2$ for the Barker rule). 

\subsubsection{Error estimates on the self-diffusion}

A straightforward adpatation of the proofs of Lemma~\ref{lem:geometric_ergodicity} (see Remark~\ref{rmk:extension_ergo_mult}), Theorems~\ref{thm:improved_GK} and~\ref{thm:fluctuation} shows that, keeping the same notation as in Section~\ref{sec:error_modified_weak},
\[
\int_0^{+\infty} \mathbb{E}\Big[\psi(q_t)\,\varphi(q_0)\Big] dt = a \sum_{n=0}^{+\infty} \mathbb{E}_\dt \left[ \psi(q^n)\, \varphi(q^0) \right] + \mathrm{O}\left(\dt^\alpha\right),
\]
so that
\begin{equation}
\label{eq:approx_GK_mult}
\frac1d\left( \int_\cM \Tr(M) \, d\mu - a\sum_{n=0}^{+\infty} \mathbb{E}_\dt \left[ \psi(q^n)\, \varphi(q^0) \right]\right) = \mathcal{D} + \mathrm{O}\left(\dt^\alpha\right),
\end{equation}
while
\begin{equation}
\label{eq:approx_Einstein_mult}
\lim_{n \to +\infty} \frac{1}{a} \mathbb{E}\left[ \left(\frac{Q^n-Q^0}{\sqrt{n\dt}}\right)^2\right] = \mathcal{D} + \mathrm{O}\left(\dt^\alpha\right),
\end{equation}
with $a=1$ and $\alpha = 1/2$ when a Metropolis-Hastings rule is used, while $a=1/2$ and $\alpha = 1$ when a Barker rule is considered. Note that, once again, the Barker rule allows to reduce the bias in the computation of the self-diffusion, here from $\sqrt{\dt}$ to~$\dt$.

\subsection{Numerical results}
\label{sec:numerics_diffusion}

We illustrate on a simple example the results obtained in the previous sections. We consider a one-dimensional system with $\cM = \mathbb{T}$, at the inverse temperature $\beta = 1$, and for the simple potential $V(q) = \cos(2\pi q)$ already used in~\cite{FHS14}. For the dynamics with multiplicative noise, we consider the positive diffusion coefficient
\[
M(q) = \left(\frac{1.5 + \cos(2\pi q)}{2}\right)^2.
\]
Reference values for the diffusion constant can be obtained as described in the Appendix. For dynamics with additive noise, $\mathcal{D} = 0.62386$ while $\mathcal{D} = 0.30478$ for dynamics with multiplicative noise. For the midpoint scheme, the proposed configuration is computed using a fixed point algorithm, initialized with the standard Euler-Maruyama scheme encoded by~\eqref{eq:EM}. The convergence treshold is set to $10^{-8}$ (this tolerance should be decreased in order to check the scaling of the rejection for smaller timesteps). About 10 fixed point iterations were needed for convergence for $\dt = 0.01$, and less iterations for smaller timesteps. 

We first study the average rejection rates, based on the computations performed in Section~\ref{sec:proof:lem:weak_type_expansion} and~\ref{sec:proof:lem:evolution_operator_multiplicative_noise}. For the Metropolis rule, the predictions are
\begin{equation}
\label{eq:scaling_Metropolis_rates}
0 \leq \mathbb{E}_\mu\left( 1 - A^{\rm MH}_\dt\Big(q,\Phi_\dt(q,G)\Big) \right) \leq K_{\rm MH} \dt^\alpha,
\end{equation}
with $\alpha = 3/2$ for dynamics with additive noise (see~\eqref{eq:rate_Metropolis_midpoint}), and $\alpha = 1/2$ for dynamics with multiplicative noise (see~\eqref{eq:rate_Metropolis_multiplicative}). The expectation is with respect to initial conditions distributed according to~$\mu$ and for all realizations of the standard Gaussian random variable~$G$. For the Barker rule, the rejection rate satisfies 
\begin{equation}
\label{eq:scaling_Barker_rates}
\left| 2\mathbb{E}_\mu\left( A^\Barker_\dt\Big(q,\Phi_\dt(q,G)\Big) \right) - 1 \right| \leq K_\Barker \dt^\alpha,
\qquad
\mathbb{E}_\mu\left|2A^\Barker_\dt\Big(q,\Phi_\dt(q,G)\Big) - 1 \right| \leq \widetilde{K}_\Barker \dt^\gamma,
\end{equation}
with $\alpha = 3$ and $\gamma = 3/2$ for dynamics with additive noise (see~\eqref{eq:scaling_rejection_rate_Barker} and~\eqref{eq:avg_rejection_rate_Barker}), and $\alpha = 1$ and $\gamma = 1/2$ for dynamics with multiplicative noise (see~\eqref{eq:rate_Barker_multiplicative}). In fact, 
\[
\mathbb{E}_\mu\left|2A^\Barker_\dt\Big(q,\Phi_\dt(q,G)\Big) - 1 \right| \simeq \mathbb{E}_\mu\left( 1 - A^{\rm MH}_\dt\Big(q,\Phi_\dt(q,G)\Big) \right). 
\]
The rejection rates were approximated using empirical averages over $10^9$ iterations at $\beta = 1$, starting from $q^0 = 0$. The results presented in Figure~\ref{fig:rejection} confirm the predicted rates.

\begin{figure}
\begin{center}
\includegraphics[width=8.2cm]{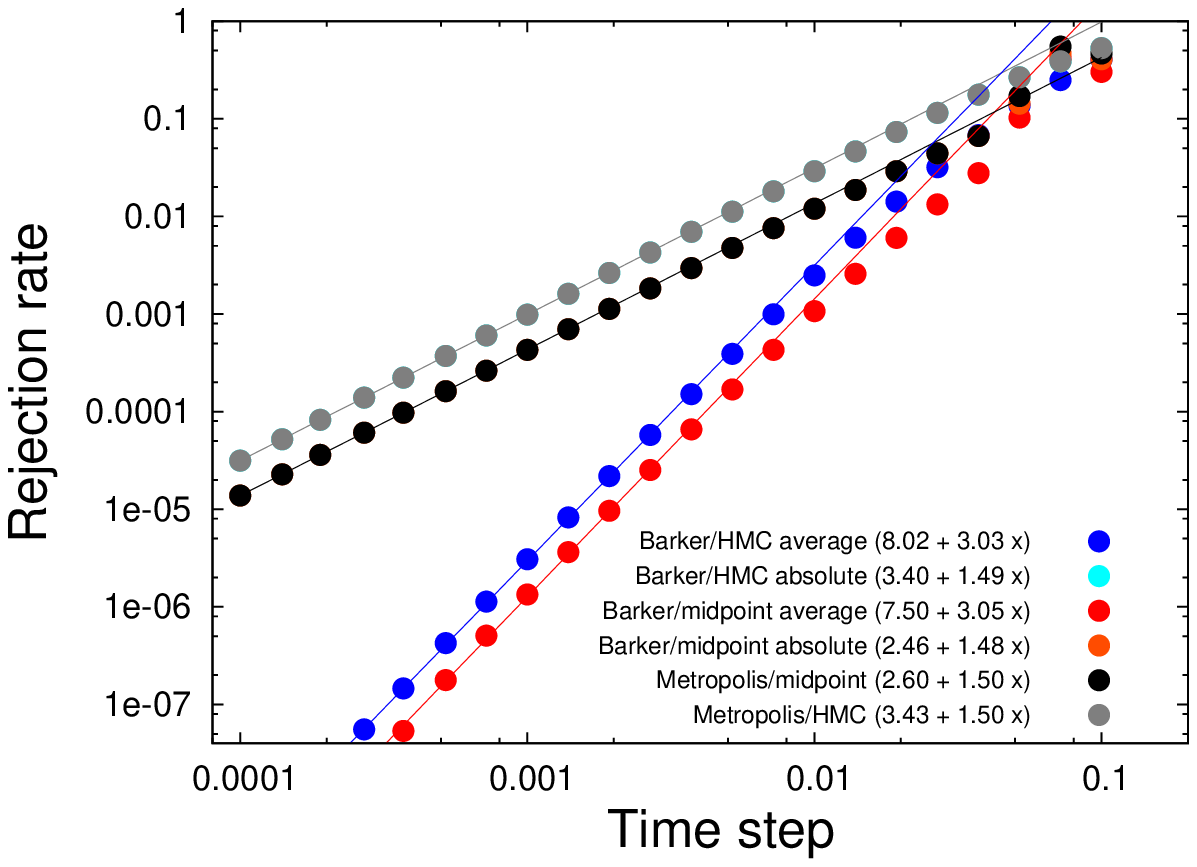}
\includegraphics[width=8.2cm]{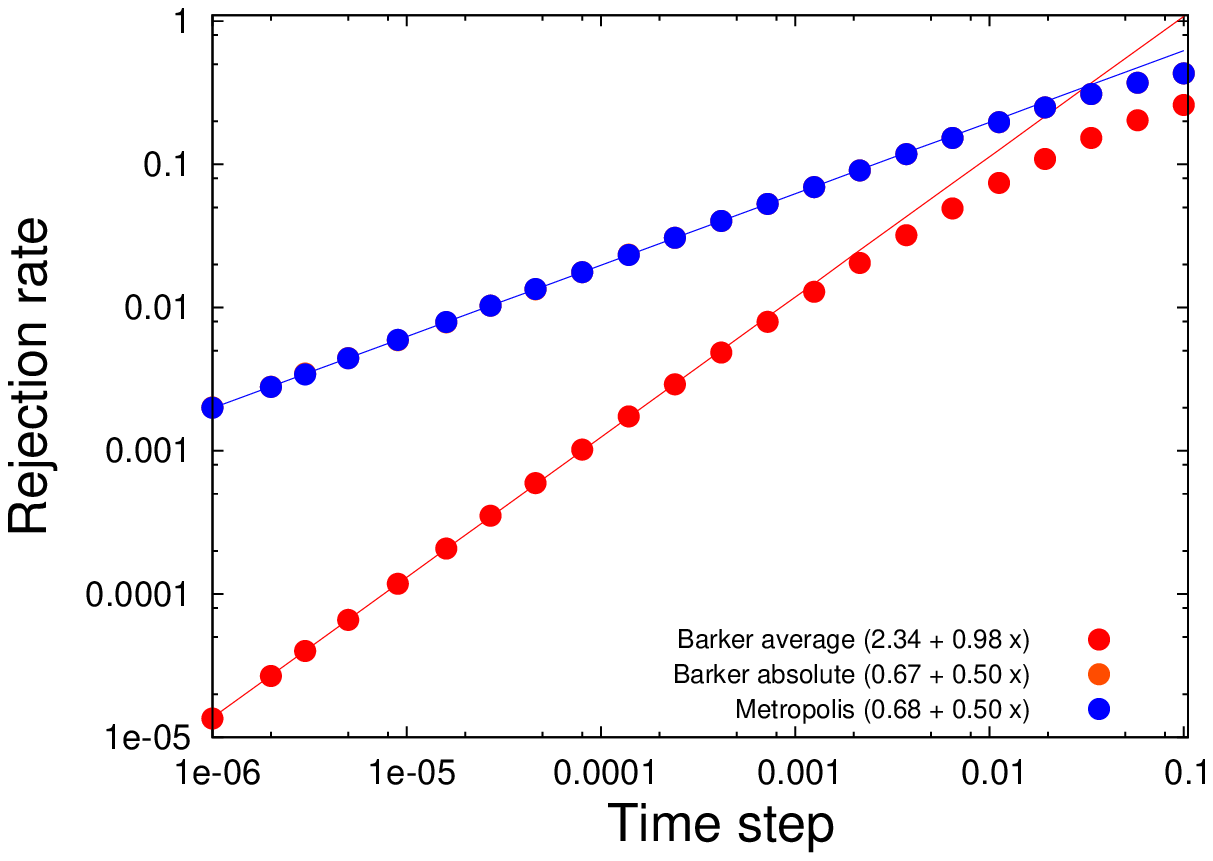}
\end{center}
\caption{\label{fig:rejection} Scaling of the rejection rates. 'Metropolis' corresponds to~\eqref{eq:scaling_Metropolis_rates}, while Barker average' corresponds to the first inequality in~\eqref{eq:scaling_Barker_rates} and 'Barker absolute' to the second one. Left: dynamics with additive noise. Right: dynamics with multiplicative noise. Note that 'Barker absolute' and 'Metropolis' are almost on top of each other. For both plots, we superimpose a linear fit in log-log scale, with $x = \log(\dt)$. }
\end{figure}

Figure~\ref{fig:error_diffusion} presents the estimated self-diffusions as a function of the timestep for the dynamics with additive noise~\eqref{eq:dynamics}. We refer to~\cite[Section~3.1]{FHS14} for the precise definitions of the numerical estimators. Expectations are replaced by empirical averages over $K$ realizations. For Green-Kubo formulas, the sum over the iterations is truncated to some maximal iteration index $\lfloor \tau/\dt \rfloor$, where $\tau$ is a given time. For the Einstein formula, we monitor the estimated mean-square displacement $\EE[(Q^n-Q^0)^2]$ as a function of time up to a maximal iteration index, the slope of this curve (obtained via a least square fit) being equal to twice the self-diffusion constant for the given timestep. The Green-Kubo estimation~\eqref{eq:approx_GK_addtive} was computed using $K=10^9$ realizations and an integration time $\tau = 0.6$, with initial conditions subsampled every 20~steps from a preliminary trajectory computed with $\dt_{\rm thm} = 0.005$. The Einstein estimation~\eqref{eq:prediction_Einstein_additive} was computed using $K = 10^7$ trajectories integrated over $N_{\rm Einstein} = 3 \times 10^5$ iterations, with initial conditions subsampled every 1000~steps from a preliminary trajectory computed with $\dt_{\rm thm} = 0.005$. As can be seen from the various curves in Figure~\ref{fig:error_diffusion}, the values estimated with all methods extrapolate to the analytically computed baseline value as $\dt \to 0$. The error is dramatically reduced by using the Barker rule, the midpoint and HMC schemes performing quite comparably. We also checked that the expected scalings of the errors as a function of~$\dt$ are indeed the ones predicted by our theoretical results. The errors are always larger with a Metropolis rule, though somewhat smaller with the HMC scheme compared to the standard Euler-Maruyama discretization. Also, the Green-Kubo formula leads to more reliable results in this simple case, as already noted in~\cite{FHS14}. In fact, with the most precise method (HMC and Barker rule), there is almost no bias up to timesteps of the order of~$\dt = 0.01$.

\begin{figure}
\begin{center}
\includegraphics[width=8.2cm]{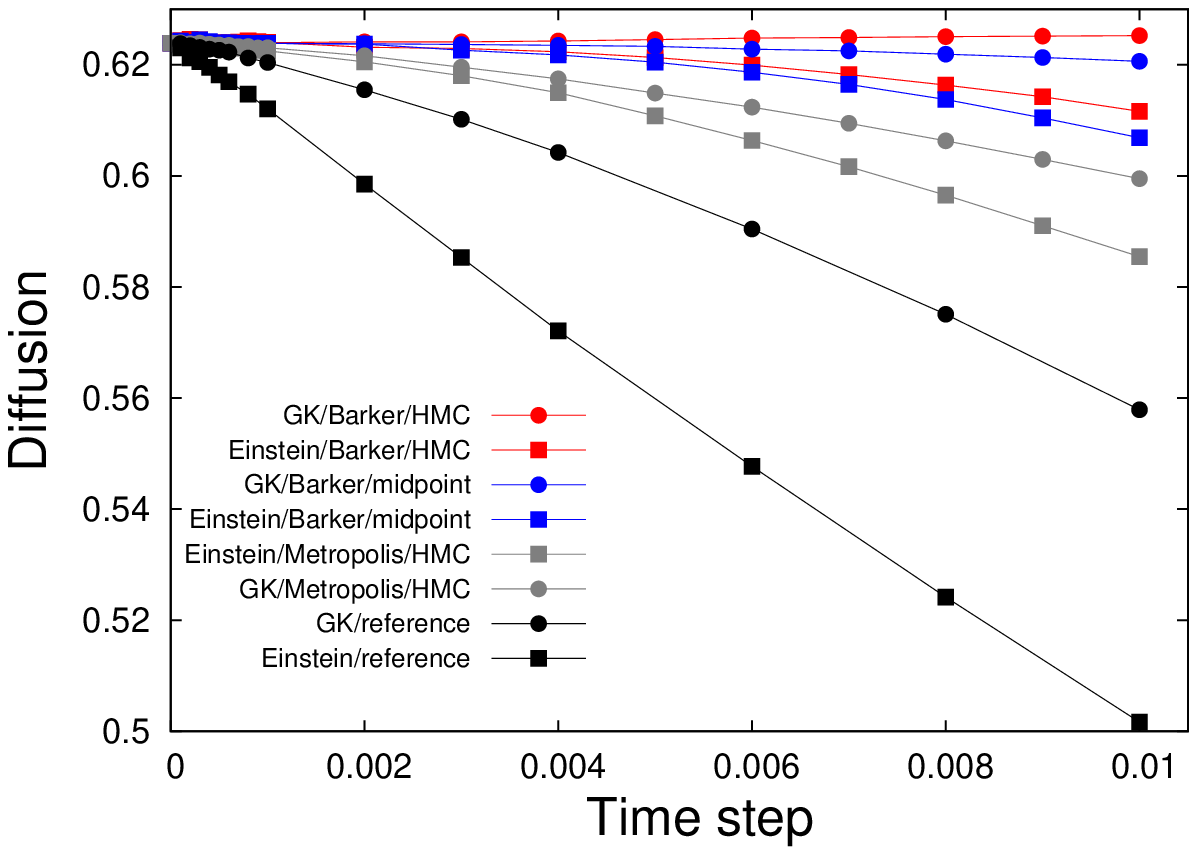}
\includegraphics[width=8.2cm]{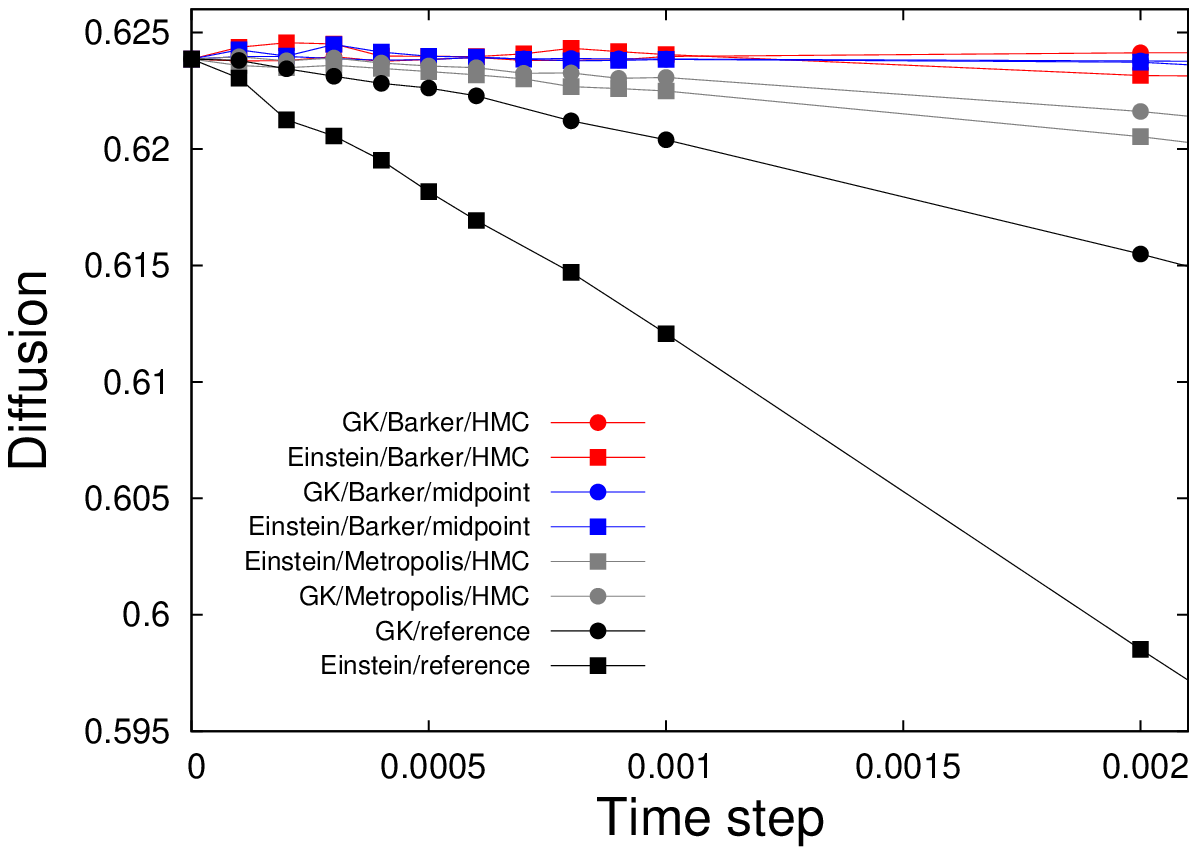}
\end{center}
\caption{\label{fig:error_diffusion} Left: estimated diffusion as a function of the time-step~$\dt$ for the dynamics with additive noise. Right: zoom on the smallest values of~$\dt$. The observed scalings of the error correspond to the ones predicted by the theory (namely, $\dt^2$ when the Barker rule is used with the mipoint or the HMC scheme, $\dt^{3/2}$ when a Metropolis rule is used with the latter two proposals, and $\dt$ in the reference case corresponding to the standard MALA scheme). The baseline value of the diffusion constant is analytically found to be $\mathcal{D} = 0.62386$. }
\end{figure}

Figure~\ref{fig:error_diffusion_multiplicative} presents the estimated self-diffusion constant as a function of the timestep for the dynamics with multiplicative noise~\eqref{eq:dynamics_mult}. The Green-Kubo estimates~\eqref{eq:approx_GK_mult} are obtained with $K=5 \times 10^6$ realizations for $\dt = 10^{-5}$, with a number of realizations increasing proportionally to the timestep up to $K = 10^{9}$. The integration time is $\tau = 2$. The Einstein estimates~\eqref{eq:approx_Einstein_mult} were computed with $K = 10^6$ trajectories integrated over $N_{\rm Einstein} = 10^7$ iterations. The values predicted with all methods extrapolate to the analytically computed baseline value as $\dt \to 0$. The use of a Barker rule instead of a Metropolis rule allows to drastically reduce the bias due to the timestep in the estimation of the self-diffusion constant. Again, the Green-Kubo formula seems more reliable. Finally, note that the error indeed seems to behave as $\sqrt{\dt}$ for very small~$\dt$ when the Metropolis rule is used. In any case, the variations of the estimated self-diffusion are quite sharp around $\dt=0$ with the Metropolis procedure. In contrast, the estimates obtained with the Barker rule are better behaved, which makes it possible to resort to Romberg extrapolation for moderately small values of the timestep.

\begin{figure}
\begin{center}
\includegraphics[width=8.2cm]{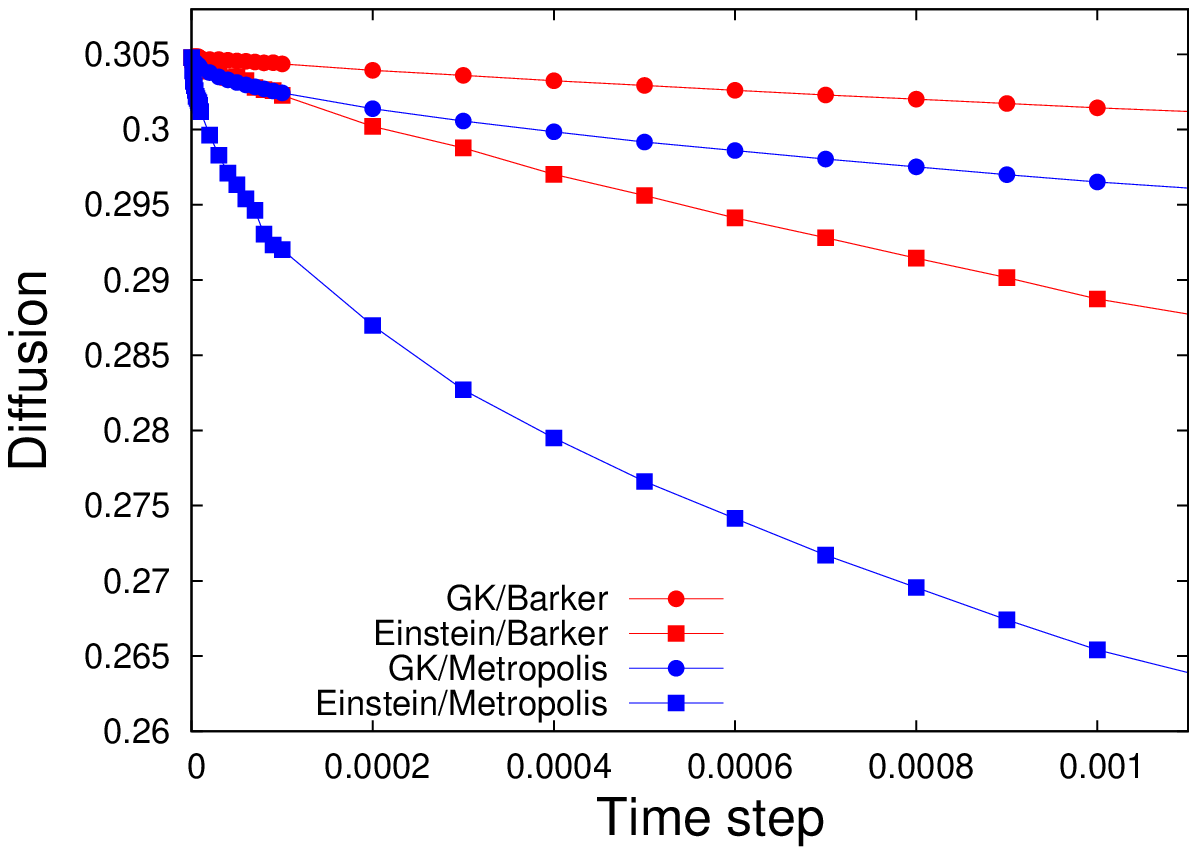}
\includegraphics[width=8.2cm]{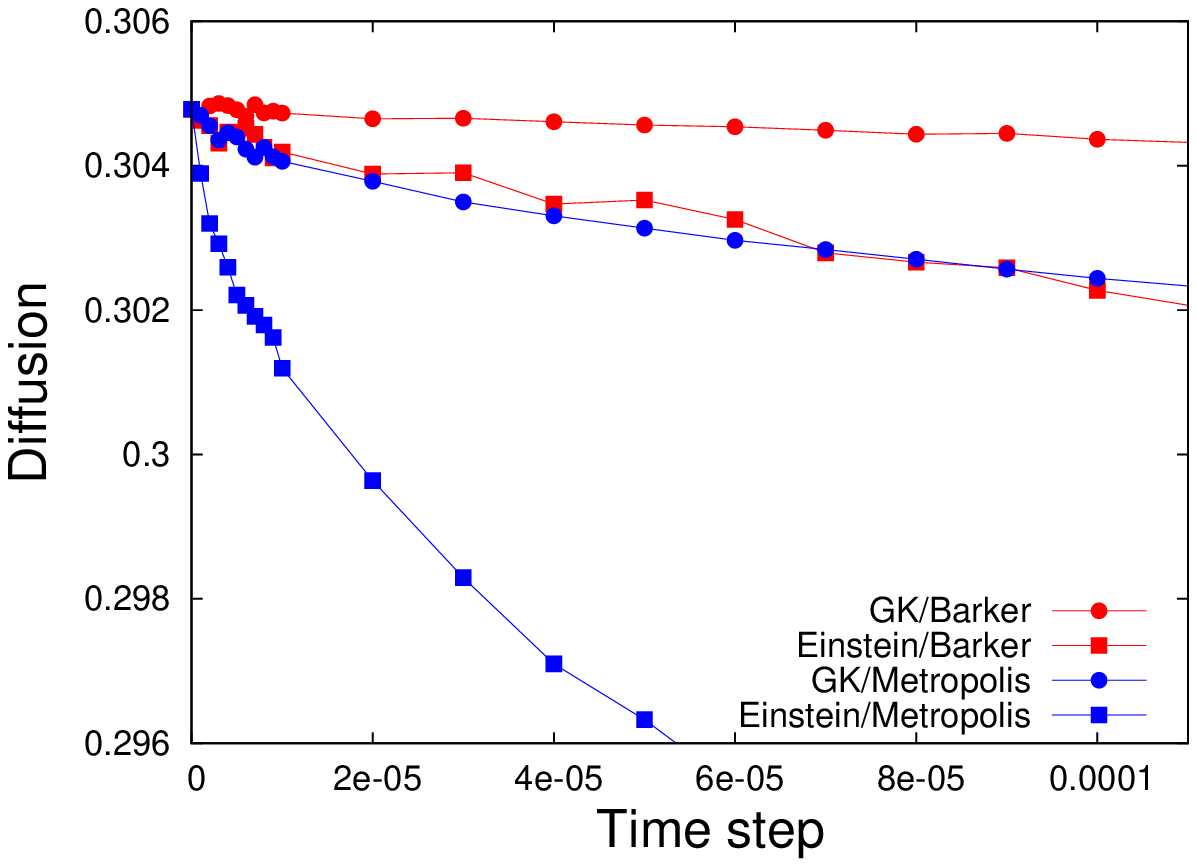}
\end{center}
\caption{\label{fig:error_diffusion_multiplicative} Left: estimated diffusion as a function of the time-step~$\dt$ for the dynamics with multiplicative noise. Right: zoom on the smallest values of~$\dt$. The observed scalings of the error correspond to the ones predicted by the theory (namely, $\dt$ for the Barker rule and $\sqrt{\dt}$ for the Metropolis rule). The baseline value of the diffusion constant is analytically found to be $\mathcal{D} = 0.30478$.}
\end{figure}

\section{Proofs}
\label{sec:proof}

The $n$th order differential of a function~$\psi$ applied to $n$ vectors $v_1,\dots,v_n \in \R^d$ is written 
\[
D^n \psi(q) \cdot (v_1 \otimes \dots \otimes v_n) = \sum_{i_1,\dots,i_n = 1}^d \partial^n_{i_1\dots i_n} \psi(q) v_{1,i_1} \dots v_{n,i_n},
\]
where $v_{k,i}$ is the $i$th component of the vector~$v_k$ and $\partial_i$ denotes the derivative with respect to the $i$-th component of~$q$. Note that $D^n \psi(q) \cdot (v_1 \otimes \dots \otimes v_n) = D^n \psi(q) \cdot (v_{\sigma(1)} \otimes \dots \otimes v_{\sigma(n)})$ for any permutation $\sigma$. When $v_1 = \dots = v_n$, the element $v_1 \otimes \dots \otimes v_n$ is simply denoted by $v_1^{\otimes n}$. These definitions can be extended to matrix-valued functions: for a given element $(v_1,\dots,v_n) \in (\R^d)^n$, the matrix $D^n M(q) \cdot v$ has components $D^n M_{ij}(q) \cdot (v_1 \otimes \dots \otimes v_n)$.

In order to simplify the notation, we write all proofs for $\beta = 1$. The constants $C > 0$ in the inequalities may also vary from one line to another. 


\subsection{Proof of Lemma~\ref{lem:improved_rejection_rate}}
\label{sec:proof:improved_rejection_rate}

In view of~\eqref{eq:Metropolis_ratio_strong}, we can rewrite $A^\mod_\dt(q,q')$ as $\min(\rme^{-\alpha_\dt(q,q')},1)$ with 
\[
\begin{aligned}
\alpha_\dt(q,q') & = V(q')-V(q) - \frac12 \Big[\ln \det\big(\Id + \dt \, \sigma(q')\big)-\ln \det\big(\Id + \dt \, \sigma(q)\big)\Big] \\
& \ \ + \frac{1}{4\dt}\big( q-q'+\dt \nabla V(q')-\dt^2 F(q') \big)^T(\Id + \dt \, \sigma(q')) \big( q-q'+\dt \nabla V(q')-\dt^2 F(q') \big) \\
& \ \ - \frac{1}{4\dt}\big( q'-q+\dt \nabla V(q)-\dt^2 F(q) \big)^T(\Id + \dt \, \sigma(q)) \big( q'-q+\dt \nabla V(q)-\dt^2 F(q) \big).
\end{aligned}
\]
We first perform an expansion of $\alpha_\dt\left(q,\Phi^\mod_\dt(q,G)\right)$ up to terms of order~$\dt^{5/2}$ in order to determine the correction terms $F(q)$ and $\sigma(q)$. The desired conclusion then follows from the inequality 
\begin{equation}
\label{eq:ineq_min_exp}
0 \leq 1 - \min(1,\rme^{-x}) \leq \max(0,x).
\end{equation}

We rewrite $\alpha_\dt(q,q')$ as
\[
\begin{aligned}
\alpha_\dt(q,q') & = V(q') - V(q) - \frac12 (q'-q)^T\left( \nabla V(q) + \nabla V(q') \right) + \frac{\dt}{4}\left(|\nabla V(q')|^2 - |\nabla V(q)|^2 \right) \\
& \ \ - \frac12 \Big[\ln \det\big(\Id + \dt \, \sigma(q')\big)-\ln \det\big(\Id + \dt \, \sigma(q)\big)\Big] + \frac{\dt}{2}(q'-q)^T\left( F(q) + F(q') \right) \\
& \ \ + \frac14 (q'-q)^T \big( \sigma(q')-\sigma(q)\big) (q'-q) - \frac{\dt}{2} (q'-q)^T \big(\sigma(q')\nabla V(q')+\sigma(q)\nabla V(q)\big)\\
& \ \ + \dt^2 \widetilde{\alpha}_\dt(q,q'),
\end{aligned}
\]
where
\[
\begin{aligned}
& \widetilde{\alpha}_\dt(q,q') = -\frac12 \left( F(q')^T \nabla V(q') - F(q)^T \nabla V(q)\right) + \frac{\dt}{4} \left(|F(q')|^2-|F(q)|^2\right) \\
& + \frac{1}{2} (q'-q)^T \big(\sigma(q')F(q')+\sigma(q)F(q)\big) \\
& + \frac14 \left[\big(\nabla V(q')-\dt F(q')\big)^T\sigma(q')\big(\nabla V(q')-\dt F(q')\big) - \big(\nabla V(q)-\dt F(q)\big)^T\sigma(q)\big(\nabla V(q)-\dt F(q)\big)\right].
\end{aligned}
\]
As made precise below in~\eqref{eq:estimate_widetilde_alpha_first}, the remainder $\widetilde{\alpha}_\dt(q,q')$ can be thought of as being of order~$\sqrt{\dt}$ when $q' = \Phi^\mod_\dt(q,G)$. We therefore focus on the first terms in the expression of $\alpha_\dt(q,q')$. A simple computation shows that
\begin{equation}
\label{eq:first_difference_strong}
\begin{aligned}
 V(q') - V(q) - \frac12 (q'-q)^T\left( \nabla V(q) + \nabla V(q') \right) & = -\frac{1}{12} D^3V(q)\cdot (q'-q)^{\otimes 3} -\frac{1}{24} D^4V(q)\cdot (q'-q)^{\otimes 4} \\
& \ - \frac{1}{24}\int_0^1 (1-t)^4 D^5V\big((1-t)q+tq'\big) \cdot(q'-q)^{\otimes 5} \, dt,
\end{aligned}
\end{equation}
so that
\[
\begin{aligned}
& V\left(\Phi^\mod_\dt(q,G)\right) - V(q) - \frac12 \left(\Phi^\mod_\dt(q,G)-q\right)^T\left( \nabla V(q) + \nabla V\left(\Phi^\mod_\dt(q,G)\right) \right) \\
& \qquad = -\frac{\sqrt{2}}{6} \dt^{3/2} \, D^3V(q)\cdot G^{\otimes 3} + \dt^2 \left( \frac12 D^3V(q)\cdot(\nabla V(q)\otimes G^{\otimes 2}) - \frac{1}{6} D^4V(q)\cdot G^{\otimes 4} \right) + \dt^{5/2} \cR_1(q,G),
\end{aligned}
\]
where the remainder $\cR_1(q,G)$ can be thought of as being uniformly bounded (see~\eqref{eq:estimate_widetilde_alpha} below and Remark~\ref{rmk:unbounded_spaces}). In the sequel, we no longer explicitly give the expressions of the integral remainder terms such as the last term on the right-hand side of~\eqref{eq:first_difference_strong} and simply write $\mathrm{O}(|q'-q|^r)$ for some integer power~$r$. With this notation,
\[
\begin{aligned}
|\nabla V(q')|^2 - |\nabla V(q)|^2 & = 2 \nabla V(q)^T \nabla^2 V(q) (q'-q) + (q'-q)^T \left(\nabla^2 V(q)\right)^2 (q'-q) \\
& \ \ + D^3V(q)\cdot \left( (q'-q)^{\otimes 2}\otimes \nabla V(q) \right) + \mathrm{O}\left(|q'-q|^3\right),
\end{aligned}
\]
so that
\[
\begin{aligned}
& \frac{\dt}{4}\left(\left|\nabla V\left(\Phi^\mod_\dt(q,G)\right)\right|^2 - |\nabla V(q)|^2 \right) = \frac{\sqrt{2} \dt^{3/2}}{2}\nabla V(q)^T \nabla^2 V(q) G \\
& \qquad\qquad + \frac{\dt^2}{2} \left( - \nabla V(q)^T \nabla^2 V(q)\nabla V(q) + G^T \left(\nabla^2 V(q)\right)^2 G + D^3V(q)\cdot \left( \nabla V(q) \otimes G^{\otimes 2} \right) \right) + \dt^{5/2} \cR_2(q,G).
\end{aligned}
\]
Here and in the sequel, remainders $\cR_k(q,G)$ satisfy the bounds made precise in~\eqref{eq:estimate_widetilde_alpha}. For the next term to consider in the expression of $\alpha_\dt(q,q')$, we use $\det\big(\Id + \dt \, \sigma(q)\big) = 1 + \dt \, \Tr(\sigma(q)) + \dt^2 r_{\det,\dt}(q)$ where the remainder $r_{\det,\dt}$ is a smooth function of~$q$. This leads to
\[
\ln \det\big(\Id + \dt \, \sigma(q')\big)-\ln \det\big(\Id + \dt \, \sigma(q)\big) = \dt \Tr\left(\sigma(q')-\sigma(q)\right) + \dt^2 \left( r_{\det,\dt}(q')- r_{\det,\dt}(q)\right),
\]
so that
\[
\begin{aligned}
& -\frac12 \left( \ln \det\left[\Id + \dt \, \sigma\left(\Phi^\mod_\dt(q,G)\right)\right]-\ln \det\left[\Id + \dt \, \sigma(q)\right] \right) = -\frac{\sqrt{2}\dt^{3/2}}{2} G^T \nabla \left(\Tr \sigma\right)(q) \\
& \qquad \qquad + \frac{\dt^2}{2} \left( \nabla V(q)^T \nabla \left(\Tr \sigma\right)(q) - \Tr\left[D^2 \sigma(q)\cdot G^{\otimes 2}\right] \right) + \dt^{5/2} \cR_3(q,G).
\end{aligned} 
\]
Moreover,
\[
\frac12 (q'-q)^T\left( F(q) + F(q') \right) = F(q)^T (q'-q) + \frac12 (q'-q)^T \left(DF(q)\cdot (q'-q)\right) + \mathrm{O}\left(|q'-q|^3\right),
\]
so that
\[
\begin{aligned}
& \frac{\dt}{2}\left(\Phi^\mod_\dt(q,G)-q\right)^T\left( F(q) + F\left(\Phi^\mod_\dt(q,G)\right) \right) \\
& \qquad = \sqrt{2}\dt^{3/2} F(q)^T G + \dt^2 \left( G^T \left( DF(q)\cdot G\right) - F(q)^T \nabla V(q)\right) + \dt^{5/2} \cR_4(q,G).
\end{aligned}
\]
In addition,
\[
\begin{aligned}
(q'-q)^T \big( \sigma(q')-\sigma(q)\big) (q'-q) & = (q'-q)^T \left( D\sigma(q)\cdot (q'-q) \right)(q'-q) \\
& \ \ + \frac12 (q'-q)^T \left( D^2\sigma(q)\cdot (q'-q)^{\otimes 2} \right)(q'-q) + \mathrm{O}\left(|q'-q|^5\right),
\end{aligned}
\]
so that
\[
\begin{aligned}
& \frac14 \left(\Phi^\mod_\dt(q,G)-q\right)^T \left( \sigma\left(\Phi^\mod_\dt(q,G)\right)-\sigma(q)\right) \left(\Phi^\mod_\dt(q,G)-q\right) = \frac{\sqrt{2}\dt^{3/2}}{2} G^T \left( D\sigma(q)\cdot G \right)G \\
& \qquad+ \dt^2 \left( - \nabla V(q)^T \left( D\sigma(q)\cdot G \right)G - \frac12 G^T \left( D\sigma(q)\cdot \nabla V(q) \right)G + \frac12 G^T \left( D^2\sigma(q)\cdot G^{\otimes 2} \right) G \right) \\
& \qquad + \dt^{5/2} \cR_5(q,G).
\end{aligned}
\]
The last term to consider is
\[
\begin{aligned}
(q'-q)^T \big(\sigma(q')\nabla V(q')+\sigma(q)\nabla V(q)\big) & = 2 (q'-q)^T \sigma(q)\nabla V(q) + (q'-q)^T \left( D\sigma(q)\cdot (q'-q)\right)\nabla V(q) \\
& \ \ + (q'-q)^T \sigma(q) \nabla^2 V(q) (q'-q) + \mathrm{O}\left(|q'-q|^3\right),
\end{aligned}
\]
so that
\[ 
\begin{aligned}
& - \frac{\dt}{2} \left(\Phi^\mod_\dt(q,G)-q\right)^T \left(\sigma\left(\Phi^\mod_\dt(q,G)\right)\nabla V\left(\Phi^\mod_\dt(q,G)\right)+\sigma(q)\nabla V(q)\right) = -\sqrt{2} \dt^{3/2} G^T \sigma(q)\nabla V(q) \\
& \qquad + \dt^2 \left(\nabla V(q)^T \sigma(q) \nabla V(q) - G^T \left( D\sigma(q)\cdot G\right)\nabla V(q) - G^T \sigma(q)\nabla^2 V(q)G\right) + \dt^{5/2} \cR_6(q,G).
\end{aligned}
\]
In conclusion,
\[
\alpha_\dt\left(q,\Phi^\mod_\dt(q,G)\right) = \dt^{3/2} \overline{\alpha}_{3/2}(q,G) + \dt^2 \overline{\alpha}_{2}(q,G) + \dt^{5/2} \cR(q,G), 
\]
with
\[
\begin{aligned}
\overline{\alpha}_{3/2}(q,G) & = \frac{\sqrt{2}}{2}\left(-\frac{1}{3} D^3V(q)\cdot G^{\otimes 3} + G^T \left( D\sigma(q)\cdot G \right)G \right) \\
& \ \ + \sqrt{2} G^T \left( \frac{1}{2}\nabla^2 V(q) \nabla V(q)^T -\frac{1}{2} \nabla \left(\Tr \sigma\right)(q) + F(q) - \sigma(q)\nabla V(q)\right).
\end{aligned}
\]
The choice $\sigma(q) = \nabla^2 V(q)/3$ allows to cancel the cubic term in~$G$, while 
\[
F(q) = \left(\sigma(q)-\frac12 \nabla^2V(q)\right)\nabla V(q) + \frac12 \nabla \left(\Tr \sigma\right)(q) = \frac16 \left( -\nabla^2V(q) \nabla V(q) + \frac16 \nabla (\Delta V)(q)\right)
\] 
ensures that the linear term in~$G$ disappear. For these choices of $\sigma$ and $F$, it is easy to see that there exist $K > 0$ and an integer~$p$ such that 
\begin{equation}
  \label{eq:estimate_widetilde_alpha_first} 
\left| \widetilde{\alpha}_\dt\left(q,\Phi^\mod_\dt(q,G)\right) \right| \leq K (1+|G|^p) \sqrt{\dt}, 
\end{equation}
and
\begin{equation}
  \label{eq:estimate_widetilde_alpha}
  |\cR_k(q,G)| \leq K(1+|G|^p).
\end{equation}
In fact, it turns out, quite unexpectedly, that the choice of~$F$ and~$\sigma$ also allows to cancel~$\overline{\alpha}_{2}(q,G)$. Indeed,
\[
\begin{aligned}
\overline{\alpha}_{2}(q,G) & = - \frac{1}{6} D^4V(q)\cdot G^{\otimes 4} + \frac12 G^T \left( D^2\sigma(q)\cdot G^{\otimes 2} \right) G \\
& \ \ + D^3V(q)\cdot(\nabla V(q)\otimes G^{\otimes 2}) + \frac12 G^T \left(\nabla^2 V(q)\right)^2 G - \frac12 \Tr\left[D^2 \sigma(q)\cdot G^{\otimes 2}\right] + G^T \left( DF(q)\cdot G\right)  \\
& \ \ - \frac12 G^T \left( D\sigma(q)\cdot \nabla V(q) \right)G - 2 G^T \left( D\sigma(q)\cdot G\right)\nabla V(q) - G^T \sigma(q)\nabla^2 V(q)G \\
& \ \ + \nabla V(q)^T \left( \left[ \sigma(q)- \frac12 \nabla^2 V(q)\right]\nabla V(q) + \frac12 \nabla \left(\Tr \sigma\right)(q) - F(q)\right).
\end{aligned}
\]
A simple computation then shows that $\overline{\alpha}_{2}(q,G) = 0$ (it is clear that the first and last lines in the above expression vanish; the fact that the sum of the second and third lines also vanishes requires a few additional manipulations). 

This finally allows to conclude that $\alpha_\dt\left(q,\Phi^\mod_\dt(q,G)\right) = \dt^{5/2} \cR(q,G)$ with a remainder $\cR(q,G)$ satisfying estimates similar to~\eqref{eq:estimate_widetilde_alpha}, and gives the desired conclusion in view of~\eqref{eq:ineq_min_exp}.

\begin{remark}
\label{rmk:unbounded_spaces}
The above proof can be extended to unbounded spaces by following the strategy presented in the proof of~\cite[Lemma~5.5]{BH13}, assuming some additional bounds on the derivatives of $V$, up to order 5. 
\end{remark}

\subsection{Proof of Theorem~\ref{thm:improved_strong}}
\label{sec:proof:thm:improved_strong}

The proof follows the same lines as the one performed in~\cite[Section~4.2]{BV09} for the usual MALA scheme (trivially adapted to the case of bounded configuration spaces), with the crucial improvement on the rejection rate made precise in Lemma~\ref{lem:improved_rejection_rate}. The main ingredient is an improved estimate on the single-step accuracy of the MALA scheme, which is itself obtained from error estimates for the modified (but un-metropolized) Euler scheme. The estimate on the strong accuracy of the scheme finally follows by a discrete Gronwall argument, where local errors are summed up over finite time-intervals.

\subsubsection*{Single-step accuracy of the un-metropolized scheme} 

We first prove that the modified Euler scheme has the same accuracy as the standard Euler scheme, compare with~\cite[Lemma~4.4]{BV09}. 

\begin{lemma}[Single step accuracy of the modified Euler scheme] 
  \label{1step_em}
  There exists $C > 0$ and $\dt^* > 0$ such that, for any $0 < \dt \leq  \dt^*$ and any $q_0 \in \cM$, 
  \[
  \begin{aligned}
  \mathbb{E}_{q_0}\left[\left|\Phi^\mod_\dt(q_0,W_\dt) - q_{\dt}\right|^2\right] & \leq C\dt^3, \\
  \left|\mathbb{E}_{q_0}\left(\Phi_\dt^\mod(q_0,W_\dt) - q_{\dt}\right) \right|  & \leq C\dt^2,
  \end{aligned}
  \]
  where $q_{\dt}$ is the solution at time $\dt$ of~\eqref{eq:dynamics} with initial condition $q_0$ at $t=0$, and where the expectation is over all realizations of the Brownian motion~$(W_t)_{0 \leq t \leq \dt}$.
\end{lemma}

\begin{proof}
In order to rely on the results already obtained in~\cite{BV09}, we introduce the solution of the standard Euler-Maruyama scheme 
\[
\Phi^\EM_\dt(q,G) = q - \dt\nabla V(q) + \sqrt{2\dt}G. 
\]
Then, in view of the equality $\Phi^\mod_\dt(q,G) = \Phi^\EM_\dt(q,G) + \dt^2 F(q) + \sqrt{2\dt} \left[\left(\Id+\dt \, \sigma(q)\right)^{-1/2} - \Id \right] G$ and by a Cauchy-Schwarz inequality, it holds
\[
\begin{aligned}
\mathbb{E}_{q_0}\left[\left|\Phi^\mod_\dt(q_0,W_\dt) - q_{\dt}\right|^2\right] & \leq 3\mathbb{E}_{q_0}\left[\left|\Phi^\EM_\dt(q_0,W_\dt) - q_{\dt}\right|^2\right] + 6\dt \, \mathbb{E}_{q_0}\left[\left|\left[\left(\Id+\dt \, \sigma(q_0)\right)^{-1/2} - \Id \right]W_\dt \right|^2\right] \\
& \ \ + 3\dt^4 |F(q_0)|^2.
\end{aligned}
\]
It is proved in~\cite[Lemma 4.4]{BV09} that $\mathbb{E}_{q_0}\left[\left|\Phi^\EM_\dt(q,W_\dt) - q_{\dt}\right|^2\right] \leq C_\EM \dt^3$. Moreover,
\[
\mathbb{E}_{q_0}\left[\left|\left[\left(\Id+\dt \, \sigma(q_0)\right)^{-1/2} - \Id \right]W_\dt \right|^2\right] = \dt \Tr\left(\left[\big(\Id+\dt \, \sigma(q_0)\big)^{-1/2} - \Id \right]^2\right) := \dt\, g(\dt).
\]
Denoting by $\lambda_1(\sigma(q))$ the smallest eigenvalue of the real, symmetric matrix $\sigma(q)$, the function $t \mapsto g(t)$ is well defined for $|t| < 3/\kappa$, with $\kappa := \min_{\cM} \lambda_1(\sigma)$ when this quantity is negative and 0 otherwise. Note also that $g(0) = 0$. Therefore, $g(\dt) = \dt g'(\theta_\dt \dt)$ for some $\theta_\dt \in [0,1]$. Now, a simple computation shows that
\[
g'(t) = \Tr\left(\sigma(q_0) \big(\Id+t \, \sigma(q_0)\big)^{-3/2} \left[\big(\Id+t \, \sigma(q_0)\big)^{-1/2} - \Id \right]\right),
\]
Therefore, 
\[
\begin{aligned}
\left|g'(t)\right| & \leq \| \sigma(q_0) \|_{\rm F} \left( \left\| \big(\Id+t \, \sigma(q_0)\big)^{-3/2} \right\|_{\rm F} + \left\| \big(\Id+t \, \sigma(q_0)\big)^{-2} \right\|_{\rm F} \right) \\
& \leq K_d \frac{\| \sigma(q_0) \|_{\rm F}}{\left\| \Id+t \, \sigma(q_0) \right\|_{\rm F}^{3/2}} \left( 1 + \frac{1}{\left\| \Id+t \, \sigma(q_0) \right\|_{\rm F}^{1/2}} \right),
\end{aligned}
\]
for some constant $K_d > 0$ depending on the dimension~$d$. In the above inequalities, we introduced the Frobenius norm $\| M \|_{\rm F} = \sqrt{\Tr(M^T M)}$ and used the fact that $\Id + t \, \sigma$ is real, positive, symmetric, and that all matrix norms are equivalent, to obtain the second inequality by spectral calculus. Since $\Id + t \, \sigma \geq (1-\kappa t)\Id$ for $t \geq 0$, we conclude that 
\[
\left|g'(t)\right| \leq \widetilde{K}_d \frac{\| \sigma(q_0) \|_{\rm F}}{(1-\kappa t)^2},
\]
for some modified constant $\widetilde{K}_d > 0$. We finally deduce that
\begin{equation}
\label{eq:estimate_Brownian_increment_modified}
\mathbb{E}_{q_0}\left[\left|\left[\left(\Id+\dt \, \sigma(q_0)\right)^{-1/2} - \Id \right]W_\dt \right|^2\right] \leq C_1 \dt^2.
\end{equation}
In addition, $|F(q)|^2$ is bounded in view of the expression of~$F$, which gives the first estimate.

For the second estimate, we note that
\[
\begin{aligned}
& \mathbb{E}_{q_0}\left(\Phi_\dt^\mod(q_0,W_\dt) - q_{\dt}\right) = \mathbb{E}_{q_0}\left(\Phi_\dt^\EM(q_0,W_\dt) - q_{\dt}\right) + \dt^2 \, F(q_0). 
\end{aligned}
\]
The conclusion then follows from the estimates provided in~\cite[Lemma~4.4]{BV09}. 
\end{proof}

\subsubsection*{Single-step accuracy of the Metropolized scheme} 

Recall that $\Psi_\dt^\mod(q,G,U)$ is defined in~\eqref{eq:def_Psimod}.

\begin{lemma}[Local accuracy of MALA] \label{loc_ac}
There exists $C > 0$ and $\dt^* > 0$ such that, for any $0 < \dt \leq  \dt^*$ and any $q_0 \in \cM$, 
\[
\begin{aligned}
  \mathbb{E}_{q_0} \left[\left|\Psi^\mod_\dt(q_0,W_\dt,U) - q_{\dt}\right|^2\right] & \leq C\dt^3, \\
  \left|\mathbb{E}_{q_0}\left[\Psi^\mod_\dt(q_0,W_\dt,U) - q_{\dt}\right]\right| & \leq C\dt^2,
\end{aligned}
\]
where the expectation is over all realizations of the Brownian motion~$(W_t)_{0 \leq t \leq \dt}$ and of the random variable $U \sim \mathcal{U}[0,1]$.
\end{lemma}

\begin{proof}
By definition of the scheme, we have
\[
\begin{aligned}
\mathbb{E}_{q_0}\left[\left|\Psi^\mod_\dt(q_0,W_\dt,U) - q_{\dt}\right|^2\right] 
& = \mathbb{E}_{q_0}\left[A^\mod_\dt\left(q_0,\Phi_\dt^\mod(q_0,W_\dt)\right)\left|\Phi^\mod_\dt(q_0,W_\dt) - q_{\dt}\right|^2\right] \\
& \ \ + \mathbb{E}_{q_0}\left[\left(1-A^\mod_\dt\left(q_0,\Phi_\dt^\mod(q_0,W_\dt)\right)\right)\left|q_0 - q_{\dt}\right|^2\right].
\end{aligned}
\]
The first term on the right-hand side is bounded using the estimates provided in Lemma~\ref{1step_em} and the inequality $0 \leq A_\dt^\mod(q,q') \leq 1$ for any $q,q' \in \R^d$. For the second term, a Cauchy-Schwarz inequality gives
\[
\begin{aligned}
& \mathbb{E}_{q_0}\left[\left(1-A^\mod_\dt\left(q_0,\Phi_\dt^\mod(q_0,W_\dt)\right)\right)\left|q_0 - q_{\dt}\right|^2\right] \\
& \qquad \qquad \leq \mathbb{E}_{q_0}\left[\left|q_0 - q_{\dt}\right|^4\right]^{1/2} \mathbb{E}_{q_0}\left[\left(1-A^\mod_\dt\left(q_0,\Phi_\dt^\mod(q_0,W_\dt)\right)\right)^2 \right]^{1/2}.
\end{aligned}
\]
From~\cite[Lemma~4.2]{BV09}, we know that, for any integer $\ell \geq 1$, there exists a constant $K_\ell > 0$ such that
\begin{equation}
\label{eq:estimate_Lemma4.2_BV09}
\mathbb{E}_{q_0}\left[\left|q_0 - q_{\dt}\right|^{2\ell}\right] \leq K_\ell \dt^\ell.
\end{equation}
The conclusion then follows by the above inequality in the case $\ell = 2$ and by using the estimates stated in Lemma~\ref{lem:improved_rejection_rate}. 

The proof of the second estimate is based on similar arguments. First,
\[
\begin{aligned}
\left| \mathbb{E}_{q_0}\left[\Psi^\mod_\dt(q_0,W_\dt,U) - q_{\dt}\right] \right|
& \leq \left| \mathbb{E}_{q_0}\left[A^\mod_\dt\left(q_0,\Phi_\dt^\mod(q_0,W_\dt)\right)\left(\Phi^\mod_\dt(q_0,W_\dt) - q_{\dt}\right)\right] \right| \\
& \ \ + \left| \mathbb{E}_{q_0}\left[\left(1-A^\mod_\dt\left(q_0,\Phi_\dt^\mod(q_0,W_\dt)\right)\right)\left(q_0 - q_{\dt}\right)\right]\right|
\end{aligned}
\]
The second term can be bounded by a Cauchy-Schwarz inequality as
\[
\begin{aligned}
& \left| \mathbb{E}_{q_0}\left[\left(1-A^\mod_\dt\left(q_0,\Phi_\dt^\mod(q_0,W_\dt)\right)\right)\left(q_0 - q_{\dt}\right)\right]\right| \\
& \qquad \leq \mathbb{E}_{q_0}\left[\left(1-A^\mod_\dt\left(q_0,\Phi_\dt^\mod(q_0,W_\dt)\right)\right)^2\right]^{1/2} \mathbb{E}_{q_0}\left[\left|q_0 - q_{\dt}\right|^2\right]^{1/2},
\end{aligned}
\]
together with~\eqref{eq:estimate_Lemma4.2_BV09} in the case $\ell = 1$ and Lemma~\ref{lem:improved_rejection_rate}. To estimate the first term, we write
\[
\begin{aligned}
& \mathbb{E}_{q_0}\left[A^\mod_\dt\left(q_0,\Phi_\dt^\mod(q_0,W_\dt)\right)\left(\Phi^\mod_\dt(q_0,W_\dt) - q_{\dt}\right)\right] = \mathbb{E}_{q_0}\left[\Phi^\mod_\dt(q_0,W_\dt) - q_{\dt}\right] \\
& \qquad \qquad - \mathbb{E}_{q_0}\left[\left(1 - A^\mod_\dt\left(q_0,\Phi_\dt^\mod(q_0,W_\dt)\right)\right)\left(\Phi^\mod_\dt(q_0,W_\dt) - q_{\dt}\right)\right],
\end{aligned}
\]
and use Lemma~\ref{1step_em} to bound the first term on the right-hand side and a Cauchy-Schwarz inequality together with Lemmas~\ref{lem:improved_rejection_rate} and~\ref{1step_em} for the second. 
\end{proof}

\subsubsection*{Global accuracy of the Metropolized scheme}

We now have all the tools we need to prove Theorem~\ref{thm:improved_strong}. For $k = 0,..,\lfloor T/\dt \rfloor$, we introduce $t_k := k\dt$ and
\[
\varepsilon_k := \mathbb{E}_{q_0}\left[\left|q^k - q_{t_k}\right|^2\right],
\]
where we recall that the expectation is over all realizations of the Brownian motion~$(W_t)_{0 \leq t \leq T}$ for a given initial position~$q_0 = q^0$. The Gaussian increments used in the Metropolis scheme are consistent with the realization of the Brownian motion used to integrate the continuous dynamics. More precisely, starting from $q^0 = q_0$, we consider, for $k \geq 1$, 
\[
q^k = \Psi_\dt^\mod\left(q^{k-1},W_{t_{k}}-W_{t_{k-1}},U^k\right).
\]
We claim that there exist $K_1,K_2 > 0$ and $\dt^* > 0$ such that, for any $0 < \dt \leq \dt^*$,
\begin{equation} 
  \label{eq2}
  \varepsilon_{k+1} \leq (1 + K_1\dt)\varepsilon_k + K_2\dt^3.
\end{equation}
Theorem~\ref{thm:improved_strong} then follows by a discrete Gronwall inequality. 

Let us now prove~\eqref{eq2}. We denote by $Q_{t,s}(q)$ be the value at time $t$ of the solution of the SDE~\eqref{eq:dynamics} starting at time $s$ from $q$, which depends on the realization of the underlying Brownian motion. Let $\mathcal{F}_k$ be the sigma-algebra of events up to the time~$t_k$. It holds
\begin{equation}
  \label{eq1}
  \begin{aligned} 
    \left|q^{k+1} - q_{t_{k+1}}\right|^2 & = \left|q^{k+1} - Q_{t_{k+1},t_k}(q^k) + Q_{t_{k+1},t_k}(q^k) - Q_{t_{k+1},t_k}(q_{t_k})\right|^2 \\
    & = \left|q^{k+1} - Q_{t_{k+1},t_k}(q^k)\right|^2 + \left|Q_{t_{k+1},t_k}(q^k) - Q_{t_{k+1},t_k}(q_{t_k})\right|^2 \\
    & \ \ + 2\left( q^{k+1} - Q_{t_{k+1},t_k}(q^k) \right)^T\left( Q_{t_{k+1},t_k}(q^k) - Q_{t_{k+1},t_k}(q_{t_k}) \right). 
  \end{aligned}
\end{equation}
Lemma~\ref{loc_ac} implies that 
\[
\mathbb{E}\left[\left. \left|q^{k+1} - Q_{t_{k+1},t_k}(q^k)\right|^2 \right|\mathcal{F}_k\right] \leq C\dt^3,
\]
so that
\[
\mathbb{E}_{q_0}\left[\left|q^{k+1} - Q_{t_{k+1},t_k}(q^k)\right|^2\right] 
\leq C\dt^3.
\]
Similarly, using~\cite[Lemma~4.3]{BV09}, there exists $\widetilde{K}_1 > 0$ such that 
\[
\mathbb{E}_{q_0} \left[\left|Q_{t_{k+1},t_k}(q^k) - Q_{t_{k+1},t_k}(q_{t_k})\right|^2\right] \leq \left(1 + \widetilde{K}_1\dt\right)\varepsilon_k.
\]
It therefore remains to bound the third term on the right-hand side of~\eqref{eq1}. Setting 
\[
\Delta_k := Q_{t_{k+1},t_k}(q^k) - Q_{t_{k+1},t_k}(q_{t_k}) - \Big( q^k - q_{t_k} \Big),
\]
we can rewrite the term under consideration as
\[
\left( q^{k+1} - Q_{t_{k+1},t_k}(q^k) \right)^T \left( q^k - q_{t_k} \right) + \left( q^{k+1} - Q_{t_{k+1},t_k}(q^k) \right)^T \Delta_k. 
\]
Using a Cauchy-Schwarz inequality and Lemma~\ref{loc_ac}, 
\[
\begin{aligned}
\EE_{q_0}\left[ \left( q^{k+1} - Q_{t_{k+1},t_k}(q^k) \right)^T \left( q^k - q_{t_k} \right) \right] & = \mathbb{E}_{q_0} \left[\left( \mathbb{E}\left[ \left. q^{k+1} - Q_{t_{k+1},t_k}(q^k) \right| \mathcal{F}_k \right]\right)^T \left( q^k - q_{t_k} \right) \right] \\
&\leq \mathbb{E}_{q_0} \left[ \left| \mathbb{E}\left[ \left. q^{k+1} - Q_{t_{k+1},t_k}(q^k)\right|\mathcal{F}_k\right] \right|^2\right]^{1/2}\varepsilon_k^{1/2} \notag \\
&\leq C\dt^{2}\varepsilon_k^{1/2} \leq \frac{C\dt}{2} (\varepsilon_k + \dt^2).
\end{aligned}
\]
Similarly, 
\[
\begin{aligned}
\mathbb{E}_{q_0} \left[ \left( q^{k+1} - Q_{t_{k+1},t_k}(q^k)\right)^T \Delta_k \right] & \leq \mathbb{E}_{q_0}\left( \mathbb{E}\left[ \left. \left| q^{k+1} - Q_{t_{k+1},t_k}(q^k)\right|^2 \right| \mathcal{F}_k \right]\right)^{1/2}\mathbb{E}_{q_0}\left(\mathbb{E}\left[\left. |\Delta_k|^2 \right|\mathcal{F}_k\right]\right)^{1/2} \\
&\leq C\dt^{3/2} \, \mathbb{E}_{q_0}\left(\mathbb{E}\left[\left. |\Delta_k|^2 \right|\mathcal{F}_k\right]\right)^{1/2}. 
\end{aligned}
\]
According to~\cite[Lemma~4.3]{BV09}, 
\[
\mathbb{E}\left[\left. |\Delta_k|^2 \right|\mathcal{F}_k\right] \leq K\dt^2 \left|q^k-q_{t_k}\right|,
\]
so that, by a Cauchy-Schwarz inequality,
\[
\mathbb{E}_{q_0}\left(\mathbb{E}\left[\left. |\Delta_k|^2 \right|\mathcal{F}_k\right]\right) \leq K\dt^2 \, \EE_{q_0}\left|q^k-q_{t_k}\right| \leq  K\dt^2 \, \varepsilon_k^{1/2}.
\]
Therefore, there exists $\widetilde{C} > 0$ such that
\[
\mathbb{E}_{q_0} \left[ \left( q^{k+1} - Q_{t_{k+1},t_k}(q^k)\right)^T \Delta_k \right]
\leq \widetilde{C}\dt^{5/2}\varepsilon_k^{1/4} \leq \frac{3\widetilde{C}}{4}(\dt\,\varepsilon_k + \dt^3),
\]
which concludes the proof of~\eqref{eq2}.

\subsection{Proof of Lemma~\ref{lem:weak_type_expansion}}
\label{sec:proof:lem:weak_type_expansion}

\paragraph{Midpoint proposal with Barker rule.}
We first consider the midpoint proposal~\eqref{eq:proposal_weak_2nd_order} together with a Barker rule. We also write the proof for a general drift $F_\dt = -\nabla V + \dt \, \widetilde{F}$ in order to prove the statements of Remark~\ref{rmk:add_drift}. We start by rewriting~\eqref{eq:def_alpha_Barker} as
\[
\begin{aligned}
\alpha_\dt(q,q') & = V\left(\frac{q+q'}{2} + \frac12 (q'-q)\right) - V\left(\frac{q+q'}{2}- \frac12 (q'-q)\right) + F_\dt\left(\frac{q+q'}{2}\right) \cdot (q'-q) \\
& = \left[ \nabla V \left(\frac{q+q'}{2}\right) + F_\dt\left(\frac{q+q'}{2}\right) \right] \cdot (q'-q) + \frac{1}{24} \left[D^3V\left(\frac{q+q'}{2}\right) \right](q'-q)^{\otimes 3} + \mathrm{O}\left(| q'-q|^5\right) \\
& = \dt\widetilde{F}\left(\frac{q+q'}{2}\right) \cdot (q'-q) + \frac{1}{24} \left[D^3V\left(\frac{q+q'}{2}\right) \right](q'-q)^{\otimes 3} + \mathrm{O}\left(| q'-q|^5\right).
\end{aligned}
\]
Note that the remainder is of order $\mathrm{O}\left(| q'-q|^5\right)$ and not $\mathrm{O}\left(| q'-q|^4\right)$. Next,
\begin{equation}
  \label{eq:expansion_Phi_dt}
  \begin{aligned}
    \Phi_\dt(q,G) = q & + \sqrt{2\dt} \, G - \dt \, \nabla V(q) -\frac{\sqrt2}{2} \, \dt^{3/2} \nabla^2V(q)\cdot G \\
    & + \dt^2 \left(-\frac14 D^3V(q) \cdot G^{\otimes 2} + \frac12 \nabla^2 V(q) \cdot \nabla V(q) + \widetilde{F}(q) \right) + \dt^{5/2} \, Q_\dt(q^n,G^n).
  \end{aligned}
\end{equation}
This leads to
\[
\begin{aligned}
\alpha_\dt\Big(q,\Phi_\dt(q,G)\Big) & = \dt\,\widetilde{F}\left(\frac{q+\Phi_\dt(q,G)}{2}\right) \cdot \left(\sqrt{2\dt} G - \dt\,\nabla V(q) \right) \\
& \ \ + \frac{1}{24} D^3V\left(\frac{q+\Phi_\dt(q,G)}{2}\right) \cdot \left( (2\dt)^{3/2} G^{\otimes 3} - 6 \dt^2 G^{\otimes 2} \otimes \nabla V(q)\right) + \dt^{5/2} \widehat{\alpha}_\dt(q,G) \\
& =  \dt^{3/2} \xi_{3/2}(q,G) + \dt^2 \xi_2(q,G) + \dt^{5/2} \xi_{5/2}(q,G) + \dt^{3} \widetilde{\alpha}_\dt(q,G),
\end{aligned}
\] 
with
\[
\begin{aligned}
\xi_{3/2}(q,G) & = \sqrt{2} \left(\widetilde{F}(q) \cdot G + \frac{1}{12} D^3V(q) \cdot \left(G\right)^{\otimes 3}\right), \\
\xi_2(q,G) & = -\widetilde{F}(q)\cdot \nabla V(q) + D\widetilde{F}(q)\cdot \left(G\right)^{\otimes 2} - \frac14 D^3V(q) \cdot \left(\left(G\right)^{\otimes 2} \otimes \nabla V(q) \right) + \frac{1}{12} D^4V(q) \cdot \left(G\right)^{\otimes 4}, \\
\end{aligned}
\]
and where $\xi_{5/2}(q,G)$ involves only odd powers of~$G$ and satisfies an estimate of the form 
\begin{equation}
  \label{eq:typical_remainder}
  |\xi_{5/2}(q,G)| \leq L (1+|G|^{p})
\end{equation}
for some integer~$p$ and some constant $L>0$. A similar bound is satisfied by the remainder~$\widetilde{\alpha}_\dt(q,G)$. Using
\begin{equation}
  \label{eq:expansion_Barker_ratio}
  \frac{\rme^{-a}}{1+\rme^{-a}} = \frac12 - \frac{a}{4} + \mathrm{O}(a^3),
\end{equation}
it follows that 
\begin{equation}
  \label{eq:scaling_rejection_rate_Barker}
  A_\dt^{\rm Barker}\Big(q,\Phi_\dt(q,G)\Big) = \frac12 - \frac{\dt^{3/2}}{4} \xi_{3/2}(q,G) - \frac{\dt^2}{4} \xi_2(q,G) - \frac{\dt^{5/2}}{4} \xi_{5/2}(q,G) + \dt^3 \widetilde{A}_\dt(q,G),
\end{equation}
where the remainder $\widetilde{A}_\dt(q,G)$ satisfies an estimate of the form~\eqref{eq:typical_remainder} uniformly in~$\dt$.

\begin{remark}[Average rejection rate]
Note that $\EE_G(\xi_{3/2}(q,G)) = \EE_G(\xi_{5/2}(q,G)) = 0$, while 
\[
\overline{\xi}_2(q) = \EE_G\left[\xi_{2}(q,G)\right] = -\widetilde{F}(q)\cdot \nabla V(q) + \mathrm{div}\left(\widetilde{F}\right)(q) - \frac14 \nabla (\Delta V) \cdot \nabla V(q) + \frac{1}{4} \Delta^2 V(q)
\]
is such that 
\[
\int_\cM \overline{\xi}_2 \, d\mu = 0,
\]
since
\begin{equation}
\label{eq:identity_Delta2}
\int_\cM \Delta^2 V \, d\mu = \int_\cM \nabla (\Delta V) \cdot \nabla V \, d\mu.
\end{equation}
Therefore, the average acceptance rate at equilibrium is 
\begin{equation}
\label{eq:avg_rejection_rate_Barker}
\EE_\mu \EE_G \left[A_\dt^{\rm Barker}\Big(q,\Phi_\dt(q,G)\Big)\right] = \frac12 + \mathrm{O}\left(\dt^3\right),
\end{equation}
where the expectation is over $q \sim \mu$ and all realizations of the Gaussian random variable~$G$. 
\end{remark}

On the other hand, the Taylor expansion
\[
\psi\left(\Phi_\dt(q,G)\right) - \psi(q) = \nabla \psi(q) \cdot \left(\Phi_\dt(q,G)-q\right) + \frac12 D^2 \psi(q) \cdot\left(\Phi_\dt(q,G)-q\right)^{\otimes 2} + \frac16 D^3 \psi(q) \cdot \left(\Phi_\dt(q,G)-q\right)^{\otimes 3} + \dots
\]
leads to
\[
\begin{aligned}
& \psi\left(\Phi_\dt(q,G)\right) - \psi(q) 
= \nabla \psi(q) \cdot \left(\sqrt{2\dt} G - \dt \, \nabla V\left(\frac{q+\Phi_\dt(q,G)}{2}\right) + \dt^2 \widetilde{F}(q) \right) \\
& \quad + \frac12 \, D^2\psi(q) \cdot \left(2\dt\,G^{\otimes 2} - 2^{3/2} \dt^{3/2} G \otimes \nabla V\left(\frac{q+\Phi_\dt(q,G)}{2}\right) + \dt^2 \nabla V(q) \otimes \nabla V(q) \right)\\
& \quad + \frac16 \, D^3\psi(q) \cdot \left( 2^{3/2} \dt^{3/2} G^{\otimes 3} - 6 \dt^2 \nabla V(q) \otimes G^{\otimes 2}\right) \\
& \quad + \frac{\dt^2}{6} \, D^4\psi(q) \cdot G^{\otimes 4} + \dt^{5/2} \widetilde{\Psi}_{5/2}(q,G) + \dt^{3} \widehat{\Psi}_\dt(q,G),
\end{aligned}
\]
with $\widetilde{\Psi}_{5/2}(q,G)$ involving only odd powers of~$G$. Upon further expanding 
\[
\begin{aligned}
& \nabla V\left(\frac{q+\Phi_\dt(q,G)}{2}\right) = \nabla V\left(q\right) + \frac12 \, D^2 V(q) \cdot\left(\Phi_\dt(q,G)-q\right) + \frac18 \, D^3 V(q) \cdot \left(\Phi_\dt(q,G)-q\right)^{\otimes 2} + \dots \\
& \quad = \nabla V\left(q\right) + \frac12 \, D^2 V(q) \cdot \left(\sqrt{2\dt} \, G - \dt \, \nabla V(q) \right) + \frac{\dt}{4} \, D^3V(q) \cdot G^{\otimes 2} + \dt^{3/2} \widetilde{V}_\dt(q,G),
\end{aligned}
\]
it follows
\begin{equation}
\label{eq:difference_psi_implicit}
\begin{aligned}
& \psi\left(\Phi_\dt(q,G)\right) - \psi(q) = \sqrt{2\dt} \, \nabla \psi(q) \cdot G \\
& \quad + \dt \left(-\nabla V(q) \cdot \nabla \psi(q) + D^2\psi(q) \cdot G^{\otimes 2} \right) \\
& \quad + \dt^{3/2} \left(-\frac{\sqrt{2}}{2} D^2V(q) \cdot(G \otimes \nabla \psi(q)) - \sqrt{2}D^2\psi(q) \cdot (G \otimes \nabla V(q)) + \frac{\sqrt{2}}{3} \, D^3\psi(q) \cdot G^{\otimes 3} \right)\\
& \quad + \dt^{2} T_{4,\psi}(q,G) + \dt^{5/2} \Psi_{5/2}(q,G) + \dt^{3} \widetilde{\Psi}_\dt(q,G),
\end{aligned}
\end{equation}
with $\Psi_{5/2}(q,G)$ involving only odd powers of~$G$ and
\[
\begin{aligned}
T_{4,\psi}(q,G) & = \widetilde{F}(q) \cdot \nabla \psi(q) +\frac12 D^2 V(q) \cdot \left(\nabla V(q)\otimes \nabla \psi(q)\right)- \frac14 D^3V(q) \cdot \left(\nabla \psi(q) \otimes G^{\otimes 2}\right) \\
& \ \  - D^2\psi(q) \cdot \left(G \otimes \left(D^2V(q)\cdot G\right) \right) + \frac12 \, D^2\psi(q)\cdot \left(\nabla V(q) \otimes \nabla V(q)\right) \\
& \ \ - D^3\psi(q) \cdot \left( \nabla V(q) \otimes G^{\otimes 2} \right) + \frac{1}{6} \, D^4\psi(q) \cdot G^{\otimes 4}.
\end{aligned}
\]
A simple computation shows that $\overline{T_{4,\psi}}(q) = \EE_G[T_{4,\psi}(q,G)]$ is equal to
\[
\overline{T_{4,\psi}} = \widetilde{F} \cdot \nabla \psi + \frac12 (\nabla V)^T \nabla^2 V (\nabla \psi) - \frac14 \nabla (\Delta V)\cdot \nabla \psi - \nabla^2 V : \nabla^2 \psi  + \frac12 (\nabla V)^T \nabla^2 \psi (\nabla V) - \nabla V \cdot \nabla (\Delta \psi) + \frac12 \Delta^2 \psi.
\]
Recalling $\cL = -\nabla V\cdot \nabla + \Delta$ and using the expression of $\cL^2$ provided in~\cite[Section~4.9]{LMS13},
\[
\overline{T_{4,\psi}} = \frac12 \cL^2\psi + \left(\widetilde{F} + \frac14 \nabla(\Delta V) \right ) \cdot\nabla \psi.
\]
Therefore, irrespectively of the choice of $\widetilde{F}$, and in view of~\eqref{eq:scaling_rejection_rate_Barker} and the definition of the discrete generator~\eqref{eq:def_discrete_generator},
\[
\begin{aligned}
\frac{P_\dt-\Id}{\dt}\psi(q) & = \frac12 \EE_G\left[\frac{\psi(\Phi_\dt(q,G))-\psi(q)}{\dt}\right] - \frac{\sqrt{2}\dt}{4} \, \EE_G\left[\xi_{3/2}(q,G)\left(\nabla \psi(q)\cdot G\right)\right] + \dt^{2} \widehat{r}_{\psi,\dt}(q) \\
& = \frac12 \left[ \cL \psi(q) + \dt\left(\frac12 \cL^2\psi(q) + \left(\widetilde{F}+\frac14 \nabla(\Delta V)\right)\cdot \nabla \psi(q) \right)\right] \\
& \ \ - \frac{\dt}{2} \left[\left(\widetilde{F} + \frac14 \nabla(\Delta V) \right ) \cdot\nabla \psi\right](q) + \dt^{2} r_{\psi,\dt}(q) \\
& = \frac12 (\cL \psi)(q) + \frac{\dt}{4} (\cL^2\psi)(q) + \dt^{2} r_{\psi,\dt}(q),
\end{aligned}
\]
which gives the claimed result~\eqref{eq:weak_type_expansion} for the midpoint proposal~\eqref{eq:proposal_weak_2nd_order} when a Barker rule is used. 

\paragraph{Midpoint proposal with Metropolis rule.}
By distinguishing the cases $x \leq 0$ and $x \geq 0$,
\[
x_+ - \frac{x_+^2}{2} \leq 1 - \min(1,\rme^{-x}) \leq x_+, \qquad x_+ = \max(0,x).
\]
Therefore, when the proposal~\eqref{eq:proposal_weak_2nd_order} is considered with a modified drift $-\nabla V + \dt \, \widetilde{F}$ in conjunction with a Metropolis-Hastings algorithm, the rejection rate can be expanded as
  \begin{equation}
    \label{eq:rate_Metropolis_midpoint}
  A^{\rm MH}_\dt(q^n,\widetilde{q}^{n+1}) = 1 - \dt^{3/2} \min\left(0,\xi_{3/2}(q^n,G^n)\right) + \dt^{2} \widetilde{A}^{\rm MH}_\dt(q^n,G^n), 
  \end{equation}
  with a remainder satisfying an inequality similar to~\eqref{eq:typical_remainder}. However, the remainder has a non-trivial average with respect to~$G$ as $\dt \to 0$, in contrast to the case where a Barker rule is used. Therefore, with computations similar to the ones of~\cite[Section~5.2]{FHS14},
  \[
  \frac{P_\dt -  \Id}{\dt} \, \psi = \cL \psi + \dt\left(\frac12 \cL^2 + \left(\widetilde{F}+\frac14 \nabla(\Delta V) + f_{3/2} \right)\cdot \nabla \psi \right)+ \dt^{3/2} r_{\psi,\dt},
  \]
  where 
  \[
  f_{3/2}(q) = -\sqrt{2} \int_{\RR^{d}} \min\left(0,\xi_{3/2}(q,g)\right)g \, \frac{\rme^{-g^2/2}}{(2\pi)^{d/2}} \, dg. 
  \]
  Let us insist on the fact that the remainder now is of order~$\dt^{3/2}$ rather than~$\dt^2$ as in the Barker case. As in~\cite[Section~5]{BDV13}, it is possible to obtain a simpler expression of $f_{3/2}$ in view of the symmetry property 
  \begin{equation}
    \label{eq:symmetry_xi_3/2}
    \xi_{3/2}(q,g)g = -\xi_{3/2}(q,-g)g. 
  \end{equation}
  Introducing $\Omega(q) = \{ g \in \RR^d \, | \, \xi_{3/2}(q,g) \geq 0\}$, 
  \[
  \begin{aligned}
    f_{3/2}(q) & = -\sqrt{2} \int_{\Omega(q)} \xi_{3/2}(q,g)\,g^T \nabla \psi \, \frac{\rme^{-g^2/2}}{(2\pi)^{d/2}} \, dg \\
  & = -\frac{\sqrt{2}}{2} \int_{\RR^d} \xi_{3/2}(q,g)\,g^T \nabla \psi \, \frac{\rme^{-g^2/2}}{(2\pi)^{d/2}} \, dg = -\left( \widetilde{F}(q) + \frac14 \nabla (\Delta V)(q) \right).
  \end{aligned}
  \]
  Therefore, the Metropolis algorithm based on the midpoint proposal is of weak order~2, but with a fractional remainder of order $\dt^{5/2}$ instead of $\dt^3$ when a Barker rule is used. 

\begin{remark}
  \label{rmk:other_rules}
  There are other acceptance/rejection rules than~\eqref{eq:acceptance_rate} ensuring that the canonical measure is invariant. We write to this end the Metropolis acceptance rate in \eqref{eq:acceptance_rate} as $A(r) = \min(1,r)$, and the Barker rate as $A(r) = r/(1+r)$. The invariance of the canonical measure is ensured by the fact that $0 \leq A(r) \leq 1$ and $r A(1/r) = A(r)$, see for instance~\cite[Section~2.1.2.2]{LRS10}. More general choices can be considered, such as (see~\cite{Gidas95}) 
  \[
  A_\gamma(r) = \frac{r}{1+r} \left(1 + 2 \left[ \frac12 \min\left(r,\frac1r\right) \right]^\gamma \right), 
  \]
  with $\gamma \geq 1$. The Metropolis rule corresponds to $\gamma = 1$, while the Barker rule is formally recovered for $\gamma = +\infty$. A key point however in our argument is that $A(r)$ is an entire function of $r$, which allows to eliminate terms with fractional powers of the timestep by averaging over the Gaussian increments. This is not possible for acceptance/rejection criteria based on $A_\gamma$ for $\gamma < +\infty$ because of the minimum over $r$ and $1/r$.
\end{remark}

\paragraph{HMC proposal with Barker rule.}
We now set $\widetilde{F} = 0$ since the previous computations show that $\widetilde{F}$ does not change the weak type properties of the algorithm. For the HMC proposal~\eqref{eq:proposal_HMC}, the expansion~\eqref{eq:expansion_Phi_dt} is changed as
\begin{equation}
  \label{eq:expansion_Phi_dt_HMC}
  \begin{aligned}
    \Phi_\dt(q,G) = q & + \sqrt{2\dt} \, G - \dt \, \nabla V(q) -\frac{\sqrt2}{2} \, \dt^{3/2} \nabla^2V(q)\cdot G \\
    & -\frac{\dt^2}{4} D^3V(q) \cdot G^{\otimes 2} + \dt^{5/2} \, Q_\dt(q^n,G^n).
  \end{aligned}
\end{equation}
Note that only the term in $\dt^2$ changes. Therefore, \eqref{eq:difference_psi_implicit} holds upon changing $T_{4,\psi}(q,G)$ to
\[
\begin{aligned}
T_{4,\psi}(q,G) & = - \frac14 D^3V(q) \cdot \left(\nabla \psi(q) \otimes G^{\otimes 2}\right) \\
& \ \  - D^2\psi(q) \cdot \left(G \otimes \left(D^2V(q)\cdot G\right) \right) + \frac12 \, D^2\psi(q)\cdot \left(\nabla V(q) \otimes \nabla V(q)\right) \\
& \ \ - D^3\psi(q) \cdot \left( \nabla V(q) \otimes G^{\otimes 2} \right) + \frac{1}{6} \, D^4\psi(q) \cdot G^{\otimes 4}.
\end{aligned}
\]
The rate $\alpha_\dt(q,\Phi_\dt(q,G))$ defined in~\eqref{eq:def_alpha_Barker_HMC} is computed by replacing $\psi$ by~$V$ in \eqref{eq:difference_psi_implicit} (with the new definition of~$T_{4,V}$):
\[
\begin{aligned}
V(\Phi_\dt(q,G))-V(q) & = \sqrt{2\dt} \, \nabla V(q) \cdot G + \dt \left(-\left|\nabla V(q)\right|^2 + D^2 V(q) \cdot G^{\otimes 2} \right) \\
& \quad + \dt^{3/2} \left(-\frac{3\sqrt{2}}{2} D^2V(q) \cdot(G \otimes \nabla V(q)) + \frac{\sqrt{2}}{3} \, D^3V(q) \cdot G^{\otimes 3} \right)\\
& \quad + \dt^{2} T_{4,V}(q,G) + \dt^{5/2} \mathcal{V}_{5/2}(q,G) + \dt^{3} \widetilde{\mathcal{V}}_\dt(q,G),
\end{aligned}
\]
and expanding 
\[
\nabla V\left(q+\frac{\sqrt{2\dt}}{2}\,G\right) = \nabla V(q) + \frac{\sqrt{2\dt}}{2} \, D^2 V(q)\cdot G + \frac{\dt}{4} D^3 V(q) \cdot G^{\otimes 2} + \frac{\sqrt{2} \dt^{3/2}}{24} D^4 V(q) \cdot G^{\otimes 3} + \dots
\]
so that 
\[
\begin{aligned}
& \frac{1}{2} \left[ \left(G-\sqrt{2\dt} \, \nabla V\left(q+\frac{\sqrt{2\dt}}{2}\,G\right)\right)^2- G^2 \right] \\
& \qquad\qquad = -\sqrt{2\dt}\, G\cdot \nabla V\left(q+\frac{\sqrt{2\dt}}{2}\,G\right) + \dt \left|\nabla V\left(q+\frac{\sqrt{2\dt}}{2}\,G\right)\right|^2 \\
& \qquad\qquad = -\sqrt{2\dt} \, G\cdot\nabla V(q) + \dt \left(\left|\nabla V(q)\right|^2 - D^2 V(q) \cdot G^{\otimes 2} \right) \\
& \qquad\qquad \ \ + \dt^{3/2} \left(-\frac{\sqrt{2}}{4} \, D^3 V(q)\cdot G^{\otimes 3} + \sqrt{2} D^2 V(q)\cdot(G \otimes \nabla V(q)) \right) \\
& \qquad\qquad \ \ + \dt^{2} \left(-\frac{1}{12}D^4 V(q)\cdot G^{\otimes 4} + \frac12 D^3 V(q)\cdot\left(G^{\otimes 2}\otimes \nabla V(q) \right) + \frac12 \left|\nabla^2 V(q)\cdot G\right|^2\right) \\
& \qquad\qquad \ \ + \dt^{5/2} \mathcal{K}_{5/2}(q,G) + \dt^3 \mathcal{K}_\dt(q). 
\end{aligned}
\]
Finally,
\[
\alpha_\dt(q,\Phi_\dt(q,G)) = \dt^{3/2} \xi_{3/2}(q,G) + \dt^2 \xi_2(q,G) + \dt^{5/2} \xi_{5/2}(q,G) + \dt^3 \widetilde{\alpha}_\dt(q,G),
\]
with
\[
\begin{aligned}
\xi_{3/2}(q,G) & = \sqrt{2} \left(\frac{1}{12} D^3V(q) \cdot \left(G\right)^{\otimes 3} - \frac12 D^2 V(q)\cdot (G \otimes \nabla V(q))\right), \\
\xi_2(q,G) & = -\frac{3}{4} D^3V(q) \cdot \left(\nabla V(q) \otimes G^{\otimes 2}\right) + \frac12 \, D^2V(q)\cdot \left(\nabla V(q) \otimes \nabla V(q)\right) + \frac{1}{12} \, D^4 V(q) \cdot G^{\otimes 4} \\
& \ \ - \frac12 \left|\nabla^2 V(q)\cdot G\right|^2,
\end{aligned}
\]
and where $\xi_{5/2}(q,G)$ involves only odd powers of~$G$. In conclusion, the term $D^2 V(q)\cdot (\nabla V(q) \otimes \nabla \psi(q))$, which is absent in the expression of~$T_{4,\psi}$ compared to the corresponding expression for the midpoint proposal, is compensated by an extra term in the expression of~$\xi_{3/2}$. It is then easy to prove that~\eqref{eq:weak_type_expansion} holds.

\begin{remark}[Average rejection rate]
As for the midpoint proposal used with the Barker rule, it is possible to prove that the average rejection rate at equilibrium is $1/2 + \mathrm{O}(\dt^3)$ (see~\eqref{eq:avg_rejection_rate_Barker} above). This computation relies on the fact that  
\[
\overline{\xi}_2(q) = \EE_G\left[\xi_{2}(q,G)\right] = - \frac34 \nabla (\Delta V) \cdot \nabla V(q) + \frac{1}{4} \Delta^2 V(q) + \frac12 \nabla V^T \nabla^2 V\nabla V(q) - \frac12 \Tr\left[(\nabla^2 V(q))^2\right]
\]
has a vanishing average with respect to~$\mu$. To prove the latter statement, we compute
\[
\begin{aligned}
\int_\cM \Tr\left[(\nabla^2 V(q))^2\right] \, d\mu & = \sum_{i,j=1}^d \int_\cM \left(\partial^2_{q_i,q_j} V\right)^2 \, d\mu = -\sum_{i,j=1}^d \int_\cM \partial_{q_i} V \left(\partial^3_{q_i,q_j,q_j} V - \partial^2_{q_i,q_j} V \partial_{q_j}V \right) \, d\mu \\
& = -\int_\cM \nabla (\Delta V) \cdot \nabla V \, d\mu + \int_\cM \nabla V^T \nabla^2 V\nabla V \, d\mu,
\end{aligned}
\]
and use~\eqref{eq:identity_Delta2}.
\end{remark}

\paragraph{HMC proposal with Metropolis rule.} The result is obtained by a straightforward modification of the argument for the midpoint scheme. We therefore omit the proof.

\subsection{Proof of Lemma~\ref{lem:geometric_ergodicity}}
\label{sec:proof:lem:geometric_ergodicity}

\paragraph{Well posedness of the midpoint scheme.}
To prove that the implicit method is well defined, we use a fixed-point argument. For a given $q \in \cM$ and $G \in \RR^d$, we define $Q^0 = q$ and $Q^{k+1} = \mathscr{F}_\dt(Q^k)$ with
\[
\mathscr{F}_\dt(Q) = q - \dt \, \nabla V\left(\frac{q+Q}{2}\right) + \sqrt{2\dt} \, G.
\]
Note that, upon introducing the global Lipschitz constant $L_V$ of $-\nabla V$,
\[
|\mathscr{F}_\dt(Q)-\mathscr{F}_\dt(Q')| = \dt \left|\nabla V\left(\frac{q+Q}{2}\right)-\nabla V\left(\frac{q+Q'}{2}\right)\right| \leq \frac{L_V \dt}{2} |Q-Q'|.
\]
When $\dt < 2/L_V$, the mapping $\mathscr{F}_\dt$ is a contraction, so that the existence and uniqueness of $\widetilde{q}^{n+1}$ is ensured by the Banach fixed point theorem.

\paragraph{Geometric ergodicity.}
We prove the geometric ergodicity of schemes based on the midpoint proposal~\eqref{eq:proposal_weak_2nd_order}, the computations for the HMC proposal being similar. Our aim is to prove that, for a given physical time $T>0$, there exists $\alpha > 0$ and a probability measure~$\nu$ on~$\cM$ such that the following uniform minorization condition holds (see for instance~\cite{BH13,LMS13} for related estimates):
\begin{equation}
  \label{eq:uniform_minorization_Pdt}
  P_\dt^{\lceil T/\dt \rceil}(q,\cdot) \geq \alpha \, \nu,
\end{equation}
where $\lceil x \rceil$ denotes the smallest integer larger than~$x$. The term ``uniform'' refers to estimates which are independent of the timestep~$\dt$. To obtain such estimates, we have to consider evolutions over fixed physical times, which amounts to iterating the elementary evolution $P_\dt$ over $\lceil T/\dt \rceil$ timesteps. By the results of~\cite{HM11} for instance, \eqref{eq:uniform_minorization_Pdt} implies that there exists $\lambda, \widetilde{C} >0$ such that, for any $f \in \widetilde{L}^\infty(\cM)$,
\[
\left\| P_\dt^{\lceil 1/\dt \rceil}f \right\|_{L^\infty} \leq \widetilde{C} \, \rme^{-\lambda} \| f \|_{L^\infty}, 
\]
from which~\eqref{eq:geom_ergod} follows.

The strategy of the proof of~\eqref{eq:uniform_minorization_Pdt} is the following. We denote by $\widetilde{P}_\dt$ the transition kernel associated with the Markov chain where we perform a move according to the proposal function, and always accept it. We first show in Lemma~\ref{lem:minorization_Phi_no_rejection} that $\widetilde{P}_\dt^n$ satisfies a uniform minorization condition when iterated for a number of steps larger than a fraction of~$T/\dt$. We next show that this property is transferred to the scheme $P_\dt$ where acceptance/rejection is performed according to the Barker or Metropolis rules. 

\begin{lemma}[Uniform minorization condition for schemes without rejection]
\label{lem:minorization_Phi_no_rejection}
Fix $T>0$. There exist $\dt^*, \widetilde{\alpha} > 0$ and a probability measure~$\nu$ such that, for any bounded measurable non-negative function $f$, any $0 < \dt \leq \dt^*$ and $q \in \cM$,
\[
\forall n \in \left\{ \left\lfloor \frac{T}{4\dt} \right\rfloor, \dots, \left\lceil \frac{T}{\dt} \right\rceil \right\}, \qquad \left(P_\dt^n f \right)(q) \geq \widetilde{\alpha} \int_\cM f \, d\nu.
\]
\end{lemma}

\begin{proof}
It is sufficient to prove the result for indicator functions of Borel sets~$E \subset \cM$ (see~\cite{Rudin}). Denoting by $\widetilde{q}^{k+1} = \Phi_\dt(\widetilde{q}^k,G^k)$ the iterates of the Markov chain, we therefore aim at proving 
\[
\forall n \in \left\{ \left\lfloor \frac{T}{4\dt} \right\rfloor, \dots, \left\lceil \frac{T}{\dt} \right\rceil \right\}, 
\qquad 
\mathbb{P}\left( \widetilde{q}^n \in E \, \left| \, \widetilde{q}^0 = q \right.\right) \geq \widetilde{\alpha} \, \nu(E),
\]
for a well chosen probability measure~$\nu$ and a constant $\widetilde{\alpha} > 0$. The idea of the proof is to explicitly rewrite $\widetilde{q}^n$ as a perturbation of the reference evolution corresponding to $\nabla V = 0$. Since we consider smooth potentials and the position space is compact, the perturbation can be uniformly controlled. More precisely,
\begin{equation}
  \label{eq:decomposition_qn}
  \widetilde{q}^n = \widetilde{q}^0 + \mathscr{G}^n + \mathscr{F}^n, 
\end{equation}
with
\[
\mathscr{G}^n = \sqrt{2\dt} \sum_{k=0}^{n-1} G^k, 
\qquad 
\mathscr{F}^n = - \dt \sum_{k=0}^{n-1} \nabla V\left(\frac{\widetilde{q}^k+\widetilde{q}^{k+1}}{2}\right).
\]
Note that $|\mathscr{F}^n| \leq \|\nabla V\|_{L^\infty} n\dt \leq \|\nabla V\|_{L^\infty} (T+\dt)$, while $\mathcal{G}^n$ is a Gaussian random variable with 
covariance matrix $2n\dt\, \Id_d$. Therefore,
\begin{equation}
  \label{eq:estimate_proba_qn_in_A}
  \mathbb{P}\left( \widetilde{q}^n \in E \, \left| \, \widetilde{q}^0 = q \right.\right) = \mathbb{P}\left( \mathscr{G}^n \in E-q-\mathscr{F}^n \right) \geq \left(\frac{1}{4\pi n\dt}\right)^{d/2} \int_{E-q-\mathscr{F}^n} \exp\left(-\frac{|g|^2}{4n\dt}\right) \, dg.
\end{equation}
In the latter expression and in the sequel, we consider that the random variable $\mathcal{G}^n$ has values in~$\mathbb{R}^{d}$ rather than~$\cM$ and understand $E - q - \mathscr{F}^n$ as a subset of~$\mathbb{R}^{d}$ rather than~$\cM$. This amounts to neglecting the possible periodic images, and henceforth leads to the second inequality. Now, for $\dt$ sufficiently small, it holds $T/8 \leq n\dt \leq 2T$, so that
\[
\left(\frac{1}{2\pi n\dt}\right)^{d/2} \int_{E-q-\mathcal{F}^n} \exp\left(-\frac{|g|^2}{2n\dt}\right) \, dg 
\geq \left(\frac{1}{4\pi T}\right)^{d/2} \int_{E-q-\mathcal{F}^n} \exp\left(-\frac{4|g|^2}{T}\right) \, dg.
\]
Since the state space is compact, there exists $R > 0$ such that $|q+\mathscr{F}^n| \leq R$ for any $q \in \cM$. We can then consider the probability measure
\[
\nu(E) = Z_R^{-1} \inf_{|Q| \leq R} \int_{E + Q} \exp\left(-\frac{4|g|^2}{T}\right) \, dg,
\]
and the constant
\[
\widetilde{\alpha} = Z_R \left(\frac{1}{4\pi T}\right)^{d/2},
\]
which gives the claimed result. 
\end{proof}

Let us now show how to adapt the proof of Lemma~\ref{lem:minorization_Phi_no_rejection} to the case when the proposals are accepted or rejected according to some rule (Metropolis or Barker). We set $a=1/2$ for the Barker rule and $a=1$ for the Metropolis one. Note first that~\eqref{eq:decomposition_qn} is modified as
\[
q^n = q^0 + \mathscr{G}^n + \mathscr{F}^n,
\]
with
\[
\begin{aligned}
\mathscr{G}^n & = \sqrt{2\dt} \sum_{k=0}^{n-1} \ind_{U^k \leq A_\dt\left(q^k,\Phi_\dt(q^k,G^k)\right)} G^k, \\
\mathscr{F}^n & = - \dt \sum_{k=0}^{n-1} \ind_{U^k \leq A_\dt\left(q^k,\Phi_\dt(q^k,G^k)\right)} \nabla V\left(\frac{\widetilde{q}^k+\Phi_\dt(q^k,G^k)}{2}\right).
\end{aligned}
\]
It still holds $|\mathscr{F}^n| \leq \|\nabla V\|_{L^\infty} (T+\dt)$. To characterize more precisely $\mathscr{G}^n$, we decompose it as $\mathscr{G}^n = \widetilde{\mathscr{G}}^n + \widehat{\mathscr{G}}^n$, where
\[
\widetilde{\mathscr{G}}^n = \sqrt{2\dt} \sum_{k=0}^{n-1} \ind_{U^k \leq a} G^k
\]
and 
\[
\widehat{\mathscr{G}}^n = \sqrt{2\dt} \sum_{k=0}^{n-1} \left(\ind_{U^k \leq A_\dt\left(q^k,\Phi_\dt(q^k,G^k)\right)} -\ind_{U^k \leq a} \right) G^k.
\]
The latter random variable can be thought of as being small. To quantify this statement, we rewrite each term in the sum defining $\widehat{\mathscr{G}}^n$ as some drift plus a martingale increment, independent of the previous increments. More precisely,
\[
\left(\ind_{U^k \leq A_\dt\left(q^k,\Phi_\dt(q^k,G^k)\right)} -\ind_{U^k \leq a} \right) G^k = D(q^k) + M^k,
\]
where $\EE(M^k | \mathcal{F}_k) = 0$ ($\mathcal{F}_k$ denoting the filtration of events up to iteration~$k$), and 
\begin{equation}
  \label{eq:drift_martingale_decomposition_minorization}
  D(q) = \EE_{U,G}\left[ \left(\ind_{U \leq A_\dt\left(q,\Phi_\dt(q,G)\right)} -\ind_{U \leq a} \right) G\right] = \EE_G\left[ \left( A_\dt\left(q,\Phi_\dt(q,G)\right)-a\right) G\right].
\end{equation}
In view of~\eqref{eq:scaling_rejection_rate_Barker} and~\eqref{eq:rate_Metropolis_midpoint}, the drift term is of order~$\dt^{3/2}$: there exists $C>0$ such that 
\begin{equation}
\label{eq:estimate_drift}
|D(q)| \leq C \dt^{3/2}.
\end{equation}
On the other hand, 
\[
\begin{aligned}
\EE\left[\left. \left(M^k\right)^2 \right| \mathcal{F}_k \right] & = \EE_{U,G}\left[ \left(\ind_{U \leq A_\dt\left(q^k,\Phi_\dt(q^k,G)\right)} -\ind_{U^k \leq a} \right)^2 G^2\right] - D(q^k)^2 \\
& = \EE_G\left( \left[ a+A_\dt\left(q^k,\Phi_\dt(q^k,G)\right)-2\min\left(A_\dt\left(q^k,\Phi_\dt(q^k,G)\right),a\right) \right]G^2\right) - D(q^k)^2 \\
& \leq K\dt^{3/2},
\end{aligned}
\]
so that
\[
\EE_{q^0}\left[\left(\sum_{k=0}^{n-1} M^k\right)^2\right] \leq nK\dt^{3/2}.
\]
Therefore, by the Chebyshev inequality,
\[
\PP\left(\left| \widehat{\mathscr{G}}^n - \sqrt{2\dt}\sum_{k=0}^{n-1} D(q^k) \right| \geq \eta \, \sqrt{2Kn\dt^{5/2}} \right) = \PP\left(\left|\sum_{k=0}^{n-1} M^k \right| \geq \eta \, \sqrt{Kn\dt^{3/2}} \right) \leq \frac{1}{\eta^2}.
\]
By considering $\eta = \widetilde{\eta} \, \dt^{-5/4}$, it follows that there exists $C >0$ such that
\begin{equation}
\label{eq:bound_sum_martingale}
\forall n \leq \left\lceil \frac{T}{\dt}\right\rceil, 
\qquad 
\PP\left(\left| \widehat{\mathscr{G}}^n - \sqrt{2\dt}\sum_{k=0}^{n-1} D(q^k) \right| \geq \frac12 \right) \leq C \dt^{5/2}.
\end{equation}
Since, by~\eqref{eq:estimate_drift},
\[
\left| \sqrt{2\dt}\sum_{k=0}^{n-1} D(q^k) \right| \leq \widetilde{C} \dt,  
\]
it finally holds, for $\dt$ sufficiently small, 
\[
\forall n \leq \left\lceil \frac{T}{\dt}\right\rceil, 
\qquad 
\PP\left(\left| \widehat{\mathscr{G}}^n \right| \geq 1 \right) \leq C \dt^{5/2} \leq \frac12.
\]

We next write, as in~\eqref{eq:estimate_proba_qn_in_A},
\[
\begin{aligned}
\PP\left(q^n \in E \, \Big| \, q^0 = q\right) & = \PP\left( \widetilde{\mathscr{G}}^n \in E-q-\mathscr{F}^n-\widehat{\mathscr{G}}^n \right) \\
& \geq \PP\left(\left.  \widetilde{\mathscr{G}}^n \in E-q-\mathscr{F}^n-\widehat{\mathscr{G}}^n \right| \, \left| \widehat{\mathscr{G}}^n \right| \leq 1\right) \PP\left(\left| \widehat{\mathscr{G}}^n \right| \leq 1 \right) \\
& \geq \frac12 \PP\left(\left.  \widetilde{\mathscr{G}}^n \in E-q-\mathscr{F}^n-\widehat{\mathscr{G}}^n \right| \, \left| \widehat{\mathscr{G}}^n \right| \leq 1\right).
\end{aligned}
\]
In view of the bounds on~$\mathscr{F}^n$, there exists $R \in (0,+\infty)$ such that $\left|q+\mathscr{F}^n+\widehat{\mathscr{G}}^n\right| \leq R$ when $\left| \widehat{\mathscr{G}}^n \right| \leq 1$. Therefore,
\begin{equation}
  \label{eq:estimate_qn_in_A_involving_inf}
  \PP\left( q^n \in E \, \Big| \, q^0 = q\right) \geq \frac12 \inf_{|Q| \leq R} \PP\left(\widetilde{\mathscr{G}}^n \in E-Q \right).
\end{equation}
In order to conclude the proof, we need to determine the law of $\widetilde{\mathscr{G}}^n$. When the Metropolis rule is used, $\widetilde{\mathscr{G}}^n$ simply is a Gaussian random variable of mean~0 and covariance matrix~$2n\dt\,\Id_d$, and the desired conclusion follows by the same manipulations as the one performed below~\eqref{eq:estimate_proba_qn_in_A}. The case of the Barker rule requires some additional work. Let us first introduce the random variable $\mathcal{N}_n$ which counts the number of times $U^k \leq a$ for $0 \leq k \leq n-1$. Of course, $\mathcal{N}_n$ is a binomial law of parameters~$1/2$ and~$n$, hence its expectation is $\mathbb{E}(\mathcal{N}_n) = n/2$ while its variance is $\mathrm{Var}(\mathcal{N}_n) = n/4$. Therefore, by the Chebyshev inequality,
\[
\PP\left( \left|\mathcal{N}_n -\frac{n}{2}\right| \geq \eta \, \frac{\sqrt{n}}{2} \right) \leq \frac{1}{\eta^2},
\]
so that 
\[
\PP\left (\mathcal{N}_n \geq \frac{n}{4} \right) \geq 1-\frac{4}{n}.
\]
On the other hand, conditionally to $\mathcal{N}_n = m$, the random variable $\widetilde{\mathscr{G}}^n$ is a Gaussian random variable of mean~0 and covariance matrix~$2m\dt\,\Id_d$. Therefore, for a given set $\widetilde{E} \subset \cM$,
\[
\begin{aligned}
\PP\left(\widetilde{\mathscr{G}}^n \in \widetilde{E} \right) & \geq \PP\left(\left. \widetilde{\mathscr{G}}^n \in \widetilde{E} \, \right| \,\mathcal{N}_n \geq \frac{n}{4} \right) \PP\left (\mathcal{N}_n \geq \frac{n}{4} \right) \\
& \geq \left( 1-\frac{4}{n} \right) \left(\frac{1}{4\pi n\dt}\right)^{d/2} \int_{\widetilde{E}} \exp\left(-\frac{|g|^2}{n\dt}\right) \, dg.
\end{aligned}
\]
Together with~\eqref{eq:estimate_qn_in_A_involving_inf}, this allows to conclude, as at the end of the proof of Lemma~\ref{lem:minorization_Phi_no_rejection}, that~\eqref{eq:uniform_minorization_Pdt} holds.

\begin{remark}[Extension to the case of dynamics with multiplicative noise]
  \label{rmk:extension_ergo_mult}
  To extend the above proof to discretization of dynamics such as~\eqref{eq:dynamics_mult}, the key point is to appropriately
  bound $\widehat{\mathscr{G}}^n$ since the rejections are encoded in this random variable. To this end, we note that the average 
  drift~\eqref{eq:drift_martingale_decomposition_minorization}, which seems to be of order~$\sqrt{\dt}$, in fact is of order~$\dt$ 
  in view of~\eqref{eq:rate_Metropolis_multiplicative}-\eqref{eq:rate_Barker_multiplicative} and Lemma~\ref{lem:avg_drit_0_mult};
  while a bound similar to~\eqref{eq:bound_sum_martingale} holds with $C\dt^{3/2}$ on the right-hand side since the variance of the martingale increments is of order~$\sqrt{\dt}$ rather than~$\dt^{3/2}$.
\end{remark}

\subsection{Proof of Theorem~\ref{thm:improved_GK}}
\label{sec:proof_thm:improved_GK}

We follow the strategy of~\cite[Section~3.8]{LMS13} (as already done in~\cite[Section~5.4]{FHS14}) and write an approximation of $\cL^{-1}$ using the discrete evolution operator~$P_\dt$. We write the proof in the general case when
\begin{equation}
  \label{eq:general_P_dt}
  \frac{P_\dt-\Id}{\dt} \psi = a_1 \cL \psi + a_2 \dt \cL^2 \psi + \dt^{\alpha} \, r_{\psi,\dt},
\end{equation}
for $a_1 > 0$, $a_2 \geq 0$ and $1 < \alpha \leq 2$, and $r_{\psi,\dt}$ is uniformly bounded for $\dt$ sufficiently small. The cases of interest are given by~\eqref{eq:weak_type_expansion}. In particular, $a_2 = a_1/2$ in all cases. Note that (with equalities in $\widetilde{L}^\infty(\cM)$),
\[
\begin{aligned}
\left(-\mathcal{L}\right)^{-1}\psi &= \left(\dt \sum_{n=0}^{+\infty} P_\dt^n \right)\left(\frac{\Id - P_\dt}{\dt}\right)\left(-\mathcal{L}^{-1}\right)\psi \\
&= \left(\dt \sum_{n=0}^{+\infty} P_\dt^n \right) \left( \left(a_1 + a_2\dt \mathcal{L}\right)\psi + \dt^{\alpha} r_{\cL^{-1}\psi,\dt} \right). 
\end{aligned}
\]
Since $\cL^{-1}\psi$ still is a smooth function (by elliptic regularity), the remainder $r_{\cL^{-1}\psi,\dt}$ is uniformly bounded in $L^\infty(\cM)$ by Lemma~\ref{lem:weak_type_expansion}. Note also that since $(\Id-P_\dt)\cL^{-1}\psi$ and $\cL\psi$ have vanishing averages with respect to~$\mu$, the remainder $r_{\cL^{-1}\psi,\dt}$ has a vanishing average with respect to~$\mu$. Moreover, in view of~\eqref{eq:general_P_dt},
\[
\cL \psi = -\frac{1}{a_1} \frac{\Id - P_\dt}{\dt} \psi + \dt\, \widehat{r}_{\psi,\dt},
\]
with a remainder uniformly bounded in~$\dt$ and with vanishing average with respect to~$\mu$,
so that
\[
\left(\dt \sum_{n=0}^{+\infty} P_\dt^n \right) \cL \psi = \left(\frac{\Id - P_\dt}{\dt}\right)^{-1}\cL \psi = -\frac{1}{a_1}\psi + \dt \, \left(\frac{\Id - P_\dt}{\dt}\right)^{-1}\widehat{r}_{\psi,\dt}.
\]
The above equalities show that
\[
\begin{aligned}
\int_\cM \left(-\mathcal{L}^{-1}\psi\right)\varphi \, d\mu & = a_1 \dt \sum_{n=0}^{+\infty} \int_\cM \left( P_\dt^n \psi\right) \varphi \, d\mu -\frac{a_2}{a_1} \dt \int_\cM \psi\, \varphi \, d\mu \\
& \ \ + \dt^\alpha \int_\cM \left[\left(\frac{\Id - P_\dt}{\dt}\right)^{-1}\left( r_{\cL^{-1}\psi,\dt} + a_2 \dt^{2-\alpha}\widehat{r}_{\psi,\dt}\right)\right] \varphi \, d\mu,
\end{aligned}
\]
where the sum is convergent in view of~\eqref{eq:geom_ergod}. This gives the result, in view of the boundedness of the operator $\dps \left(\frac{\Id - P_\dt}{\dt}\right)^{-1}$ (given by~\eqref{eq:bound_discrete_generator}).

\subsection{Proof of Theorem~\ref{thm:fluctuation}}
\label{sec:proof_improved_Einstein}

The proof follows the lines of the proof of~\cite[Theorem~3]{FHS14}. We write it in the more general case when $P_\dt$ satisfies~\eqref{eq:general_P_dt}. We need preliminary results on the average behavior of the increments. The first result is obtained by considering, for a given function~$f$, the conditional increment $\overline{f}(q) = \mathbb{E}\left(\left. f(q^1-q^0)\, \right| q^0=q\right)$ and using the expansion of $P_\dt$ in powers of~$\dt$. In order to state the result, we introduce the operator
\[
(\cL_Q^i f)(q) = (\cL^i \tau_Q f)(q), \qquad \tau_a f(\delta) = f(\delta-a).
\]
To be more explicit, $(\cL_Q f)(q) = -\nabla V(q)\cdot \nabla f(q-Q) + \Delta f(q-Q)$, so that, in particular, $(\cL_q f)(q) = -\nabla V(q)\cdot \nabla f(0) + \Delta f(0)$. The expression of $\left(\cL_q^2 f\right)(q)$ is similarly obtained by writing the expression of $\cL^2 f(q)$ and replacing the arguments of~$f$ by~0 everywhere.

\begin{lemma}
  \label{lem:expansion_overline}
  Consider two smooth functions $f,g:\RR^d \to \RR$. Then,
  \begin{equation}
    \label{eq:avg_cLq}
  \overline{f}(q) = \mathbb{E}_{G,U}\Big(f\big(\delta_\dt(q,G,U)\big)\Big) = f(0) 
  + a_1 \dt (\cL_q f)(q) + a_2 \dt^2 \left(\cL_q^2 f\right)(q) + \dt^{\alpha+1} r_{f,\dt},
  \end{equation}
  and
  \begin{equation}
    \label{eq:correlation}
    \frac{1}{\dt} \left( \mathbb{E}_{G,U}\Big(f\big(\delta_\dt(q,G,U)\big)g\big(\delta_\dt(q,G,U)\big)\Big) - \overline{f}(q)\overline{g}(q) \right) = 2a_1 \nabla f(0) \cdot \nabla g(0) + \dt C_{f,g}(q) + \dt^\alpha r_{f,g,\dt},
  \end{equation}
  where
  \begin{equation}
    \label{eq:def_C_fg}
    C_{f,g}(q) = a_2 \Big( \left[\cL^2_q(fg)\right](q) - f(0) \left(\cL_q^2 g\right)(q) - \left(\cL_q^2 f\right)(q) g(0) \Big) - a_1^2 (\cL_q f)(q) (\cL_q g)(q).
  \end{equation}
\end{lemma}

The equality~\eqref{eq:correlation} is obtained by an application of~\eqref{eq:avg_cLq} to the function $fg$. The second result on the average behavior of the increments is the following.

\begin{lemma}
  \label{lem:solution_Poisson}
  Set $\alpha = 3/2$, $\gamma = 2$ and $a=1$ for schemes based on the Metropolis rule, while $\alpha = 2$, $\gamma = 3$ and $a=1/2$ for schemes based on the Barker rule. For $\dt$ sufficiently small,
  \begin{equation}
    \label{eq:expansion_increment}
    \overline{\delta}_\dt(q) = \EE_{G,U}[\delta_\dt(q,G,U)] = -\dt \left(a_1 + a_2\dt \, \cL\right)\nabla V(q) + \dt^{\alpha+1} r_{\dt,\delta}(q),
  \end{equation}
  where $r_{\dt,\delta}$ is uniformly bounded in $L^\infty(\cM)$ for~$\dt$ sufficiently small. The function~$\overline{\delta}_\dt$ has average~0 with respect to~$\mu$, and the unique solution in $L^\infty(\cM)$ of the Poisson equation
  \[
  \left(P_\dt-\mathrm{Id}\right) N_{\dt} = \overline{\delta}_\dt, \qquad \int_\cM N_\dt\,d\mu = 0, 
  \]
  can be expanded as 
  \begin{equation}
    \label{eq:expansion_N_dt}
    N_{\dt} = -\mathcal{L}^{-1} \nabla V + \dt^\alpha \widetilde{N}_\alpha + \dt^\gamma \widetilde{N}_{\gamma,\dt},
  \end{equation}
  where $\widetilde{N}_\alpha$ is smooth and $\widetilde{N}_{\gamma,\dt}$ is uniformly bounded in $L^\infty(\cM)$ for~$\dt$ sufficiently small. 
\end{lemma}

\begin{proof}
The expansion~\eqref{eq:expansion_increment} is a direct consequence of Lemma~\ref{lem:expansion_overline} with the choice $f(\delta) = \delta$. The fact that $\overline{\delta}_\dt(q)$ has average~0 can be proved as in the proof of~\cite[Lemma~4]{FHS14}. From~\eqref{eq:general_P_dt},
\[
\frac{P_\dt-\mathrm{Id}}{\dt}\left(\mathcal{L}^{-1} \nabla V \right) = a_1 \nabla V + a_2 \dt \, \cL \nabla V + \dt^{\alpha+1} r_{\nabla V,\dt},
\]
so that, in view of the equation satisfied by $N_\dt$ and the expansion of~$\overline{\delta}_\dt$, 
\[
\left(P_\dt-\mathrm{Id}\right)\left( N_{\dt} + \mathcal{L}^{-1} \nabla V \right) = \dt^{\alpha+1} r_{\nabla V,\dt}.  
\]
In view of~\eqref{eq:bound_discrete_generator}, this shows that $N_{\dt} = -\mathcal{L}^{-1} \nabla V + \dt^\alpha \widetilde{N}_{\alpha,\dt}$. We need at this stage to obtain weak type expansions such as~\eqref{eq:weak_type_expansion} at higher order. More precisely,
\[
\frac{P_\dt-\Id}{\dt} \psi = a \left(\cL \psi + \frac{\dt}{2} \cL^2 \psi\right) + \dt^{\alpha} S\psi + \dt^\gamma \, r_{\psi,\dt},
\]
where $S$ is some differential operator of finite order which preserves~$\mu$, and the remainder is uniformly bounded for $\dt$ sufficiently small. The proof is a slight extension of the proof of Lemma~\ref{lem:weak_type_expansion} performed in Section~\ref{sec:proof:lem:weak_type_expansion} and is therefore omitted. The important point to note is that $\gamma = 3$ when the Barker rule is used since terms with fractional powers of $\dt$ always come with odd powers of~$G$; while in contrast $\gamma=2$ in the Metropolis case since all terms $\dt^{k/2}$ contribute for $k \geq 1$. By the same computations as above,
\[
\left(P_\dt-\mathrm{Id}\right)\left( N_{\dt} + \mathcal{L}^{-1} \nabla V \right) = \dt^{\alpha+1} S \mathcal{L}^{-1} \nabla V + \dt^\gamma r_{\nabla V,\dt}.
\]
This allows us to identify $\widetilde{N}_\alpha = \cL^{-1} S \mathcal{L}^{-1} \nabla V$ since 
\[
\left(\frac{P_\dt-\mathrm{Id}}{\dt}\right)^{-1} S \mathcal{L}^{-1} \nabla V = \cL^{-1} S \mathcal{L}^{-1} \nabla V + \dt \, \widehat{r}_{\nabla V,\dt}.
\]
The latter equality can be checked by applying $P_\dt-\mathrm{Id}$ on both sides and using~\eqref{eq:bound_discrete_generator}. 
\end{proof}

We can now turn to the proof of Theorem~\ref{thm:fluctuation}.

\begin{proof}[Proof of Theorem~\ref{thm:fluctuation}]
We rewrite the increment $\delta_\dt(q^m,G^m,U^m)$ as the sum of a discrete martingale $\delta_\dt(q^m,G^m,U^m) - \overline{\delta}_\dt(q^m)$ and the average increment~$\overline{\delta}_\dt(q^m)$. We also use Lemma~\ref{lem:solution_Poisson} to rewrite~$\overline{\delta}_\dt(q^m)$ as
\[
\overline{\delta}_\dt(q^m) = P_\dt N_\dt(q^m) - N_\dt(q^m) = \Big( P_\dt N_\dt(q^m) - N_\dt(q^{m+1}) \Big) + \Big( N_\dt(q^{m+1})-N_\dt(q^m) \Big).
\] 
Therefore,
\[
\sum_{m=0}^{n-1} \delta_\dt(q^m,G^m,U^m) = N_{\dt}(q^n) - N_{\dt}(q^0) + \sum_{m=0}^{n-1} M_{\dt}^m,
\]
with
\[
M_{\dt}^m = \delta_\dt(q^m,G^m,U^m) - N_{\dt}(q^m + \delta_\dt(q^m,G^m,U^m)) - \overline{\delta}_\dt(q^m) + \overline{N}_{\dt}(q^m),
\]
where $\overline{N}_{\dt}(q) = (P_\dt N_{\dt})(q) = \EE_{G,U}\left[ N_{\dt}(q + \delta_\dt(q,G,U)) \right]$. Note that $(M_{\dt}^m)_{m \geq 0}$ are stationary, independent martingale increments when $q^0 \sim \mu$. In view of Lemma~\ref{lem:solution_Poisson}, $N_{\dt} \in L^\infty(\cM)$, so that, for a given $\xi \in \RR^d$, it holds in view of~\eqref{eq:def_Qn} and~\eqref{eq:def_D_Einstein_dt},
\begin{equation}
\label{eq:expansion_correlation_Einstein}
\xi^T \mathscr{D}_\dt \xi = \frac{1}{\dt} \mathbb{E}\left[\left(\xi^T M_{\dt}^0\right)^2\right] = \frac{1}{\dt} \int_\cM \left(\mathbb{E}_{G,U}\left(\big[B_{\dt,\xi}(q,\delta_\dt(q,G,U))\big]^2\right)-\left(\overline{B}_{\dt,\xi}(q)\right)^2\right) \mu(dq),
\end{equation}
with $B_{\dt,\xi}(q,\delta) = \xi^T \left(\delta - N_\dt(q+\delta)\right)$ and where the expectation in the first equality is with respect to $G,U$ and $q^0 \sim \mu$. We now use~\eqref{eq:correlation} to compute the right-hand side. Note first that, in view of~\eqref{eq:expansion_N_dt} (setting $f(\delta) = g(\delta) = B_{\dt,\xi}(q,\delta)$, the first argument~$q$ in $B_{\dt,\xi}$ being a parameter),
\[
\begin{aligned}
\xi^T \mathscr{D}_\dt \xi & = \frac{1}{\dt} \int_\cM \left(\mathbb{E}_{G,U}\left(\big[B_{0,\xi}(q,\delta_\dt(q,G,U))\big]^2\right)-\left(\overline{B}_{0,\xi}(q)\right)^2\right) \mu(dq) + \dt^\alpha \xi^T \widetilde{\mathscr{D}}_\dt \xi \\
& = 2a_1 \int_\cM \left( |\xi|^2 - 2 \nabla (\xi^T N_0)\cdot \xi + \left|\nabla\left(\xi^T N_0\right)\right|^2\right) d\mu+ \dt \int_\cM C_{B_{0,\xi},B_{0,\xi}}\, d\mu + \dt^\alpha \xi^T \widehat{\mathscr{D}}_\dt \xi,
\end{aligned}
\]
where $C_{f,g}$ is defined in~\eqref{eq:def_C_fg}. Next, introducing $h_0(\delta) = B_{0,\xi}(q,\delta) = \xi^T \delta - \cL^{-1} \left(\xi^T\nabla V\right)(q+\delta)$ (where again $q$ is a parameter), a simple computation shows that $\nabla h_0(\delta) = \xi - \left[ \nabla \cL^{-1}(\xi^T \nabla V) \right](q+\delta)$ and $\Delta h_0(\delta) = - \left[ \Delta \cL^{-1}(\xi^T \nabla V) \right](q+\delta)$. Therefore, $\left(\cL_Q h_0\right)(q) = -\xi^T \nabla V(q) + \nabla V(q)^T \nabla (\xi^T N_0)(2q-Q) + \Delta (\xi^T N_0)(2q-Q)$ so that $\left(\cL_q h_0\right)(q) = 0$ in view of the definition of~$N_0$. Using the identity
\[
\cL^2(\varphi \psi) = 2 (\cL \varphi) (\cL \psi) + \varphi \cL^2 \psi + (\cL^2 \varphi)\psi + 2 \nabla \varphi \cdot \nabla (\cL \psi) + 2 \nabla (\cL \varphi) \cdot \nabla \psi + 2 \cL\left(\nabla \varphi \cdot \nabla \psi \right),
\]
obtained by iterating $\cL(\varphi \psi) = \varphi \cL \psi + (\cL \varphi)\psi + 2 \nabla \varphi \cdot \nabla \psi$, we also compute
\[
C_{B_{0,\xi},B_{0,\xi}}(q) = 2\cL\left(\left[\xi-\nabla(\xi^TN_0)\right]^2 \right),
\]
which has average~0 with respect to~$\mu$. In conclusion,
\[
\xi^T \mathscr{D}_\dt \xi = 2a_1 \int_\cM \left( |\xi|^2 - 2 \nabla (\xi^T N_0)\cdot \xi + \left|\nabla\left(\xi^T N_0\right)\right|^2\right) d\mu + \dt^\alpha \xi^T \widehat{\mathscr{D}}_\dt \xi.
\]
The result is finally obtained by manipulations similar to the ones used to establish~\cite[Eq.~(32)]{FHS14}.
\end{proof}

\subsection{Proof of Lemma~\ref{lem:evolution_operator_multiplicative_noise}}
\label{sec:proof:lem:evolution_operator_multiplicative_noise}

The result crucially relies on the expansion in powers of~$\dt$ of $\alpha_\dt(q,\Phi_\dt(q,G))$ (defined in~\eqref{eq:alpha_dt_multiplicative}), where $\Phi_\dt(q^n,G^n)$ encodes the proposal~\eqref{eq:EM_multiplicative}:
\[
\Phi_\dt(q,G) = q + \dt F(q) + \sqrt{2\dt} B(q) G.
\]
For notational convenience, we introduced the symmetric, definite, positive matrix $B(q) = M^{1/2}(q)$. We also write remainder terms as $\mathrm{O}(\dt^\delta)$. The equality $c(q,G) = \mathrm{O}(\dt^\delta)$ should be understood as: there exists $K_c > 0$ and $p_c \in \mathbb{N}$ such that, for all $q \in \cM$ and $G \in \R^d$,
\[
\left|c(q,G)\right| \leq K_c \dt^\delta \left(1+|G|^{p_c}\right).
\]
In particular, $\mathbb{E}_G\left|c(q,G)\right|^r \leq C_r \dt^{\delta r}$ for any $r \geq 0$.

Let us now evaluate the various terms in $\alpha_\dt(q,\Phi_\dt(q,G))$.
First,
\[
V(\Phi_\dt(q,G)) - V(q) = \sqrt{2\dt} \nabla V(q)^T B(q) G + \mathrm{O}(\dt).
\]
Consider next the terms corresponding to $\det M$. Since
\[
M(\Phi_\dt(q,G)) = M(q) + \sqrt{2\dt} DM(q) \cdot \big( B(q) G \big)+ \mathrm{O}(\dt),
\]
it holds 
\[
\begin{aligned}
\det M(\Phi_\dt(q,G)) & = \Big(\det M(q) \Big) \det \left(\Id + \sqrt{2\dt} M^{-1}(q) \left[ DM(q) \cdot \big( B(q) G \big)\right] + \mathrm{O}(\dt)\right) \\
& = \Big( \det M(q) \Big)\left(1 + \sqrt{2\dt} \Tr \Big( M^{-1}(q) \left[ DM(q) \cdot \big( B(q) G \big)\right]\Big) + \mathrm{O}(\dt) \right),
\end{aligned}
\]
so that
\[
\frac{1}{2} \Big(\ln \big(\det M(\Phi_\dt(q,G))\big)-\ln \big(\det M(q)\big)\Big) = \frac{1}{2} \sqrt{2\dt} \Tr \Big( M^{-1}(q) \left[ DM(q) \cdot \big( B(q) G \big)\right]\Big) + \mathrm{O}(\dt).
\]
We finally turn to the remaining term, which, using the short-hand notation $q' = \Phi_\dt(q,G)$, we decompose as
\begin{equation}
\label{eq:decomposition_Mq_norm}
\begin{aligned}
|q-q'-\dt F(q')|^2_{M(q')} - |q'-q-\dt F(q)|^2_{M(q)} & = \Big( |q-q'-\dt F(q')|^2_{M(q)} - |q'-q-\dt F(q)|^2_{M(q)} \Big) \\
& \ \ + \Big( |q-q'-\dt F(q')|^2_{M(q')} - |q-q'-\dt F(q')|^2_{M(q)} \Big).
\end{aligned}
\end{equation}
The first term on the right-hand side of~\eqref{eq:decomposition_Mq_norm} is the difference between two vectors in the same norm, while the second term is the variation of the norm a given vector when the matrix inducing the scalar product changes. We rely on the expansion
\[
\begin{aligned}
q-q'-\dt F(q') & = q - \Phi_\dt(q,G) - \dt F(\Phi_\dt(q,G)) \\
& = -\sqrt{2\dt} B(q) G - \dt \Big( F(q)+F(\Phi_\dt(q,G)\Big) \\
& = -\sqrt{2\dt} B(q) G - 2 \dt F(q) + \mathrm{O}(\dt^{3/2}), 
\end{aligned}
\]
so that
\[
|q-q'-\dt F(q')|^2_{M(q)} - |q'-q-\dt F(q)|^2_{M(q)} = 4 \sqrt{2} \, \dt^{3/2} G^T B(q) M^{-1}(q) F(q) + \mathrm{O}(\dt^{2}).
\]
For the second term, we use $M(q+x)^{-1} = M(q)^{-1} - M^{-1}(q) \left[ DM(q)\cdot x\right] M^{-1}(q) + \mathrm{O}(|x|^2)$:
\[
\begin{aligned}
|q-q'-\dt F(q')|^2_{M(q')} & - |q-q'-\dt F(q')|^2_{M(q)} \\
& = \big(q-q'-\dt F(q')\big)^T \big(M(q')^{-1} - M(q)^{-1}\big) \big( q-q'-\dt F(q') \big) \\
& = -(2\dt)^{3/2} G^T B(q) M(q)^{-1} \Big( DM(q)\cdot B(q)G\Big)M(q)^{-1} B(q)G + \mathrm{O}(\dt).
\end{aligned}
\]
The combination of all terms gives
\[
\alpha_\dt(q,\Phi_\dt(q,G)) = \sqrt{2\dt} \, \xi_{1/2}(q,G) + \mathrm{O}(\dt),
\]
with
\[
\begin{aligned}
\xi_{1/2}(q,G) & = \nabla V(q)^T B(q) G + \frac{1}{2} \Tr \Big( M^{-1}(q) \left[ DM(q) \cdot \big( B(q) G \big)\right]\Big) + G^T B(q) M^{-1}(q) F(q) \\
& \ \ - \frac{1}{2} G^T B(q) M^{-1}(q) \Big( DM(q)\cdot B(q)G\Big) M^{-1}(q) B(q)G \\
& = (\mathrm{div}M)^T M^{-1/2} G + \frac{1}{2} \Tr \Big( M^{-1}(q) \left[ DM(q) \cdot \big( B(q) G \big)\right]\Big) \\
& \ \ - \frac12 G^T M^{-1/2}(q) \Big( DM(q)\cdot B(q)G\Big) M^{-1/2}(q) G.
\end{aligned}
\]
Therefore,
\begin{equation}
  \label{eq:rate_Metropolis_multiplicative}
  A_\dt^{\rm MH}\big(q,\Phi_\dt(q,G)\big) = 1 - \sqrt{2\dt} \min\left(0,\xi_{1/2}(q,G)\right) + \mathrm{O}(\dt),
\end{equation}
while, in view of~\eqref{eq:expansion_Barker_ratio},
\begin{equation}
  \label{eq:rate_Barker_multiplicative}
  A_\dt^{\rm Barker}\big(q,\Phi_\dt(q,G)\big) = \frac{1}{2} - \frac14 \sqrt{2\dt} \, \xi_{1/2}(q,G) + \dt \, \xi_2(q,G) + \mathrm{O}(\dt^{3/2}).
\end{equation}
Since $\xi_{1/2}(q,G)$ is odd in~$G$, its expectation with respect to~$G$ vanishes. The term $\xi_2(q,G)$ involves only even powers of~$G$.

To conclude the proof, we write
\[
\begin{aligned}
\frac{P_\dt-\Id}{\dt}\psi(q) & = \EE_G\left[A_\dt\big(q,\Phi_\dt(q,G) \, \frac{\psi(\Phi_\dt(q,G))-\psi(q)}{\dt}\right] \\ 
& = a \, \EE_G\left[\frac{\psi(\Phi_\dt(q,G))-\psi(q)}{\dt}\right] + \EE_G\left[ \big( A_\dt\big(q,\Phi_\dt(q,G)\big)-a\big) \frac{\psi(\Phi_\dt(q,G))-\psi(q)}{\dt} \right] \\
& = a \, \cL\psi(q) +\EE_G\left[ \frac{A_\dt\big(q,\Phi_\dt(q,G)\big)-a}{\sqrt{\dt}} \, \frac{\psi(\Phi_\dt(q,G))-\psi(q)}{\sqrt{\dt}} \right].
\end{aligned} 
\]
In the Metropolis case, $a=1$ and, by the symmetry property $\xi_{1/2}(q,G)G = -\xi_{1/2}(q,-G)G$ (similar to~\eqref{eq:symmetry_xi_3/2}),
\[
\begin{aligned}
\EE_G\left[ \frac{A^{\rm MH}_\dt\big(q,\Phi_\dt(q,G)\big)-a}{\sqrt{\dt}} \, \frac{\psi(\Phi_\dt(q,G))-\psi(q)}{\sqrt{\dt}} \right] & = -2 \EE_G\left[ \min\left(0,\xi_{1/2}(q,G)\right)\, G^T B(q) \nabla \psi(q)\right] + \mathrm{O}(\sqrt{\dt}) \\
& = - \EE_G\left[ \xi_{1/2}(q,G)\, G^T B(q) \nabla \psi(q)\right] + \mathrm{O}\left(\sqrt{\dt}\right);
\end{aligned}
\]
while, in the Barker case, $a=1/2$ and
\[
\EE_G\left[ \frac{A^{\rm Barker}_\dt\big(q,\Phi_\dt(q,G)\big)-a}{\sqrt{\dt}} \, \frac{\psi(\Phi_\dt(q,G))-\psi(q)}{\sqrt{\dt}} \right] = -\frac12 \EE_G\left[ \xi_{1/2}(q,G)\, G^T B(q) \nabla \psi(q)\right] + \mathrm{O}\left(\dt\right).
\]
Let us emphasize that the remainder indeed is of order~$\dt$ and not~$\sqrt{\dt}$ since $\xi_{1/2}(q,G)$ involves only odd powers of~$G$ while $\xi_2(q,G)$ and the $\dt$ term in the expansion of $\psi(\Phi_\dt(q,G))-\psi(q)$ in powers of~$\dt$ involve only even powers of~$G$. The claimed result now follows from the following lemma.

\begin{lemma}
\label{lem:avg_drit_0_mult}
For any $v \in \R^d$, it holds $\EE_G\left[ \xi_{1/2}(q,G) \, \left(G^T v\right)\right] = 0$.
\end{lemma}

\begin{proof}
  Note first that, for a given vector $b \in \R^d$ and a given tensor $\mathcal{A}$ of order~3,
  \[
  \EE_G \Big[\left(G^T b\right)\left(G^T v\right)\Big] = b^T v,
  \qquad
  \EE_G \Big[\left(\mathcal{A} : G^{\otimes 3}\right)\left(G^T v\right)\Big] = \sum_{i,j=1}^d v_i \left(\mathcal{A}_{jji} + \mathcal{A}_{jij} + \mathcal{A}_{ijj}\right).
  \]
  In view of the expression of~$\xi_{1/2}$, we introduce  
  \[
  \mathcal{A} : G^{\otimes 3} = G^T M^{-1/2}(q) \Big( DM(q)\cdot B(q)G\Big) M^{-1/2}(q) G,
  \]
  and
  \[
  b^T G  = 2(\mathrm{div}M)^T M^{-1/2} G + \Tr \Big( M^{-1}(q) \left[ DM(q) \cdot \big( B(q) G \big)\right]\Big).
  \]
  Recall that $M$, $M^{-1}$ and $B$ are symmetric. The components of~$\mathcal{A}$ and~$b$ respectively read
  \[
  \mathcal{A}_{ijk} = \sum_{r,s,t=1}^d \left[M^{-1/2}\right]_{r,i}\left[M^{-1/2}\right]_{s,j} \left(\partial_{q_t} M_{rs} \right) B_{tk},
  \]
  and
  \[
  b_i = 2 \sum_{r=1}^d \left[M^{-1/2}\right]_{i,r} \mathrm{div}(M_r) + \sum_{r,s,t=1}^d \left[M^{-1}\right]_{r,s} \left(\partial_{q_t} M_{rs} \right) B_{ti}.
  \]
  Now, in view of the equality
  \[
  \sum_{j=1}^d \left[M^{-1/2}\right]_{s,j} B_{tj} = \left[M^{-1/2} B \right]_{s,t} = \delta_{s,t}, 
  \]
  it holds
  \[
  \sum_{j=1}^d \mathcal{A}_{ijj} + \mathcal{A}_{jij} = 2\sum_{r,s}^d \left[M^{-1/2}\right]_{r,i} \left(\partial_{q_t} M_{rt} \right) = 2 \sum_{r=1}^d \left[M^{-1/2}\right]_{i,r} \mathrm{div}(M_r).
  \]
  Since
  \[
  \sum_{j=1}^d \left[M^{-1/2}\right]_{r,j}\left[M^{-1/2}\right]_{s,j} = \left[M^{-1}\right]_{r,s},
  \]
  we also have
  \[
  \sum_{j=1}^d \mathcal{A}_{jji} = \sum_{r,s,t=1}^d \left[M^{-1}\right]_{r,s} \left(\partial_{q_t} M_{rs} \right) B_{ti}.
  \]
  Therefore,
  \[
  \sum_{j=1}^d \mathcal{A}_{ijj} + \mathcal{A}_{jij} + \mathcal{A}_{jji} = b_i,
  \]
  which shows that $\EE_G \Big[\left(\mathcal{A} : G^{\otimes 3} - G^Tb \right)\left(G^T v\right)\Big] = 0$ and gives the expected result.
\end{proof}

\section*{Appendix: Computation of the refence value for the diffusion constant for one-dimensional systems}
\label{app:ref_diff}

We describe here how to analytically compute the self-diffusion coefficient~\eqref{eq:self_diff_mult_noise} for a one-dimensional system. We present the derivation for dynamics with multiplicative noise, the case of additive noise being recovered by setting $M(q) = 1$. The first task is to rewrite the integrated autocorrelation function as the linear response of a perturbation of the equilibrium dynamics, and next to obtain an analytic expression of the invariant measure of the system in order to evaluate the linear response. We refer to~\cite[Section~5]{ActaNum} for a mathematical introduction to the theory of linear reponse for the computation of transport coefficients.

Let us first make precise the nonequilibrium dynamics we consider. We perturb the force $-V'$ in the equilibrium dynamics \eqref{eq:dynamics_mult} by a constant force of magnitude $\eta \in \mathbb{R}$, as follows:
\[
dq^\eta_t = \Big( M(q^\eta_t)\left[- \beta V'(q^\eta_t) + \eta\right] + M'(q^\eta_t) \Big) dt + \sqrt{2} M^{1/2}(q^\eta_t) \, dW_t,
\]
It can be shown that this dynamics admits, for any value of $\eta \in \R$, a unique invariant measure which has a smooth density $\psi_\eta(q)$ with respect to the Lebesgue measure. Linear response results show that
\begin{equation}
\label{eq:D_LR}
  \mathcal{D} = \int_\cM M(q) \, \mu(dq) + \lim_{\eta \to 0} \frac1\eta \int_\cM F(q) \, \psi_\eta(q) \, dq,
\end{equation}
where $F$ is defined in~\eqref{eq:def_F_mult}. Now, the density $\psi_\eta$ satisfies the stationnary Fokker-Planck equation
\[
\frac{d}{dq}\left( M(q) \left[ (\beta V'-\eta)\psi_\eta + \frac{d\psi_\eta}{dq}\right] \right) = 0.
\]
The unique solution of this equation turns out to be the following periodic function:
\begin{equation}
\label{eq:psi_eta}
\psi_\eta(q) = \frac{1}{Z_\eta} \rme^{-\beta V(q)} \int_0^1 \frac{\rme^{\beta V(q+y)-\eta y}}{M(q+y)} \, dy,
\end{equation}
where $Z_\eta$ is a normalization constant ensuring that $\psi_\eta$ integrates to~1. 

The value of $\mathcal{D}$ is finally obtained by a finite difference approximation of the linear response in~\eqref{eq:D_LR}, with the value of the integral with respect to~$\psi_\eta$ computed using a double numerical quadrature based on~\eqref{eq:psi_eta}.

\subsection*{Acknowledgements}

We thank Alain Durmus and Gilles Vilmart for fruitful discussions, Markos Katsoulakis for pointing us to reference~\cite{Gidas95} on which Remark~\ref{rmk:other_rules} is based, as well as Marie Jardat for suggesting us to study the influence of the rejections on the dynamical properties of corrected discretizations of overdamped Langevin dynamics. The work of G.S. is supported by the Agence Nationale de la Recherche, under grant ANR-14-CE23-0012 (COSMOS) and by the European Research Council under the European Union's Seventh Framework Programme (FP/2007-2013) / ERC Grant Agreement number 614492. M.F. gratefully ackowledges the kind hospitality of the Hausdorff Research Institute for Mathematics. We also benefited from the scientific environment of the Laboratoire International Associ\'e between the Centre National de la Recherche Scientifique and the University of Illinois at Urbana-Champaign. Funding from NEEDS ``Milieux poreux'' and from GdR MOMAS is gratefully acknowledged.


\end{document}